\documentclass[graybox]{SNmult}

\usepackage[T1]{fontenc}
\usepackage[utf8]{inputenc}
\usepackage{type1cm}
\usepackage{makeidx}
\usepackage{graphicx}
\usepackage{multicol}
\usepackage[bottom]{footmisc}
\usepackage{newtxtext}
\usepackage[varvw]{newtxmath}
\usepackage{amsthm}
\usepackage{latexsym}
\usepackage{amsmath}
\usepackage{mathrsfs}
\usepackage{amstext}
\usepackage{enumitem}
\usepackage{mathtools}

\DeclarePairedDelimiter{\floor}{\lfloor}{\rfloor}
\mathtoolsset{showonlyrefs}
\usepackage{stmaryrd}
\usepackage{longtable}
\usepackage{booktabs}

\makeindex

\makeatletter
\newsavebox{\@brx}
\newcommand{\llangle}[1][]{\savebox{\@brx}{\(\m@th{#1\langle}\)}%
  \mathopen{\copy\@brx\kern-0.6\wd\@brx\usebox{\@brx}}}
\newcommand{\rrangle}[1][]{\savebox{\@brx}{\(\m@th{#1\rangle}\)}%
  \mathclose{\copy\@brx\kern-0.6\wd\@brx\usebox{\@brx}}}
\newcommand{\VERT}[1][]{\savebox{\@brx}{\(\m@th{#1|}\)}%
 \mathopen{\copy\@brx\kern-0.6\wd\@brx\copy\@brx\kern-0.6\wd\@brx\usebox{\@brx}}}
\newcommand{\VERTT}[1][]{\savebox{\@brx}{\(\m@th{#1|}\)}%
 \mathopen{\copy\@brx\kern-0.6\wd\@brx\copy\@brx\kern-0.6\wd\@brx\copy\@brx\kern-0.6\wd\@brx\usebox{\@brx}}} 
\newcommand{\RECT}[1][]{\savebox{\@brx}{\(\m@th{#1[\hspace{-0.3mm}]}\)}}
\makeatother

\newcommand{\bN}{\mathbb{N}}
\newcommand{\bE}{\mathbb{E}}
\newcommand{\bZ}{\mathbb{Z}}
\newcommand{\bR}{\mathbb{R}}
\newcommand{\bT}{\mathbb{T}}
\newcommand{\bM}{\mathbb{M}}

\newcommand{\cE}{\mathcal{V}_\rt}
\newcommand{\cV}{\mathcal{V}}
\newcommand{\cP}{\mathcal{P}}
\newcommand{\cX}{\mathcal{X}}
\newcommand{\cK}{\mathcal{K}}

\newcommand{\fA}{\mathbf{A}}
\newcommand{\fB}{\mathbf{B}}
\newcommand{\fI}{\mathbf{I}}
\newcommand{\fQ}{\mathbf{Q}}
\newcommand{\fP}{\mathbf{P}}
\newcommand{\fR}{\mathbf{R}}
\newcommand{\fT}{\mathbf{T}}

\newcommand{\sS}{\mathscr{S}}
\newcommand{\sC}{\mathscr{C}}
\newcommand{\sB}{\mathscr{B}}

\newcommand{\ri}{\mathrm{i}}
\newcommand{\rd}{\mathrm{d}}

\newcommand{\rb}{\mathrm{b}}
\newcommand{\ry}{\mathrm{y}}
\newcommand{\rt}{\mathrm{t}}
\newcommand{\rD}{\mathrm{D}}

\newcommand{\vI}{\mathtt{I}}
\newcommand{\vJ}{\mathtt{J}}
\newcommand{\vK}{\mathtt{K}}
\newcommand{\vL}{\mathtt{L}}
\newcommand{\vM}{\mathtt{M}}

\newcommand{\MM}{\mathfrak{m}}

\DeclareMathOperator{\supp}{supp}

\usepackage{tikz}
\colorlet{symbols}{black}
\usetikzlibrary{shapes.misc}
\usetikzlibrary{shapes.symbols}
\usetikzlibrary{shapes.geometric}
\usetikzlibrary{decorations}
\usetikzlibrary{decorations.markings}
\usetikzlibrary{decorations.pathmorphing}
\usetikzlibrary{matrix,arrows.meta}

\tikzset{
dot/.style={circle,fill=black,inner sep=0pt, minimum size=1mm},
tdot/.style={circle,fill=black,inner sep=0pt, minimum size=0.5mm},
odot/.style={circle,draw,black,inner sep=0pt,minimum size=1mm,path picture={
\fill (0,0) circle[radius=0.15mm];},fill=white,node contents={}},
Odot/.style={circle,draw,black,inner sep=0pt,minimum size=1.4mm,path picture={
\fill (0,0) circle[radius=0.2mm];},fill=white,node contents={}},
tri/.style={isosceles triangle,
    isosceles triangle apex angle=60,fill=white,rotate=90,inner sep=0pt,draw,minimum size=1.4mm},
sss/.style={diamond,inner sep=0pt,draw,minimum size=1.4mm}    
}
\tikzset{snake it/.style={decorate, decoration=snake}}

\makeatletter
\def\DefTree#1#2{\expandafter\gdef\csname tree#1\endcsname{\tikz[scale=0.2,baseline=-1mm,draw=symbols]{#2}}}
\def\<#1>{\csname tree#1\endcsname}
\makeatother

\DefTree{0}{\draw[white] (0,-0.8) -- (0,0) node[dot] {};}
\DefTree{o}{\draw[white] (0,-0.8) -- (0,0) node[Odot] {};}
\DefTree{1}{\draw (0,-.7) -- (0,.7) node[dot] {};}
\DefTree{1o}{\draw[blue,thick] (0,-.7) node[odot] -- (0,.7) node[dot] {};}
\DefTree{1b}{\draw[blue,thick] (0,-.7) -- (0,.7) node[dot] {};}
\DefTree{oo}{\draw[blue,thick] (0,-.7) node[odot] -- (0,.7) node[odot] {};}
\DefTree{2}{\draw (-0.6,.7) node[dot] {} -- (0,-.7) -- (0.6,.7) node[dot] {};}
\DefTree{2b}{\draw[blue,thick] (-0.6,.7) node[dot] {} -- (0,-.7) -- (0.6,.7) node[dot] {};}
\DefTree{2o}{\draw[blue,thick] (-0.6,.7) node[dot] {} -- (0,-.7) node[odot]-- (0.6,.7) node[dot] {};}
\DefTree{3}{\draw (0,-.7) -- (0,1) node[dot] {}; \draw (-.7,.7) node[dot] {} -- (0,-.7) -- (.7,.7) node[dot] {};}
\DefTree{3b}{\draw[blue,thick] (0,-.7) -- (0,1) node[dot] {}; \draw[blue,thick] (-.7,.7) node[dot] {} -- (0,-.7) -- (.7,.7) node[dot] {};}
\DefTree{30}{\draw (0,0) -- (0,-1); \draw (0,0) -- (0,1.2) node[dot] {}; \draw (-.7,1) node[dot] {} -- (0,0) -- (.7,1) node[dot] {};}
\DefTree{30b}{\draw[blue,thick] (0,0) -- (0,-1); \draw[blue,thick] (0,0) -- (0,1.2) node[dot] {}; \draw[blue,thick] (-.7,1) node[dot] {} -- (0,0) -- (.7,1) node[dot] {};}
\DefTree{30o}{\draw (0,0) -- (0,-1); \draw (0,0) -- (0,1.2) node[dot] {}; \draw (-.7,1) node[dot] {} -- (0,0) -- (.7,1) node[dot] {};}
\DefTree{31}{\draw (0,0) -- (0,-1) -- (1,0) node[dot] {}; \draw (0,0) -- (0,1.2) node[dot] {}; \draw (-.7,1) node[dot] {} -- (0,0) -- (.7,1) node[dot] {};}
\DefTree{32}{\draw (0,-1) -- (1,0) node[dot] {}; \draw (0,0) -- (0,-1); \draw (0,-1) -- (-1,0) node[dot] {}; \draw (0,0) -- (0,1.2) node[dot] {}; \draw (-.7,1) node[dot] {} -- (0,0) -- (.7,1) node[dot] {};}
\DefTree{20}{\draw (0,0) -- (0,-1);\draw (-.7,1) node[dot] {} -- (0,0) -- (.7,1) node[dot] {};}
\DefTree{22}{\draw (0,-1) -- (1,0) node[dot] {}; \draw (0,0.3) -- (0,-1); \draw (0,-1) -- (-1,0) node[dot] {};\draw (-.7,1) node[dot] {} -- (0,0.3) -- (.7,1) node[dot] {};}
\DefTree{32}{\draw (0,-1) -- (1,0) node[dot] {}; \draw (0,0) -- (0,-1); \draw (0,-1) -- (-1,0) node[dot] {}; \draw (0,0) -- (0,1.2) node[dot] {}; \draw (-.7,1) node[dot] {} -- (0,0) -- (.7,1) node[dot] {};}
\DefTree{321}{\draw (0,-1.1) -- (0,-.3); \draw (0,-.3) -- (.7,.4) node[dot] {}; \draw (0,.3) -- (0,-.3); \draw (0,-.3) -- (-.7,.4) node[dot] {}; \draw (0,.3) -- (0,1.3) node[dot] {}; \draw (-.6,1.1) node[dot] {} -- (0,.3) -- (.6,1.1) node[dot] {};}
\DefTree{31o}{\draw[blue,thick] (0,0) -- (0,-1) node[odot] {} -- (1,0) node[dot] {}; \draw[blue,thick] (0,0) -- (0,1.2) node[dot] {}; \draw[blue,thick] (-.7,1) node[dot] {} -- (0,0) -- (.7,1) node[dot] {};}
\DefTree{32b}{\draw[blue,thick] (0,-1) -- (1,0) node[dot] {}; \draw[blue,thick] (0,0) -- (0,-1); \draw[blue,thick] (0,-1) -- (-1,0) node[dot] {}; \draw[blue,thick] (0,0) -- (0,1.2) node[dot] {}; \draw[blue,thick] (-.7,1) node[dot] {} -- (0,0) -- (.7,1) node[dot] {}; \draw[blue,thick] (0,-1)  {};}
\DefTree{3b1}{
	\draw[blue,thick] (0,-.7) -- (0,1) node[dot] {};
	\draw[blue,thick] (-.7,.7) node[dot] {} -- (0,-.7);
	\draw[red,thick] (0,-.7) -- (.7,.7) node[dot] {};
}
\DefTree{32b1}{
	\draw[blue,thick] (0,-1) -- (1,0) node[dot] {};
	\draw[blue,thick] (0,0) -- (0,-1);
	\draw[red,thick]  (0,-1) -- (-1,0) node[dot] {};
	\draw[blue,thick] (0,0) -- (0,1.2) node[dot] {};
	\draw[blue,thick] (-.7,1) node[dot] {} -- (0,0) -- (.7,1) node[dot] {};
}

\DefTree{32b2}{
	\draw[blue,thick] (0,-1) -- (1,0) node[dot] {};
	\draw[blue,thick] (0,0) -- (0,-1);
	\draw[blue,thick] (0,-1) -- (-1,0) node[dot] {};
	\draw[red,thick]  (0,0) -- (0,1.2) node[dot] {};
	\draw[blue,thick] (-.7,1) node[dot] {} -- (0,0) -- (.7,1) node[dot] {};
}

\DefTree{32b3}{
	\draw[blue,thick] (0,-1) -- (1,0) node[dot] {};
	\draw[red,thick]  (0,0) -- (0,-1);
	\draw[blue,thick] (0,-1) -- (-1,0) node[dot] {};
	\draw[blue,thick] (0,0) -- (0,1.2) node[dot] {};
	\draw[blue,thick] (-.7,1) node[dot] {} -- (0,0) -- (.7,1) node[dot] {};
}
\DefTree{22o}{\draw[blue,thick] (0,-1) -- (1,0) node[dot] {}; \draw[blue,thick] (0,0.3) -- (0,-1); \draw[blue,thick] (0,-1) -- (-1,0) node[dot] {};\draw[blue,thick] (-.7,1) node[dot] {} -- (0,0.3) node[odot] {} -- (.7,1) node[dot] {}; \draw[blue,thick] (0,-1);}
\DefTree{32t}{\draw (0,-1) -- (1,0) node[dot] {}; \draw (0,0) -- (0,-1); \draw (0,-1) -- (-1,0) node[dot] {}; \draw (0,0) -- (0,1.2) node[dot] {}; \draw (-.7,1) node[dot] {} -- (0,0) -- (.7,1) node[dot] {}; \draw (0,-1) node[tri] {};}
\DefTree{32s}{\draw (0,-1) -- (1,0) node[dot] {}; \draw (0,0) -- (0,-1); \draw (0,-1) -- (-1,0) node[dot] {}; \draw (0,0) -- (0,1.2) node[dot] {}; \draw (-.7,1) node[dot] {} -- (0,0) -- (.7,1) node[dot] {}; \draw (0,-1) node[sss] {};}
\DefTree{321t}{\draw (0,-1.1) -- (0,-.3); \draw (0,-.3) -- (.7,.4) node[dot] {}; \draw (0,.3) -- (0,-.3); \draw (0,-.3) -- (-.7,.4) node[dot] {}; \draw (0,.3) -- (0,1.3) node[dot] {}; \draw (-.6,1.1) node[dot] {} -- (0,.3) -- (.6,1.1) node[dot] {}; \draw (0,-.3) node[tri] {};}
\DefTree{G}{\draw[white] (0,-1) -- (-1.4,0) node[tdot] {};  \draw[thick,black] (-1.4,0) -- (1.4,0) node[tdot] {};}
\DefTree{GB}{\draw[white] (0,-1) -- (-1.4,0) node[tdot,blue] {}; \draw[thick,blue] (-1.4,0) -- (1.4,0) node[tdot,blue] {};}
\DefTree{GR}{\draw[white] (0,-1) -- (-1.4,0) node[tdot,red] {}; \draw[thick,red] (-1.4,0) -- (1.4,0) node[tdot,red] {};}
\DefTree{C}{\draw[white] (0,-1) -- (-1.4,0) node[dot] {}; \draw[thick,dash pattern=on 1.5pt off 1pt] (-1.4,0) -- (1.4,0) node[dot] {};}

\DefTree{cum3b}{
	\draw[blue,thick] (0,-.7) -- (0,1) node[dot] {};
	\draw[blue,thick] (-.7,.7) node[dot] {} -- (0,-.7) -- (.7,.7) node[dot] {};
	\node[dot] at (3,-.7) {};
	\draw[thick,dash pattern=on 1.5pt off 1pt] (3,-.7) -- (.7,.7);
}

\DefTree{cumulant1}{
	\begin{scope}[scale=2]
		\draw[blue,thick] (0,-.7) -- (0,1) node[dot] {};
		\draw[blue,thick] (-.7,.7) node[dot] {} -- (0,-.7) -- (.7,.7) node[dot] {};
		\node[dot] at (1.5,-.7) {};
		\draw[thick,dash pattern=on 1.5pt off 1pt] (1.5,-.7) -- (.7,.7);
		\draw[thick,dash pattern=on 1.5pt off 1pt] (-.7,.7) -- (0,1);
	\end{scope}
}

\DefTree{cumulant11}{
	\begin{scope}[scale=2]
		\draw[blue,thick] (0,-.7) -- (0,1) node[dot] {};
		\draw[blue,thick] (-.7,.7) node[dot] {} -- (0,-.7);
		\draw[red,thick] (0,-.7) -- (.7,.7) node[dot] {};
		\node[dot] at (1.5,-.7) {};
		\draw[thick,dash pattern=on 1.5pt off 1pt] (1.5,-.7) -- (.7,.7);
		\draw[thick,dash pattern=on 1.5pt off 1pt] (-.7,.7) -- (0,1);
	\end{scope}
}

\DefTree{cumulant12}{
	\begin{scope}[scale=2]
		\draw[red,thick] (0,-.7) -- (0,1) node[dot] {};
		\draw[blue,thick] (-.7,.7) node[dot] {} -- (0,-.7);
		\draw[blue,thick] (0,-.7) -- (.7,.7) node[dot] {};
		\node[dot] at (1.5,-.7) {};
		\draw[thick,dash pattern=on 1.5pt off 1pt] (1.5,-.7) -- (.7,.7);
		\draw[thick,dash pattern=on 1.5pt off 1pt] (-.7,.7) -- (0,1);
	\end{scope}
}

\DefTree{cumulant21}{
	\begin{scope}[scale=2]
		\draw[blue,thick] (0,-.7) node[odot] -- (.7,.7) node[dot] {};
		\draw[blue,thick] (-.7,.7) node[dot] {} -- (0,-.7);
		\node[dot] at (1.5,-.7) {};
		\draw[thick,dash pattern=on 1.5pt off 1pt] (1.5,-.7) -- (.7,.7);
		\draw[thick,dash pattern=on 1.5pt off 1pt] (-.7,.7) --  (-1.5,-0.7) node[dot] {};
	\end{scope}
}

\DefTree{cumulant22}{
	\begin{scope}[scale=2]
		\draw[blue,thick] (0,-.7) -- (0,1) node[dot] {};
		\draw[blue,thick] (-.7,.7) node[dot] {} -- (0,-.7) node[odot] {};
		\draw[thick,dash pattern=on 1.5pt off 1pt] (-.7,.7) -- (0,1);
	\end{scope}
}

\makeatletter
\newcommand*\botimes{{\mathpalette\botimes@{1.5}}}
\newcommand*\botimes@[2]{\mathbin{\vcenter{\hbox{\scalebox{#2}{\hspace{0.0mm}$\m@th#1\otimes$\hspace{0.0mm}}}}}}
\newcommand*\Cdot{{\mathpalette\Cdot@{.6}}}
\newcommand*\Cdot@[2]{\mathbin{\vcenter{\hbox{\scalebox{#2}{\hspace{0.5mm}$\m@th#1\bullet$\hspace{0.5mm}}}}}}
\makeatother

\newtheorem{thm}{Theorem}
\newtheorem{lem}[thm]{Lemma}
\newtheorem{cor}[thm]{Corollary}
\newtheorem{ex}[thm]{Exercise}

\theoremstyle{remark}
\newtheorem{rem}[thm]{Remark}

\theoremstyle{definition}
\newtheorem{dfn}[thm]{Definition}

\numberwithin{thm}{section}

\usepackage{color,hyperref}
\definecolor{darkblue}{rgb}{0.0,0.0,0.5}
\hypersetup{
  linkcolor  = darkblue,
  citecolor  = darkblue,
  urlcolor   = darkblue,
  colorlinks = true,
}

\makeatletter
\AtBeginDocument{
  \hypersetup{
    pdftitle = {\@title},
    pdfauthor = {\@author}
  }
}
\makeatother

\begin{document}
\title*{Lecture notes on the flow equation approach to singular stochastic PDEs}
\titlerunning{Flow equation approach to singular stochastic PDEs}
\author{Pawe{\l} Duch}
\authorrunning{P. Duch}
\institute{Pawe{\l} Duch \at EPFL, Switzerland, \email{pawel.duch@epfl.ch}}

\maketitle

\abstract{The flow equation approach is a robust framework applicable to a broad class of singular SPDEs, including those with fractional Laplacians, throughout the entire subcritical regime. Inspired by Wilson's renormalization group, this method studies the coarse-grained process, which captures the behavior of solutions across spatial scales. The corresponding flow equation describes how the nonlinear terms in the effective dynamics evolve with the coarse-graining scale, playing a role analogous to the Polchinski equation in quantum field theory. The renormalization problem is then solved inductively by imposing appropriate boundary conditions on the flow equation.
}

\section{Singular subcritical SPDEs}

Stochastic partial differential equations (SPDEs) arise naturally in the mathematical description of systems subject to random influences, such as fluctuating interfaces, turbulent fluids, or fields in statistical mechanics and quantum field theory. They provide a framework for studying the interplay between deterministic dynamics and randomness, often serving as continuum limits of large stochastic systems or as models of fluctuating physical quantities at small scales.

While many SPDEs can be treated using classical analytic techniques, a particularly challenging class, known as singular SPDEs, has emerged as a central object of modern probability theory and mathematical physics. These equations are \emph{singular} in the sense that their solutions are too irregular for the nonlinear terms appearing in the equations to be well-defined in the classical sense. In particular, the driving noises (such as the space-time white noise) are so rough that even the notion of a product of the solution with itself may not make sense. This issue mirrors the ultraviolet divergences encountered in quantum field theory, where renormalization techniques are required to define interacting models.

A prominent motivation for studying singular SPDEs comes from stochastic quantization, a program introduced by Parisi and Wu~\cite{parisi1981} in the early 1980s to construct quantum field theories as invariant measures of stochastic dynamics. The corresponding evolution equations---such as the dynamical $\Phi^4_d$ and Sine-Gordon models---are prototypical examples of singular SPDEs. These models also serve as paradigms for understanding universality classes of stochastic systems, including the celebrated Kardar--Parisi--Zhang (KPZ) equation, which describes random interface growth.

Over the last decade, several groundbreaking theories have been developed to give rigorous meaning to singular SPDEs and to construct their solutions in a systematic way. Among these, the theory of regularity structures introduced by Hairer~\cite{hairer2014structures}, and the paracontrolled distributions approach developed by Gubinelli, Imkeller, and Perkowski~\cite{gubinelli2015}, stand as two landmark frameworks. Both methods combine analytic and probabilistic ideas to extend the scope of classical PDE theory into the singular regime.

The present chapter focuses on an approach to singular SPDEs grounded in the renormalization group (RG) method. Originating from Wilson’s seminal ideas~\cite{wilson1971}, this framework analyzes how the effective description of an SPDE transforms across scales, determining which terms require renormalization to obtain a mathematically consistent solution theory. The RG perspective on singular SPDEs was first rigorously developed by Kupiainen~\cite{kupiainen2016rg,kupiainen2017kpz}, who employed a discrete RG scheme to construct solutions to the $\Phi^4_3$ model and the KPZ equation. The flow equation formalism~\cite{Du21,Du22} provides a continuous-scale counterpart of this philosophy, yielding a unified and systematic approach to singular SPDEs throughout the full subcritical regime.

We restrict attention to equations driven by white noise (or its periodic version). Recall that the white noise $\xi$ on $\mathbb{R}^n$ is the centered Gaussian random distribution taking values in $\mathcal{S}'(\mathbb{R}^n)$ and characterized by
$$\bE(\langle\xi,f\rangle\langle\xi,g\rangle)=\int_{\bR^n} f(x)g(x)\,\rd x$$
for all $f,g\in\sS(\bR^n)$. We now list a few canonical examples of singular SPDEs that arise in this setting:
\begin{enumerate}
\item[(1)] Dynamical $\Phi^4_d$ model, also known as the parabolic stochastic quantization equation of the $\Phi^4_d$ model, with $d\in\{2,3\}$ -- a parabolic SPDE of the form
\begin{equation}
 (\partial_t+1-\Delta)\varPhi = \xi - \lambda\,\varPhi^3 + \infty\,\varPhi 
\end{equation}
posed in spacetime $\bR_+\times\bR^d$ driven by the spacetime white noise $\xi$.
 
\item[(2)] Dynamical fractional $\Phi^4_{d,\sigma}$ model with $d\in\{2,3,4\}$ and $\sigma\in(d/2,d]$ -- a parabolic non-local SPDE of the form
\begin{equation}\label{eq:intro_fractional_Phi_4}
 (\partial_t+1+(-\Delta)^{\sigma/2})\varPhi = \xi - \lambda\,\varPhi^3 + \infty\,\varPhi 
\end{equation}
posed in spacetime $\bR_+\times\bR^d$ driven by the spacetime white noise $\xi$ and involving the fractional Laplacian $(-\Delta)^{\sigma/2}$ of order $\sigma$. Note that for $\sigma=2$ the above SPDE coincides with the standard dynamical $\Phi^4_d$ model. Formally, the $\Phi^4_{d,\sigma}$ measure
$$
 \mu(\rd\phi) = \frac{1}{\mathcal{Z}} \exp\left(-\int_{\bR^d} \left(\phi\big(1+(-\Delta)^{\sigma/2}\big)\phi +\lambda\,\phi^4/2-\infty\,\phi^2\right)\rd x \right) \,\rd\phi,
$$
is invariant for the above SPDE, i.e. if $\varPhi$ is a solution of the above SPDE with the initial condition $\varPhi(0,\Cdot)=\phi$ distributed according to the above measure, then $\varPhi(t,\Cdot)$ is also distributed according to this measure for all $t\in\bR_+$.
 
\item[(3)] Elliptic stochastic quantization equation of $\Phi^4_d$ model with $d\in\{2,3\}$ -- an elliptic SPDE of the form
\begin{equation}
 (1-\Delta)\varPhi = \xi - \lambda\,\varPhi^3 + \infty\,\varPhi 
\end{equation}
posed in space $\bR^2\times\bR^d$ driven by the white noise $\xi$. Formally, the solution of the above equation restricted to the codimension-two hyperplane $\varPhi(0,0,\Cdot)$ is distributed according to the $\Phi^4_{d,\sigma=2}$ measure.

\item[(4)] Dynamical Sine-Gordon model with $\beta\in(0,\sqrt{8\pi})$ -- a parabolic SPDE of the form
\begin{equation}
 (\partial_t+1-\Delta)\varPhi = \xi - \lambda\,\infty\,\beta\,\sin(\beta\varPhi) 
\end{equation}
posed in spacetime $\bR_+\times\bR^2$ driven by the spacetime white noise $\xi$. Formally, the measure
$$
 \mu(\rd\phi) = \frac{1}{\mathcal{Z}} \exp\left(-\int_{\bR^2} \Big(\phi(1-\Delta)\phi +2\lambda\,\infty\,\cos(\beta\phi)\Big)\rd x \right) \,\rd\phi
$$
is invariant for the above SPDE.
\end{enumerate}

All of the above SPDEs are singular whenever $\lambda \neq 0$. In the linear case $\lambda = 0$, the equations can be solved explicitly, but the resulting solution is a random Schwartz distribution rather than a genuine function. This simple observation already signals the central difficulty: for $\lambda \neq 0$, we must look for solutions in a space of distributions. However, since multiplication of distributions is in general ill-defined, the nonlinear terms appearing in these equations lack an a priori meaning. 

The standard strategy to overcome this problem is to introduce an ultraviolet (UV) regularization---that is, to smooth out the noise or the Green’s function at small spatial scales. One then aims to show that, after performing suitable renormalizations of the nonlinear terms, a well-defined and nontrivial limit exists as the regularization is removed. The renormalization procedure compensates for the divergences arising at small scales by subtracting suitable local terms with diverging prefactors (formally denoted above by ``infinite'' quantities), thereby ensuring that the limiting solution remains finite and mathematically meaningful.

In this chapter we focus on a representative example, namely the elliptic SPDE
\begin{equation}\label{eq:singular_elliptic}
	(1-\Delta)^{\sigma/2}\varPhi = \xi + \lambda\,\varPhi^3 - \infty\,\varPhi,
\end{equation}
posed on $\bR^d$ with $2\pi$-periodic white noise $\xi$. We consider
$d\in\{1,\ldots,6\}$ and $\sigma\in(d/3,d/2]$, which is the singular
subcritical regime for this cubic equation. In this range the equation is
genuinely ill-defined before renormalization, but the renormalized theory is
still expected to be mathematically well posed. The concrete case
$d=5$, $\sigma=2$ may be kept in mind as a useful model example, because it is
closely analogous to the dynamical $\Phi^4_3$ model. Although~\eqref{eq:singular_elliptic} is not motivated by a direct physical
application, it provides a clean testing ground for the renormalization group
(RG) approach. The methods developed here are designed to extend, at least in
principle, to broader classes of singular SPDEs, including those introduced
above.

To give meaning to~\eqref{eq:singular_elliptic}, we first regularize the driving noise. Let $\vartheta \in C^\infty(\mathbb{R}^d)$ be an even, nonnegative function supported in the unit ball and normalized so that $\int \vartheta = 1$. For $\kappa \in (0,1]$, define the rescaled mollifier
\begin{equation}\label{eq:mollifier}
	\vartheta_\kappa(x):=[\kappa]^{-d}\vartheta(x/[\kappa]),
	\qquad
	[\kappa]:=\kappa^{1/\sigma}.
\end{equation}
The function $\vartheta_\kappa$ is supported in a ball of radius $[\kappa]$ and converges to the Dirac delta as $\kappa\searrow0$. For $\kappa\in(0,1]$ we define the regularized noise by $\xi_\kappa:=\vartheta_\kappa\ast\xi\in C^\infty(\bR^d)$, where $\ast$ denotes the convolution. We have $\xi_\kappa\in C^\infty(\bR^d)$ for all $\kappa\in(0,1]$ and $\lim_{\kappa\searrow0}\xi_\kappa=\xi\in\sS'(\bR^d)$ almost surely. We rewrite the singular SPDE~\eqref{eq:singular_elliptic} in the following regularized mild form
\begin{equation}\label{eq:intro_mild}
 \varPhi = G \ast F_\kappa[\varPhi],\qquad \kappa\in(0,1],
\end{equation}
where $G\in L^1(\bR^d)$ is the fundamental solution for the pseudo-differential operator \mbox{$(1-\Delta)^{\sigma/2}$}. The functional $F_\kappa[\varphi]$, called the force, is defined by
\begin{equation}\label{eq:force}
 F_\kappa[\varphi](x):= \xi_\kappa(x) 
 +\lambda \varphi^3(x)
 +\sum_{i=1}^{i_\sharp}\lambda^i c_\kappa^{(i)} \varphi(x),
\end{equation}
where $i_\sharp:=\floor{\sigma/(3\sigma-d)}$. The additional terms
\[
\lambda^i c_\kappa^{(i)} \varphi
\]
are called the \emph{mass counterterms}, while the coefficients
\[
c_\kappa^{(i)} \in \bR
\]
are called the \emph{mass renormalization constants}. These constants are chosen so as to cancel the divergences generated by the nonlinear interaction in the limit \(\kappa \searrow 0\). The number of counterterms required to renormalize the cubic nonlinearity increases without bound as \(\sigma \searrow d/3\). Consequently, the renormalization problem becomes progressively more intricate near the critical threshold \(\sigma = d/3\), where the equation ceases to be subcritical.

We can now formulate the main result of this chapter, which establishes that, after suitable renormalization, the family of regularized equations~\eqref{eq:intro_mild} admits a well-defined and nontrivial limit as $\kappa \searrow 0$.

\begin{thm}\label{thm:main}
	Let $d\in\{2,\ldots,6\}$ and $\sigma\in(d/3,d/2]$. There exists a choice of mass renormalization constants, and there exist random variables $\lambda_\star\in[0,1]$ and $\varPhi_0\in \sS'(\bR^d)$ such that:
	\begin{enumerate}
		\item[(0)] for all \mbox{$\lambda\in[-\lambda_\star,\lambda_\star]$} and $\kappa\in(0,1]$ there exists a solution \mbox{$\varPhi_\kappa\in C^\infty(\bR^d)$} of~\eqref{eq:intro_mild},
		\item[(1)] $\varPhi_0=\lim_{\kappa\searrow0}\varPhi_\kappa$ almost surely in $\sS'(\bR^d)$, 
		\item[(2)] $\bE(\lambda_\star^{-n})<\infty$ for every $n\in\bN_+$.
	\end{enumerate}
\end{thm}
\begin{proof}
	First note that by Theorem~\ref{thm:cumulants} and Exercise~\ref{ex:cumulants} the assumption of Lemma~\ref{lem:probabilistic_bounds} is satisfied. 
	By Lemma~\ref{lem:probabilistic_bounds} and Exercise~\ref{ex:convergence} one shows along the lines of the proof of Corollary~\ref{cor:stochastic_estimates} that the families of functionals $(\tilde F_{\kappa,\mu})_{\mu\in(0,1]}$ and $(\tilde H_{\kappa,\mu})_{\mu\in(0,1]}$ defined by~\eqref{eq:stopped_eff_force},~\eqref{eq:H_F} and~\eqref{eq:tilde_F_H_2} fulfill all the assumptions formulated in Lemma~\ref{lem:fixed_convergence}. The conclusion therefore follows directly from that lemma.
\end{proof}

The parameter \(\lambda\), which controls the strength of the nonlinearity, is required to be sufficiently small, and its maximal admissible size \(\lambda_\star\) is random and depends on the realization of the noise. The result is nevertheless highly nontrivial, since condition~(2) ensures that \(\lambda_\star\) admits inverse moments of all orders. In particular, \(\lambda_\star>0\) almost surely, so that for almost every realization of the noise there exists a nontrivial interval of admissible coupling constants around the origin.

As an aside, let us mention that for parabolic equations the flow equation approach can be used to construct solutions on a sufficiently small random time interval without any smallness assumption on the coupling constant. Let $\xi$ denote spacetime white noise on $\bR\times\bR^d$, periodized in the
spatial variables with period $2\pi$, and let $\xi_\kappa$ be its mollification at scale $\kappa\in(0,1]$. Consider the fractional parabolic SPDE
\begin{equation}\label{eq:intro_mild_parabolic}
	(\partial_t+(-\Delta)^{\sigma/2})\varPhi = \xi_\kappa
	+\lambda \varPhi^3
	+\sum_{i=1}^{i_\sharp}\lambda^i c_\kappa^{(i)} \varPhi,
	\qquad\mathrm{on}\quad \bR_+\times\bR^d,
\end{equation}
where $i_\sharp:=\floor{\sigma/(2\sigma-d)}$.

\begin{thm}\label{thm:parabolic}
	Let $d\in\{2,\ldots,4\}$, $\sigma\in(d/2,d]$ and $\lambda\in\bR$. For $\kappa\in(0,1]$ let $\varPhi_\kappa\in C(\bR_+,C^\infty(\bR^d))$ be the unique mild solution of~\eqref{eq:intro_mild_parabolic} with a suitable initial condition. There are mass renormalization constants, a stopping time $T_\star\in\bR_+$ and a random variable $\varPhi\in\sS'(\bR_+\times\bR^d)$ such that:
	\begin{enumerate}
		\item[(1)] for all deterministic $T>0$, $\lim_{\kappa\searrow0}\varPhi_\kappa=\varPhi$ almost surely in $\sS'((0,T)\times\bR^d)$ on the event $\{T_\star>T\}$,
		\item[(2)] $\bE(T_\star^{-n})<\infty$ for every $n\in\bN_+$.
	\end{enumerate}
\end{thm} 
When the coupling constant \(\lambda\) is negative, the cubic nonlinearity has a damping effect, and it is in fact possible to construct solutions globally in time. Establishing such a result requires suitable a priori estimates preventing finite-time blow-up of the solution. These arguments lie beyond the scope of the present chapter, and we refer the interested reader to~\cite{Du24b,GR23,EW24} for further details.

We conclude this introduction with a brief outline of the remainder of the
chapter, placing the material in the context of the relevant literature.

\begin{itemize}	
	\item Section~\ref{sec:effective_equation} introduces the flow equation
	approach in a systematic way. The original SPDE is reformulated as an
	effective equation involving an effective force, and we state the analytic
	conditions under which this reformulation is well posed.
	
	\item In Section~\ref{sec:effective_force}, we construct the effective force
	explicitly in terms of effective force kernels obtained from the flow
	equation. We then define the enhanced noise as the finite collection of
	these kernels and formulate the stochastic estimates required for the
	analysis.
	
	\item To control the enhanced noise, it is useful to study not only moments of
	the effective force kernels, but also their joint cumulants.
	Section~\ref{sec:cumulants_estimates} introduces joint cumulants and derives
	the flow equation governing their evolution. This equation plays a role
	analogous to the Polchinski equation~\cite{polchinski1984} in quantum field
	theory.
	
	\item In Section~\ref{sec:cumulants_uniform_bounds}, we use the cumulant flow
	equation to establish uniform bounds. The strength of the method becomes
	particularly visible here: once the counterterms are chosen according to the
	boundary conditions of the flow equation, the required bounds propagate
	automatically. This avoids the cumbersome tree-based inductive arguments
	which typically appear in traditional renormalization proofs.
	
	\item Section~\ref{sec:probabilistic} converts the cumulant estimates into
	almost sure bounds on the enhanced noise. The argument combines the
	moment--cumulant formula, deterministic smoothing estimates, and a
	two-parameter Kolmogorov-type argument.
	
	\item Section~\ref{sec:app} collects auxiliary lemmas and technical results needed to finalize the proof of the main theorem. This section can be skipped by readers primarily interested in the conceptual aspects of the flow approach.
	
	\item Section~\ref{sec:symbolic_index} provides a symbolic index collecting
	the main notation used throughout the chapter.
\end{itemize}

Material marked with the symbol $(\spadesuit)$ is of a technical nature and is
not essential for understanding the core ideas. Exercises marked with
$(\spadesuit)$ are intended to fill in gaps in proofs rather than introduce new
conceptual insights.

Let us also indicate how the approach developed here fits into the broader
theory of singular SPDEs. A general solution theory beyond the
Da~Prato--Debussche regime~\cite{daprato2003} was first developed by
Hairer~\cite{hairer2014structures} through the theory of regularity structures.
An alternative approach based on paracontrolled distributions was introduced
in~\cite{gubinelli2015}. The flow equation method discussed in these notes was
developed in~\cite{Du21,Du22} and is inspired by the renormalization group
perspective~\cite{wilson1971}, in particular by its formulation for singular
SPDEs in~\cite{kupiainen2016rg,kupiainen2017kpz}. Its main advantage is that it
applies directly throughout the subcritical regime and provides an alternative
to the regularity structures approach to renormalization developed
in~\cite{chandra2016bphz,bruned2019algebraic,bruned2021renormalising,BH23,HS23}
and~\cite{OSSW21,LOTT21,BOT25,BOS25}.

These notes are primarily based on~\cite{Du21,Du22}, while also drawing on
ideas from~\cite{GR23}, where the flow equation method was used to construct a
generalized $\Phi^4_3$ model with a fractional Laplacian in the entire
subcritical regime. See also~\cite{Du24b} for a pedagogical exposition. Recent
developments extend the method to singular SPDEs with general non-polynomial
nonlinearities and to rough differential equations~\cite{CF24a,CF24b}.
Furthermore,~\cite{BM25} shows that the renormalization scheme naturally
encoded by the flow equation formalism agrees with the BPHZ scheme of
regularity structures. For recent applications of the related Polchinski flow
equation in constructive quantum field theory, see
\cite{bauerschmidt2021,BBD23,GM24,Du24a,DFG25}.

\section{Effective equation}\label{sec:effective_equation}

The starting point of the flow equation approach is a reformulation of the original SPDE~\eqref{eq:intro_mild} in terms of an effective equation, an equivalent description that captures how the nonlinear interactions evolve under progressive coarse-graining. The key advantage of this reformulation, shared by all pathwise approaches to singular SPDEs, is that the resulting equation remains meaningful even in the limit where the regularization is removed, unlike the original ill-posed equation itself.

We begin by fixing the basic notation.

\begin{dfn}\label{dfn:basic}
	We use the notation
	\[
	\bM:=\bR^d,
	\qquad
	\bT:=(\bR/2\pi\bZ)^d.
	\]
	We identify functions on \(\bT\) with \(2\pi\)-periodic functions on \(\bM\). Thus
	\(C(\bT)\) is viewed as the subspace of \(C(\bM)\) consisting of continuous
	\(2\pi\)-periodic functions. For \(f\in L^\infty(\bM)\), we write
	\[
	\|f\|:=\|f\|_{L^\infty(\bM)}.
	\]
	Similarly, for \(p\in[1,\infty]\), we identify \(L^p(\bT)\) with the space of
	\(2\pi\)-periodic measurable functions \(f:\bM\to\bR\) such that
	\[
	\|f\|_{L^p(\bT)}
	:=
	\left(\int_{[-\pi,\pi)^d} |f(x)|^p\,\rd x\right)^{1/p}
	<\infty,
	\qquad p<\infty,
	\]
	with the usual modification
	\[
	\|f\|_{L^\infty(\bT)}
	:=
	\|f\|.
	\]
	The space \(\sS'(\bT)\) denotes the subspace of \(\sS'(\bM)\) consisting of periodic distributions.
\end{dfn}

We would like to construct a solution \(\varPhi_\kappa\) of the regularized
mild equation
\begin{equation}\label{eq:spde_F_mild}
 \varPhi_\kappa=G\ast F_\kappa[\varPhi_\kappa],
 \qquad
 F_\kappa[\varphi]:= \xi_\kappa 
 +\lambda\, \varphi^3
 +\sum_{i=1}^{i_\sharp}\lambda^i\, c_\kappa^{(i)} \varphi,
\end{equation}
and to prove that, for a suitable choice of the mass renormalization constants
\((c_\kappa^{(i)})_{i\in\{1,\ldots,i_\sharp\}}\), the limit
\[
\lim_{\kappa\searrow0}\varPhi_\kappa=:\varPhi_0
\]
exists in \(\sS'(\bM)\). Here the kernel \(G\in L^1(\bM)\) denotes the fundamental solution
of the operator \((1-\Delta)^{\sigma/2}\), and the local functional \(F_\kappa\) is called the
force.

In the regime \(\sigma\in(d/3,d/2]\) considered here, the limiting field
\(\varPhi_0\) is expected to be a distribution rather than a function. Therefore
the cubic term in \eqref{eq:spde_F_mild} has no canonical meaning as
\(\kappa\searrow0\), and the original mild equation becomes ill-defined in the
limit. The aim of this section is to replace \eqref{eq:spde_F_mild} by an
equivalent scale-dependent system with better limiting properties. More
precisely, for every fixed \(\kappa\in(0,1]\), the new system will be equivalent
to the regularized equation and, crucially, after a suitable choice of
counterterms, it will remain meaningful as \(\kappa\searrow0\).

We begin with the two basic ingredients of this reformulation. The first is a
scale decomposition of the Green function \(G\). This is a family of integrable
kernels
\[
G_\mu \in L^1(\bM),
\qquad \mu \in [0,1],
\]
such that \(G_0 = G\), \(G_1 = 0\), and the map
\[
[0,1] \ni \mu \longmapsto G_\mu \in L^1(\bM)
\]
is continuous and piecewise continuously differentiable. We denote its derivative by \(\dot G_\mu\). Given a force functional \(F_\kappa\), an \emph{effective force} is a family of functionals
\begin{equation}\label{eq:effective_force_intro}
F_{\kappa,\mu} : C(\bT) \to C(\bT),
\qquad \mu \in [0,1],
\end{equation}
such that \(F_{\kappa,0} = F_\kappa\) and, for every \(\varphi \in C(\bT)\), the map
\[
[0,1] \ni \mu \longmapsto F_{\kappa,\mu}[\varphi] \in C(\bT)
\]
is continuous and piecewise continuously differentiable. Moreover, we assume that each \(F_{\kappa,\mu}\) is of polynomial type.

\begin{dfn}\label{dfn:functional}
	A functional of polynomial type is a map $V\,:\,C(\bT)\to C(\bT)$ such that the directional derivatives of $V$ at $\varphi\in C(\bT)$ of order $k\in\bN_+$ along $\psi\in C(\bT)$, i.e.
	\begin{equation}
		\rD^k V[\varphi] \cdot \psi^{\otimes k}:= \partial_\tau^k V[\varphi + \tau \psi]\big|_{\tau=0},
	\end{equation}
	exist for all $k\in\bN_+$ and are non-zero for only finitely many $k\in\bN_+$.
\end{dfn}

Even if the original force \(F_\kappa\) is local, in the sense that \(F_\kappa[\varphi](x)\) depends only on the values of \(\varphi\) and its derivatives at the point \(x\), the corresponding effective force \(F_{\kappa,\mu}\) need not retain this property as nonlocal interactions are generated through convolution with the kernels \(G_\mu\).

In what follows, we will make a specific choice of a scale decomposition of the Green function and an effective force that is suitable for the problem at hand. For the time being, let us continue the informal discussion at a general level. Suppose that $\varPhi_\kappa$ is a solution of the original equation $\varPhi_\kappa=G\ast F_\kappa[\varPhi_\kappa]$. For $\mu\in[0,1]$ we define
\begin{equation}\label{eq:coarse_grained}
 \varPhi_{\kappa,\mu} = G_\mu\ast F_\kappa[\varPhi_\kappa].
\end{equation}
We call $(\varPhi_{\kappa,\mu})_{\mu\in(0,1]}$ the coarse-grained process. By our assumptions about $G_\Cdot$ we have $\varPhi_{\kappa,0}=\varPhi_\kappa$ and $\varPhi_{\kappa,1}=0$. For $\mu\in[0,1]$ the so-called remainder $\zeta_{\kappa,\mu}$ is defined by the equation
\begin{equation}\label{eq:def_zeta}
 F_\kappa[\varPhi_\kappa] = F_{\kappa,\mu}[\varPhi_{\kappa,\mu}] + \zeta_{\kappa,\mu}.
\end{equation}
Because $F_{\kappa,0}=F_\kappa$ and $\varPhi_{\kappa,0}=\varPhi_\kappa$ we obtain $\zeta_{\kappa,0}=0$. Since the LHS of~\eqref{eq:def_zeta} does not depend on $\mu$ it holds
\begin{equation}\label{eq:intro_partial_Phi}
 \partial_\mu\varPhi_{\kappa,\mu} = \dot G_\mu\ast (F_{\kappa,\mu}[\varPhi_{\kappa,\mu}] + \zeta_{\kappa,\mu})
\end{equation}
and
\begin{equation}\label{eq:intro_partial_zeta}
 \partial_\mu \zeta_{\kappa,\mu} = 
 -(\partial_\mu F_{\kappa,\mu})[\varPhi_{\kappa,\mu}] 
 -\rD F_{\kappa,\mu}[\varPhi_{\kappa,\mu}] \cdot \partial_\mu \varPhi_{\kappa,\mu}.
\end{equation}
Plugging~\eqref{eq:intro_partial_Phi} into~\eqref{eq:intro_partial_zeta} we obtain
\begin{equation}\label{eq:intro_partial_zeta3}
 \partial_\mu \zeta_{\kappa,\mu} =
 -(\partial_\mu F_{\kappa,\mu})[\varPhi_{\kappa,\mu}] 
 -\rD F_{\kappa,\mu}[\varPhi_{\kappa,\mu}] \cdot 
 (\dot G_\mu\ast (F_{\kappa,\mu}[\varPhi_{\kappa,\mu}] + \zeta_{\kappa,\mu})).
\end{equation}
Using the notation
\begin{equation}\label{eq:def_H}
 H_{\kappa,\mu}[\varphi]:=
 \partial_\mu F_{\kappa,\mu}[\varphi] 
 +
 \rD F_{\kappa,\mu}[\varphi]\cdot (\dot G_\mu\ast F_{\kappa,\mu}[\varphi])
\end{equation}
we rewrite~\eqref{eq:intro_partial_zeta3} in the following way
\begin{equation}\label{eq:intro_partial_zeta2}
 \partial_\mu \zeta_{\kappa,\mu} =
 -H_{\kappa,\mu}[\varPhi_{\kappa,\mu}] 
 -\rD F_{\kappa,\mu}[\varPhi_{\kappa,\mu}] \cdot 
 (\dot G_\mu\ast \zeta_{\kappa,\mu}).
\end{equation}
Summing up,~\eqref{eq:intro_partial_Phi} and~\eqref{eq:intro_partial_zeta2} together with the boundary conditions $\varPhi_{\kappa,1}=0$ and $\zeta_{\kappa,0}=0$ imply the following system of equations
\begin{equation}\label{eq:effective}
\begin{cases}\displaystyle
\varPhi_{\kappa,\mu} = -\int_\mu^1 \dot G_\eta \ast (F_{\kappa,\eta}[\varPhi_{\kappa,\eta}]+\zeta_{\kappa,\eta})\,\rd \eta
\\
\displaystyle
\zeta_{\kappa,\mu} = -\int_0^\mu (H_{\kappa,\eta}[\varPhi_{\kappa,\eta}] + \rD F_{\kappa,\eta}[\varPhi_{\kappa,\eta}] \cdot (\dot G_\eta\ast\zeta_{\kappa,\eta}))\,\rd \eta\,.
\end{cases}
\end{equation}
The system of equations for
\((\varPhi_{\kappa,\Cdot},\zeta_{\kappa,\Cdot})\) is called the
\emph{effective equation}. Formally, it is equivalent to the original mild
equation
\begin{equation}\label{eq:original_eq}
\varPhi_\kappa=G\ast F_\kappa[\varPhi_\kappa].
\end{equation}
Indeed, starting from a solution \(\varPhi_\kappa\) of the mild equation, one
constructs a pair \((\varPhi_{\kappa,\Cdot},\zeta_{\kappa,\Cdot})\) solving the
effective equation, with \(H_{\kappa,\Cdot}\) defined by~\eqref{eq:def_H}, by
the procedure described above.

Conversely, suppose that
\((\varPhi_{\kappa,\Cdot},\zeta_{\kappa,\Cdot})\) solves the effective equation
with \(H_{\kappa,\Cdot}\) given by~\eqref{eq:def_H}. Then one verifies that
\eqref{eq:def_zeta} holds, and hence that
\(\varPhi_\kappa:=\varPhi_{\kappa,0}\) solves~\eqref{eq:original_eq}.
The advantage of the effective equation is that, for a suitable choice of the
effective force \(F_{\kappa,\Cdot}\), it continues to make sense in the limit
\(\kappa\searrow0\). This is the key reason for introducing the
scale-dependent formulation.

\begin{rem}\label{rem:scale}
	The framework described above is useful only if the scale decomposition of
	\(G\) is chosen compatibly with the scaling of the equation. In practice, one
	chooses \(G_\mu\) so that, for every \(\mu\in(0,1]\), the kernel \(G_\mu\) is
	smooth, or at least has sufficiently many derivatives, and varies only on
	spatial scales of order
	\[
	[\mu]:=\mu^{1/\sigma}.
	\]
	Equivalently, \(G_\mu\) should contain only the part of \(G\) acting at spatial scales larger than
	\([\mu]\).
	With such a choice, the coarse-grained field \(\varPhi_{\kappa,\mu}\) should be
	thought of as the part of the solution \(\varPhi_\kappa\) visible at spatial
	scales larger than \([\mu]\). Informally, \(\varPhi_{\kappa,\mu}\) is obtained
	from \(\varPhi_\kappa\) by averaging over blocks of size \([\mu]\). It is
	therefore smooth for every fixed \(\mu>0\), while the singular behavior of
	\(\varPhi_\kappa\) is recovered only in the limit \(\mu\searrow0\).
	This interpretation suggests the heuristic equivalence
	\begin{equation}\label{eq:Phi_Besov_simeq}
		\|\varPhi_\kappa\|_{\sC^\alpha(\bM)}
		:=
		\sup_{\mu\in(0,1]}[\mu]^{-\alpha}\,
		\|K_\mu\ast\varPhi_\kappa\|
		\simeq
		\sup_{\mu\in(0,1]}[\mu]^{-\alpha}\,
		\|\varPhi_{\kappa,\mu}\|,
	\end{equation}
	at least for \(\alpha<\sigma-d/2\). Thus, uniform control of the
	coarse-grained fields \(\varPhi_{\kappa,\mu}\) in the right-hand side of
	\eqref{eq:Phi_Besov_simeq} should be viewed as a way of obtaining uniform
	Besov-type bounds on the original fields \(\varPhi_\kappa\). We refer the
	reader to~\cite{BCD11} for the standard definition and basic properties of
	Besov spaces.
\end{rem}

\begin{rem}
A possible choice for an effective force $F_{\kappa,\Cdot}$, made in~\cite{Du21,Du22} as well as in~\cite{kupiainen2016rg,kupiainen2017kpz}, is to define it in such a way that $H_{\kappa,\Cdot}$ given by~\eqref{eq:def_H} vanishes identically. An effective force satisfies then the so-called flow equation
\begin{equation}\label{eq:flow_eq}
 \partial_\mu F_{\kappa,\mu}[\varphi] 
 +
 \rD F_{\kappa,\mu}[\varphi]\cdot (\dot G_\mu\ast F_{\kappa,\mu}[\varphi])
 =
 0.
\end{equation}
In situations in which a small parameter is available solving the above equation is usually unproblematic. However, the solution is not a functional of polynomial type. Since the equation for the remainder $\zeta_{\kappa,\Cdot}$ in the system~\eqref{eq:effective} is linear and $H_{\kappa,\Cdot}=0$ the unique solution is given by $\zeta_{\kappa,\Cdot}=0$. Consequently, $\varPhi_{\kappa,\Cdot}$ satisfies the following effective equation
\begin{equation}
 \varPhi_{\kappa,\mu} = -\int_\mu^1 \dot G_\eta \ast F_{\kappa,\eta}[\varPhi_{\kappa,\eta}]\,\rd \eta.
\end{equation}
The advantage of the formulation involving $(\varPhi_{\kappa,\Cdot},\zeta_{\kappa,\Cdot})$, proposed in~\cite{GR23}, is that there is more flexibility in the choice of an effective force $F_{\kappa,\Cdot}$ as it has to satisfy the flow equation only up to some error term $H_{\kappa,\Cdot}$. In particular, a~suitable effective force can usually be constructed without exploiting the presence of a small parameter, even though a small parameter is typically needed anyway to solve~\eqref{eq:effective}.
\end{rem}

In order to prove well-posedness of the system of equations~\eqref{eq:effective} an effective force $F_{\kappa,\Cdot}$ has to satisfy some additional conditions, which we shall formulate below. To this end, let us first introduce some regularizing kernels and fix a convenient scale decomposition of the Green function.

\begin{dfn}\label{def:K_space}
Let $\cK\subset\sS'(\bM)$ be the space of signed measures on $\bM$ with finite total variation. We set $\|K\|_{\cK} = \int_{\bM} |K(\rd x)|$. For $x\in\bM$ we denote by $\delta_x\in\cK$ the Dirac delta at~$x$.
\end{dfn}

\begin{rem} 
It holds $\|\delta_x\|_\cK=1$ and $\|K\|_{\cK}=\|K\|_{L^1(\bM)}$ for all $K\in L^1(\bM)\subset\cK$.
\end{rem}

\begin{dfn}\label{def:K}
Let $\mu\in[0,1]$ and $[\mu]=\mu^{1/\sigma}$. The kernel $\tilde K_\mu\in\cK$ is the unique solution of $\tilde\fP_\mu \tilde K_\mu=\delta_0$, where $\tilde\fP_\mu:=(1-[\mu]^{2}\Delta)^{d+2}$. We define $K_\mu:=\tilde K_\mu\ast\tilde K_\mu\ast\tilde K_\mu\in\cK$ and $\fP_\mu:=\tilde\fP_\mu^3$.
\end{dfn}

The kernels \(K_\mu\) and \(\tilde K_\mu\) have exponential decay on the
characteristic length scale \([\mu]\). The kernels \(K_\mu\) and \(\tilde K_\mu\) are not \(C^\infty\) at the origin,
but they have the weak derivatives needed in the estimates below, and
convolution with them has a smoothing effect provided $\mu>0$. In what follows, we will often use the identity
\(\fP_\mu K_\mu=\delta_0\) to insert such a regularization where needed. For
instance, for suitable distributions \(\psi\) and \(\phi\), we may write
\[
\psi\ast\phi
=
(\fP_\mu\psi)\ast(K_\mu\ast\phi).
\]
We shall repeatedly use the basic estimates on these regularizing kernels
collected in the following exercise.

\begin{ex}\label{ex:K_mu}
Prove the following statements:
\begin{itemize}
\item[(1)] $\tilde K_0=\delta_0$ and $\tilde K_\mu\in L^1(\bM)\cap C^d_\rb(\bM)$ for $\mu\in(0,1]$. 

\item[(2)] For all $\mu\in[0,1]$ the kernel $\tilde K_\mu$ is a positive measure and $\|\tilde K_\mu\|_\cK=1$. 

\item[(3)] For all $0\leq\eta\leq\mu\leq1$ there exist $\tilde K_{\mu,\eta}\in\cK$ such that $\|\tilde K_{\mu,\eta}\|_\cK=1$ and $\tilde K_\mu=\tilde K_{\mu,\eta}\ast\tilde K_\eta$. 
\end{itemize}
Conclude that the kernels $K_\mu$, $\mu\in[0,1]$, also have the above properties. 
Hint for Item~(3): Let $\hat K_\mu\in\cK$ be the solution of \mbox{$\hat\fP_\mu \hat K_\mu=\delta_0$}, where $\hat\fP_\mu:=(1-[\mu]^2\Delta)$. It holds $\hat K_\mu = \hat K_{\mu,\eta}\ast \hat K_\eta$ for $\hat K_{\mu,\eta}=[\eta/\mu]^2\,\delta_0 + (1-[\eta/\mu]^2) \,\hat K_\mu\in\cK$.
\end{ex}

\begin{rem}\label{rem:G}
Recall that $G\in L^1(\bM)$ is the fundamental solution for the pseudo-differential operator $(1-\Delta)^{\sigma/2}$. Note that $G$ is smooth outside the origin. For every multi-index $a\in\bN_0^d$ it holds $|\partial^a G(x)|\lesssim |x|^{\sigma-d-|a|}$ uniformly for $x\in\bR^d\setminus\{0\}$. Furthermore, $\partial^a G$ is of fast decay at infinity for every $a\in\bN_0^d$.
\end{rem}

\begin{dfn}[Scale decomposition of $G$]\label{dfn:kernel_G}
Fix $\chi\in C^\infty(\bR)$ such that $\chi(r)=0$ for $|r|\leq1/4$ and $\chi(r)=1$ for $|r|\geq1/2$ and let $\chi_\mu(r) := \chi(r(1-\mu)/\mu)$ for $\mu\in(0,1]$. For $\mu\in(0,1]$ the smooth kernels $G_{\mu},\dot G_{\mu}\in C^\infty(\bM)$ are defined by
\begin{equation}
 G_\mu(x) := 
 \chi_\mu(|x|^\sigma)\,
 G(x),
 \qquad
 \dot G_\mu:=\partial_\mu G_\mu.
\end{equation}
\end{dfn}

\begin{rem}\label{rem:support_G}
	We choose to work with a scale decomposition of \(G\) in physical space rather than in Fourier space. The main advantage of this choice is the localization property of the kernels \(\dot G_\mu\). Indeed, \(\dot G_\mu\) is supported in the shell
	\[
	\bigl\{
	x\in\bM
	\,\big|\,
	\mu/4 < (1-\mu)|x|^\sigma < \mu/2
	\bigr\},
	\]
	and in particular
	\[
	\supp\,\dot G_\mu
	\subset
	\{x\in\bM \mid |x|\leq [\mu]\},
	\qquad
	\mu\in(0,1/2].
	\]
	This support property plays a crucial role in the subsequent analysis, since it allows us to avoid the use of weighted spaces altogether. The trade-off is that the last identity in~\eqref{eq:Phi_Besov_simeq} is no longer immediate. Establishing convergence in the corresponding Besov norm requires an additional argument, which is provided in Lemma~\ref{lem:fixed_convergence}. By contrast, for a Fourier-space decomposition of \(G\), such convergence properties would be essentially automatic. However, that approach would require working with weighted spaces, making several later estimates considerably more cumbersome.
\end{rem}

\begin{lem}\label{lem:kernel_G}
For all $l\in\bN_0$ it holds $\|\tilde\fP^{l}_\mu \dot G_\mu\|_\cK \lesssim 1$ uniformly in $\mu\in(0,1]$.
\end{lem}
\begin{proof}[Proof $\mathrm{(}\spadesuit\mathrm{)}$.]
First note that $\partial_\mu\chi_\mu(|x|^\sigma)$ vanishes unless $\mu/4<(1-\mu)|x|^\sigma\leq \mu/2$. Moreover, for all $a\in\bN_0^d$ we have
\begin{equation}
 |\partial^a \partial_\mu\chi_\mu(|x|^\sigma)|\lesssim |x|^{\sigma-|a|}/\mu^2
\end{equation}
uniformly in $\mu\in(0,1]$ and $x\in\bM$. Using the properties of the kernel $G$ mentioned in Remark~\ref{rem:G} we obtain $\|\partial^a \dot G_{\mu}\|_\cK \lesssim [\mu]^{-|a|}$ for all $a\in\bN_0^d$. This implies the lemma since $\tilde\fP_\mu = (1-[\mu]^2\Delta)^{d+2}$.
\end{proof}

In order to make sense of the system of equations~\eqref{eq:effective} in the limit \(\kappa \searrow 0\), we rewrite it in terms of the regularized functionals
\begin{equation}\label{eq:tilde_F_H}
	\tilde F_{\kappa,\eta}[\varphi]
	:=
	K_\eta \ast F_{\kappa,\eta}[K_\eta \ast \varphi],
	\qquad
	\tilde H_{\kappa,\eta}[\varphi]
	:=
	K_\eta \ast H_{\kappa,\eta}[K_\eta \ast \varphi].
\end{equation}
Thus, the original functionals are regularized at scale \(\mu\) both in their argument and in their output through convolution with the kernel \(K_\mu\). For \(\kappa > 0\), the system~\eqref{eq:effective} is equivalent to
\begin{equation}\label{eq:system2}
	\begin{cases}
		\displaystyle
		\tilde\varPhi_{\kappa,\mu}
		=
		-
		\int_\mu^1
		K_{\eta,\mu}\ast \tilde G_\eta \ast
		\big(
		\tilde F_{\kappa,\eta}[\tilde\varPhi_{\kappa,\eta}]
		+
		\tilde\zeta_{\kappa,\eta}
		\big)
		\,\rd \eta,
		\\[1.2ex]
		\displaystyle
		\tilde\zeta_{\kappa,\mu}
		=
		-
		\int_0^\mu
		K_{\mu,\eta}\ast
		\Big(
		\tilde H_{\kappa,\eta}[\tilde\varPhi_{\kappa,\eta}]
		+
		\rD \tilde F_{\kappa,\eta}[\tilde\varPhi_{\kappa,\eta}]
		\cdot
		(\tilde G_\eta\ast \tilde\zeta_{\kappa,\eta})
		\Big)
		\,\rd \eta,
	\end{cases}
\end{equation}
where
\begin{equation}\label{eq:tilde_G}
	\tilde G_\eta := \fP_\eta^2 \dot G_\eta,
\end{equation}
and the kernels \((K_{\mu,\eta})_{0\le \eta \le \mu \le 1}\) are defined through the relation
\begin{equation}\label{eq:inter_ref_kernel}
	K_\mu=K_{\mu,\eta}\ast K_\eta
\end{equation}
and satisfy the properties established in Exercise~\ref{ex:K_mu}. The solutions \((\tilde\varPhi_{\kappa,\Cdot},\tilde\zeta_{\kappa,\Cdot})\) are related to the solutions \((\varPhi_{\kappa,\Cdot},\zeta_{\kappa,\Cdot})\) of~\eqref{eq:effective} by
\begin{equation}\label{eq:tilde_Phi_zeta}
	\tilde\varPhi_{\kappa,\mu}
	=
	\fP_\mu \varPhi_{\kappa,\mu},
	\qquad
	\tilde\zeta_{\kappa,\mu}
	=
	K_\mu \ast \zeta_{\kappa,\mu}.
\end{equation}

\begin{rem}
	The equivalence between the systems~\eqref{eq:effective} and~\eqref{eq:system2} follows directly from the properties of the regularizing kernels \(K_\mu\). In particular, using the identity \(\fP_\mu K_\mu = \delta_0\) and integrating by parts in the scale parameter, we obtain
	\begin{equation}
		\begin{cases}
			\displaystyle
			\tilde\varPhi_{\kappa,\mu} = - \int_\mu^1 \fP_\mu \fP_\eta \dot G_\eta \ast K_\eta \ast \big(F_{\kappa,\eta}[\varPhi_{\kappa,\eta}] + \zeta_{\kappa,\eta}\big) \,\rd \eta,
			\\[3ex]
			\begin{aligned}
				\tilde\zeta_{\kappa,\mu} = - \int_0^\mu K_{\mu,\eta} \ast K_\eta \ast \Big( & H_{\kappa,\eta}[\varPhi_{\kappa,\eta}] \\
				& + \rD F_{\kappa,\eta}[\varPhi_{\kappa,\eta}] \cdot \big(K_\eta \ast \fP_\eta^2 \dot G_\eta \ast K_\eta \ast \zeta_{\kappa,\eta}\big) \Big) \,\rd \eta.
			\end{aligned}
		\end{cases}
	\end{equation}
	We then use
	\begin{equation}
		\varPhi_{\kappa,\mu}
		=
		K_\mu \ast \fP_\mu \varPhi_{\kappa,\mu}
		=
		K_\mu \ast \tilde\varPhi_{\kappa,\mu},
		\qquad
		\fP_\eta^2\dot G_\eta=\tilde G_\eta
	\end{equation}
	together with
	\[
	\fP_\mu\fP_\eta \dot G_\eta
	=
	K_{\eta,\mu}\ast \tilde G_\eta,
	\qquad
	\eta \ge \mu.
	\]
\end{rem}

\begin{rem}\label{rem:lift_bounds}
Note that $\sup_{\mu\in(0,1]}\|\tilde G_\mu\|_\cK=:C_G<\infty$ by Lemma~\ref{lem:kernel_G} and $\|K_{\mu,\eta}\|_\cK=1$ for all $0\leq\eta\leq\mu\leq1$ by Exercise~\ref{ex:K_mu}.  
\end{rem}

Thanks to the presence of the regularizing kernels in the definition of $\tilde F_{\kappa,\mu}$ and $\tilde H_{\kappa,\mu}$ we will be able to control the limit of these functionals as $\kappa\searrow0$. For the choice of an effective force, which will be specified in the next section, $F_{\kappa,\mu}$ is in some sense a small perturbation of the noise $\xi_\kappa$. Hence, $\tilde F_{\kappa,\mu}$ is in some sense a small perturbation $K_\mu\ast\xi_\kappa$. Since the bound $\|K_\mu\ast\xi_\kappa\|\lesssim [\mu]^{\alpha-\sigma}$ uniform in $\kappa,\mu\in(0,1]$ is satisfied almost surely for all $\alpha<\sigma-d/2\leq0$ we expect a bound of the form $\|\tilde F_{\kappa,\mu}[\varphi]\|\lesssim [\mu]^{\alpha-\sigma}$ uniform in $\kappa,\mu\in(0,1]$ for arbitrary fixed $\varphi\in C(\bT)$ and all $\alpha<\sigma-d/2$. Then by the first of the equations~\eqref{eq:system2} we can hope that
\begin{equation}
 \|\tilde\varPhi_{\kappa,\mu}\|\lesssim \int_\mu^1 [\eta]^{\alpha-\sigma}\,\rd\eta\lesssim [\mu]^{\alpha}
\end{equation}
uniformly in $\kappa,\mu\in(0,1]$, which, as we argued in Remark~\ref{rem:scale}, is consistent with the fact that $\|\varPhi_\kappa\|_{\sC^\alpha(\bM)}$ should be uniformly bounded in $\kappa\in(0,1]$. In order to make sense of the second of the equations~\eqref{eq:system2} we need a bound of the form $\|\tilde H_{\kappa,\mu}[\varphi]\|\lesssim [\mu]^{\beta-\sigma}$ uniform in $\kappa,\mu\in(0,1]$ for arbitrary fixed $\varphi\in C(\bT)$ and some $\beta>0$. Then
\begin{equation}
 \|\tilde\zeta_{\kappa,\mu}\|\lesssim \int_0^\mu [\eta]^{\beta-\sigma}\,\rd\eta \lesssim [\mu]^\beta.
\end{equation}
We stress that it is crucial that $\beta>0$ for the above bound to be valid.

\begin{dfn}
For $\alpha\in(-\infty,0)$, $\beta\in(0,\infty)$ and $R\in[1,\infty)$ we define $\sB_R$ to be the set of continuous maps
\begin{equation}
 (0,1]\ni\mu\mapsto (\tilde\varPhi_{\mu},\tilde\zeta_{\mu})\in C(\bT)\times C(\bT)
\end{equation}
such that 
\begin{equation}
 \|(\tilde\varPhi_\Cdot,\tilde\zeta_\Cdot)\|_{\sB_R}
 :=
 \sup_{\mu\in(0,1]}\,[\mu]^{-\alpha}\,\|\tilde\varPhi_{\mu}\|
 +
 R \sup_{\mu\in(0,1]}\,[\mu]^{-\beta}\,\|\tilde\zeta_{\mu}\|\leq R^2.
\end{equation} 
\end{dfn}

\begin{lem}\label{lem:lift}
Fix $\alpha\in(-\infty,0)$, $\beta\in(0,\infty)$ and $R\in[1,\infty)$ such that 
\begin{equation}
 R\,(|\alpha|\wedge\beta)>100\,\sigma\,(C_G\vee1),
 \qquad
 \qquad
 C_G:=\sup_{\mu\in(0,1]}\|\tilde G_\mu\|_\cK<\infty.
\end{equation}
Suppose that $$(\tilde F_\mu)_{\mu\in(0,1]},\qquad\qquad (\tilde H_\mu)_{\mu\in(0,1]}$$ are families of functionals of polynomial type depending continuously on $\mu\in(0,1]$ such that for some $m_\flat\in\bN_0$ it holds
\begin{equation}\label{eq:bound_F1}
\begin{aligned}
 [\mu]^{\sigma-\alpha}\,&\|\rD^k \tilde F_{\mu}[\varphi]\cdot \psi^{\otimes k}\,\|
 \\&\leq
 R\,(|\lambda|^{1/3}[\mu]^{-\alpha}\,\|\psi\|)^k\,(1/2+|\lambda|^{1/3}[\mu]^{-\alpha}\,\|\varphi\|)^{m_\flat},
\end{aligned}
\end{equation}
\begin{equation}\label{eq:bound_H1}
\begin{aligned}
 [\mu]^{\sigma-\beta}\,&\|\rD^k \tilde H_{\mu}[\varphi]\cdot\psi^{\otimes k}\,\|
 \\&\leq |\lambda|^{1/3}\,R^2\, (|\lambda|^{1/3}[\mu]^{-\alpha}\,\|\psi\|)^k\,(1/2+|\lambda|^{1/3}[\mu]^{-\alpha}\,\|\varphi\|)^{m_\flat},
\end{aligned}
\end{equation}
for all $k\in\{0,1,2\}$, $\mu\in(0,1]$, $\varphi,\psi\in C(\bT)$ and $\lambda\in[-1,1]$. Let $\lambda_\star:=1/(2R^2)^3$ and suppose that $\lambda\in[-\lambda_\star,\lambda_\star]$. Under the above assumptions the map $\fQ\,:\,\sB_R\to \sB_R$,
\begin{equation}\label{eq:map_Q}
\fQ\big[\tilde\varPhi_\Cdot,\tilde\zeta_\Cdot\big]
:=
\begin{pmatrix}
\mu\mapsto-\int_\mu^1 K_{\eta,\mu}\ast \tilde G_\eta \ast (\tilde F_{\eta}[\tilde\varPhi_{\eta}]+\tilde\zeta_{\eta})\,\rd \eta
\\
\mu\mapsto-\int_0^\mu K_{\mu,\eta}\ast (\tilde H_{\eta}[\tilde\varPhi_{\eta}] + \rD \tilde F_{\eta}[\tilde\varPhi_{\eta}] \cdot (\tilde G_\eta\ast \tilde\zeta_{\eta}))\,\rd \eta
\end{pmatrix}\,,
\end{equation}
is well defined and is a contraction with the Lipschitz constant less than~$1/2$.
\end{lem}

\begin{rem}
The bounds~\eqref{eq:bound_F1} and~\eqref{eq:bound_H1} stated in the above lemma say that the functionals $\tilde F_{\mu}$ and $\tilde H_{\mu}$ are compatible with the growth of the norm $\|\tilde\varPhi_\mu\|$ when $\mu$ tends to zero. The bounds also take into account the fact that there is one power of $\lambda^{1/3}$ for each factor of $\varphi$ in the expression~\eqref{eq:spde_F_mild} for the force $F_\kappa[\varphi]$.
\end{rem}

\begin{proof}[Proof sketch]
First note that for $(\tilde\varPhi_\Cdot,\tilde\zeta_\Cdot)\in\sB_R$ it holds
\begin{equation}
\begin{gathered}
 [\mu]^{\sigma-\alpha}\,\|\tilde F_{\mu}[\tilde\varPhi_{\mu}]\|
 \leq 
 R,
 \qquad
 [\mu]^{\sigma-\beta}\,\|\rD \tilde F_{\mu}[\tilde\varPhi_{\mu}]\cdot (\tilde G_\mu\ast\tilde\zeta_{\mu})\|
 \leq 
 C_G,
 \\
 [\mu]^{\sigma-\beta}\,\|\tilde H_{\mu}[\tilde\varPhi_{\mu}]\|
 \leq 1.
\end{gathered} 
\end{equation}
By Remark~\ref{rem:lift_bounds} we obtain
\begin{multline}
\|\fQ(\tilde\varPhi_\Cdot,\tilde\zeta_\Cdot)\|_{\sB_R}
 \leq
 C_G\,\sup_{\mu\in(0,1]}\,[\mu]^{-\alpha} \int_\mu^1 \|\tilde F_{\eta}[\tilde\varPhi_{\eta}]+\tilde\zeta_{\eta}\|\,\rd \eta
 \\
 +
 R\,\sup_{\mu\in(0,1]}\,[\mu]^{-\beta} \int_0^\mu \|\tilde H_{\eta}[\tilde\varPhi_{\eta}] + \rD \tilde F_{\eta}[\tilde\varPhi_{\eta}] \cdot (\tilde G_\eta\ast \tilde\zeta_{\eta})\|\,\rd \eta 
 \\
 \leq \sigma/|\alpha|~C_G\, R + \sigma/|\alpha|~C_G\,R + \sigma/\beta~R + \sigma/\beta~C_G\,R \leq R^2.
\end{multline}
By similar estimates one shows that $\fQ\,:\,\sB_R\to \sB_R$ is a contraction with the Lipschitz constant less than~$1/2$.
\end{proof}

\begin{rem}
	Note that the map \(\fQ\) is a contraction only when the coupling constant \(\lambda\) is sufficiently small. By contrast, for parabolic equations such as~\eqref{eq:intro_mild_parabolic}, the analogous map is contractive for arbitrary \(\lambda\in\bR\), provided the (random) time interval of existence is chosen sufficiently small, depending on both \(\lambda\) and the realization of the noise; see Theorem~\ref{thm:parabolic}.
\end{rem}

We close this section by explaining how the fixed point equation for the map
$\fQ$ is related to the original equation. The details are deferred to
Section~\ref{sec:app}, where we collect the auxiliary arguments needed to
pass from the scale-dependent fixed point to a solution of the original
equation and then to identify its limit.

Let us first describe the relation between the fixed point equation for
$\fQ$ and the original equation
\[
\varPhi_\kappa=G\ast F_\kappa[\varPhi_\kappa].
\]
In general, such a relation need not hold. It does hold, however, under the
following structural assumptions:
\begin{itemize}
	\item[(1)] $\tilde F_\Cdot\equiv \tilde F_{\kappa,\Cdot}$ and
	$\tilde H_\Cdot\equiv \tilde H_{\kappa,\Cdot}$ are defined by
	\eqref{eq:tilde_F_H} in terms of $F_{\kappa,\Cdot}$ and
	$H_{\kappa,\Cdot}$;
	
	\item[(2)] $F_{\kappa,\Cdot}$ is an effective force, in particular
	$F_{\kappa,0}=F_\kappa$;
	
	\item[(3)] $H_{\kappa,\Cdot}$ is defined by \eqref{eq:def_H} in terms of
	$F_{\kappa,\Cdot}$.
\end{itemize}
Let $\fQ_\kappa$ be the map defined by \eqref{eq:map_Q} in terms of
$\tilde F_{\kappa,\Cdot}$ and $\tilde H_{\kappa,\Cdot}$, and suppose that
$(\tilde\varPhi_{\kappa,\Cdot},\tilde\zeta_{\kappa,\Cdot})$ is a fixed point of
$\fQ_\kappa$. Under the additional boundedness assumptions verified in
Section~\ref{sec:app}, one can show that the limit
\[
\lim_{\mu\searrow0}
K_\mu\ast\tilde\varPhi_{\kappa,\mu}
=:\varPhi_\kappa
\]
exists in $C(\bT)$ and satisfies the original equation
\[
\varPhi_\kappa=G\ast F_\kappa[\varPhi_\kappa].
\]
Moreover, if the regularized functionals
$\tilde F_{\kappa,\Cdot}$ and $\tilde H_{\kappa,\Cdot}$ converge in the appropriate
scale-dependent sense as $\kappa\searrow0$, then the corresponding solutions
$\varPhi_\kappa$ converge in the Besov space $\sC^\alpha(\bM)$. This is the
content of Lemma~\ref{lem:fixed_convergence}.

Let us also explain why the above indirect route is necessary in the elliptic
setting. For parabolic equations, a natural strategy is to construct the
maximal classical solution of the mild formulation with regularized noise
$\xi_\kappa$ by patching together local solutions obtained by the contraction
principle on short time intervals. One can then verify that the restriction
of this maximal solution to a sufficiently short interval is a fixed point of
the analogue of the map $\fQ$ from Lemma~\ref{lem:lift}.

For the elliptic equation considered here, this direct strategy is not
available uniformly in the regularization parameter. Indeed, solving
\[
\varPhi_\kappa=G\ast F_\kappa[\varPhi_\kappa]
\]
directly by the contraction principle yields a classical solution only for
$\lambda\in[-\lambda_{\star,\kappa},\lambda_{\star,\kappa}]$, with
$\lambda_{\star,\kappa}\to0$ as $\kappa\searrow0$. We therefore proceed in the
opposite direction. We first construct, for some
$\lambda_\star>0$ independent of $\kappa\in(0,1]$, a fixed point of the
regularized scale-dependent map $\fQ_\kappa$ for all
$\lambda\in[-\lambda_\star,\lambda_\star]$. We then use the identification
results from Section~\ref{sec:app} to prove that, under the structural
assumptions above, this fixed point actually gives rise to a solution of the
original regularized equation. This avoids the degenerating direct
contraction argument and provides a formulation that remains stable as the
regularization is removed.

\section{Construction of effective force}\label{sec:effective_force}

In the previous section, we argued that, under certain assumptions, the equation 
\begin{equation}\label{eq:spde_F_mild2}
 \varPhi_\kappa=G\ast F_\kappa[\varPhi_\kappa],
 \qquad
 F_\kappa[\varphi]:= \xi_\kappa 
 +\lambda\, \varphi^3
 +\sum_{i=1}^{i_\sharp}\lambda^i\, c_\kappa^{(i)} \varphi,
\end{equation}
which we want to solve, can be formulated as a fixed point problem for the map $\fQ$ defined by~\eqref{eq:map_Q}. Recall that the map $\fQ$ involves a scale decomposition $(G_\mu)_{\mu\in[0,1]}$ of the Green function $G$ and two families of functionals $(\tilde F_\mu)_{\mu\in(0,1]}$ and $(\tilde H_\mu)_{\mu\in(0,1]}$. The scale decomposition of the Green function was fixed in Definition~\ref{dfn:kernel_G}. As we argued in the previous section, a fixed point of the map $\fQ$ corresponds to a solution of the original equation if 
$$
(\tilde F_\mu)_{\mu\in(0,1]}\equiv(\tilde F_{\kappa,\mu})_{\mu\in(0,1]},\qquad\qquad(\tilde H_\mu)_{\mu\in(0,1]}\equiv (\tilde H_{\kappa,\mu})_{\mu\in(0,1]}
$$
are defined in terms of an effective force 
$$(F_{\kappa,\mu})_{\mu\in[0,1]}$$
by the following equations
\begin{equation}\label{eq:tilde_F_H_2}
 \tilde F_{\kappa,\mu}[\varphi]:= K_\mu\ast F_{\kappa,\mu}[K_\mu\ast\varphi],
 \qquad
 \tilde H_{\kappa,\mu}[\varphi]:=
 K_\mu\ast H_{\kappa,\mu}[K_\mu\ast\varphi]
\end{equation}
and
\begin{equation}\label{eq:H_F}
 H_{\kappa,\mu}[\varphi]:=
 \partial_\mu F_{\kappa,\mu}[\varphi] 
 +
 \rD F_{\kappa,\mu}[\varphi]\cdot (\dot G_\mu\ast F_{\kappa,\mu}[\varphi]).
\end{equation}
In view of Lemma~\ref{lem:lift}, for all $\kappa\in(0,1]$ we would like to construct an effective force 
$$(F_{\kappa,\mu})_{\mu\in[0,1]}$$
such that for some random $R\in[1,\infty]$ with finite moments of all orders it holds
\begin{multline}\label{eq:bound_F}
 [\mu]^{\sigma-\alpha}\,\|\rD^k \tilde F_{\kappa,\mu}[\varphi]\cdot \psi^{\otimes k}\,\|
 \\\leq 
 R\,(|\lambda|^{1/3}[\mu]^{-\alpha}\,\|\psi\|)^k\,(1/2+|\lambda|^{1/3}[\mu]^{-\alpha}\,\|\varphi\|)^{m_\flat},
\end{multline}
and
\begin{multline}\label{eq:bound_H}
 [\mu]^{\sigma-\beta}\,\|\rD^k \tilde H_{\kappa,\mu}[\varphi]\cdot\psi^{\otimes k}\,\|
 \\
 \leq |\lambda|^{1/3}\,R^2\, (|\lambda|^{1/3}[\mu]^{-\alpha}\,\|\psi\|)^k\,(1/2+|\lambda|^{1/3}[\mu]^{-\alpha}\,\|\varphi\|)^{m_\flat},
\end{multline}
for all $k\in\{0,1,2\}$, $\kappa,\mu\in(0,1]$, $\varphi,\psi\in C(\bT)$ and $\lambda\in[-1,1]$. 

\begin{rem}
After establishing the above-mentioned result we will be able to conclude that for every $\kappa\in(0,1]$ there exists $\varPhi_\kappa$ such that $\varPhi_\kappa = G\ast F_\kappa[\varPhi_\kappa]$ and almost surely $\|\varPhi_\kappa\|_{\sC^\alpha(\bM)}\lesssim 1$ uniformly in $\kappa\in(0,1]$. In order to prove almost sure convergence of $\varPhi_\kappa$ as $\kappa\searrow0$ in the Besov space $\sC^\alpha(\bM)$ one has to show in addition that the functionals $\tilde F_{\kappa,\Cdot}$ and $\tilde H_{\kappa,\Cdot}$ converge as $\kappa\searrow0$ in the sense specified in Lemma~\ref{lem:fixed_convergence}.
\end{rem}

The guiding idea is to construct the effective force \(F_{\kappa,\mu}\) so that the associated remainder \(H_{\kappa,\mu}\) is of sufficiently high order in the coupling constant, namely
\[
H_{\kappa,\mu}=O(\lambda^{i_\flat+1})
\]
for a suitably large integer \(i_\flat\in\bN_+\). Equivalently, one approximately solves the flow equation up to an order that captures all relevant renormalization effects. The remaining error term is then expected to behave perturbatively and to satisfy improved estimates, making it plausible that the bound~\eqref{eq:bound_H} holds for some \(\beta>0\). As we will see, the most singular contribution to the effective force \(F_{\kappa,\mu}\) is the noise \(\xi_\kappa\). Since
\[
\|K_\mu\ast\xi_\kappa\|
\lesssim
[\mu]^{\alpha-\sigma},
\qquad
\alpha<\sigma-d/2\le0,
\]
one expects the bound~\eqref{eq:bound_F} to hold for all
\(
\alpha<\sigma-d/2,
\)
provided that the coefficients of the mass counterterms are chosen appropriately. 

The starting point of the construction of an effective force is the ansatz
\begin{equation}\label{eq:intro_ansatz}
 \langle F_{\kappa,\mu}[\varphi],\psi\rangle
 :=\sum_{i=0}^{i_\flat} \sum_{m=0}^{3i} \lambda^i\,\langle F^{i,m}_{\kappa,\mu},\psi\otimes\varphi^{\otimes m}\rangle,
\end{equation} 
for all $\psi,\varphi\in\sS(\bM)$. The distributions $F^{i,m}_{\kappa,\mu}\in\sS'(\bM^{1+m})$ that appear on the RHS of the above equality are called the effective force kernels. By definition $F^{i,m}_{\kappa,\mu}\in\sS'(\bM^{1+m})$ are such that the expression $\langle F^{i,m}_{\kappa,\mu},\psi\otimes\varphi_1\otimes\ldots\otimes\varphi_m\rangle$ is invariant under permutations of the test functions $\varphi_1,\ldots,\varphi_m\in\ \sS(\bM)$. The kernels $F^{i,m}_\kappa$ of the force $F_\kappa$ are defined by an equality analogous to~\eqref{eq:intro_ansatz}. Note that by~\eqref{eq:H_F} we have
\begin{equation}\label{eq:intro_ansatz_H}
 \langle H_{\kappa,\mu}[\varphi],\psi\rangle
 :=\sum_{i=0}^{2i_\flat} \sum_{m=0}^{3i} \lambda^i\,\langle H^{i,m}_{\kappa,\mu},\psi\otimes\varphi^{\otimes m}\rangle,
\end{equation}
for some $H^{i,m}_{\kappa,\mu}\in\sS'(\bM^{1+m})$. Recall that we want to construct the effective force such that $H_{\kappa,\mu}=O(\lambda^{i_\flat+1})$, which implies $H^{i,m}_{\kappa,\mu}=0$ for all $i\in\{0,\ldots,i_\flat\}$.

\begin{rem}\label{rem:force_coefficients}
Upon inspection of~\eqref{eq:spde_F_mild2}, one readily identifies the non-vanishing kernels of the force functional \(F^{i,m}_{\kappa}\), which are given by
\begin{equation}
\begin{gathered}
 F^{0,0}_\kappa(x)=\xi_\kappa(x),
 \qquad 
 F^{1,3}_\kappa(x;\rd y_1,\rd y_2,\rd y_3)=\delta_x(\rd y_1)\delta_x(\rd y_2)\delta_x(\rd y_3),
 \\
 F^{i,1}_\kappa(x;\rd y_1)=c_{\kappa}^{(i)} \delta_x(\rd y_1),
 \quad i\in\{1,\ldots,i_\sharp\},
\end{gathered}
\end{equation}
where \((c_\kappa^{(i)})_{i\in\{1,\ldots,i_\sharp\}}\) are the mass renormalization constants. It is immediate from the definitions that
\(F^{0,0}_\kappa \in \cV^0\), \(F^{1,3}_\kappa \in \cV^3\), and \(F^{i,1}_\kappa \in \cV^1\), where the spaces \(\cV^m\) are introduced in the definition below.
\end{rem}

\begin{dfn}\label{dfn:sVm}
For $m\in\bN_0$ the space $\cV^m$ consists of maps $V\,:\,\bM \times \mathrm{Borel}(\bM^m) \to \bR$ satisfying the following conditions: 
\begin{itemize}
\item[(1)] for every $A\in\mathrm{Borel}(\bM^m)$ the map $x\mapsto V(x;A+x)$ is continuous and $2\pi$ periodic, where $A+x:=\{(y_1+x,\ldots,y_m+x)\in\bM^m\,|\,(y_1,\ldots,y_m)\in A\}$ for $A\subset\bM^m$,
\item[(2)] for every $x\in\bM$ the map $A\mapsto V(x;A)$ is a measure with finite total variation, 
\item[(3)] the following norm
\begin{equation}
 \|V\|_{\cV^m}:=\sup_{x\in\bM}\int_{\bM^m} |V(x;\rd y_1\ldots\rd y_m)| 
\end{equation}
is finite. 
\end{itemize}
\end{dfn}

\begin{rem}
$(\cV^m,\|\Cdot\|_{\cV^m})$ is a Banach space. For every $V\in\cV^m$ and $A\in\mathrm{Borel}(\bM^m)$ the map $x\mapsto V(x;A)$ is measurable. We identify $V\in\cV^m$ with a distribution $V\in\sS'(\bM^{1+m})$, denoted by the same symbol, defined by the measure $V(x;\rd y_1,\ldots,\rd y_m)\,\rd x$.
\end{rem}

Since we require that \(H_{\kappa,\mu}=O(\lambda^{i_\flat+1})\), it follows from~\eqref{eq:H_F} that
\begin{equation}
 \partial_\mu F_{\kappa,\mu}[\varphi] 
 +
 \rD F_{\kappa,\mu}[\varphi]\cdot (\dot G_\mu\ast F_{\kappa,\mu}[\varphi])=H_{\kappa,\mu}[\varphi]=O(\lambda^{i_\flat+1}).
\end{equation}
Using the result of Exercise~\ref{ex:functional_D}, one readily verifies that this condition holds provided the effective force kernels $(F^{i,m}_{\kappa,\Cdot})_{i\in\{0,\ldots,i_\flat\}, m\in\bN_0}$ satisfy the flow equation
\begin{equation}\label{eq:flow_deterministic_i_m}
 \partial_\mu F_{\kappa,\mu}^{i,m}=-\sum_{j=0}^i\sum_{k=0}^m
 \,(1+k)\, \fB(\dot G_\mu,F^{j,1+k}_{\kappa,\mu},F^{i-j,m-k}_{\kappa,\mu}),
\end{equation}
where the map $\fB$ is introduced below.

\begin{dfn}\label{dfn:map_B}
Let $m\in\bN_0$, $k\in\{0,\ldots,m\}$. The map $$\fB\,:\,\sS(\bM)\times \cV^{1+k}\times\cV^{m-k}\to\cV^m$$ is defined by
\begin{multline}\label{eq:fB1_dfn}
 \fB(G,W,U)(x;\rd y_1,\ldots,\rd y_m)
 :=
 \frac{1}{m!} \sum_{\pi\in\cP_m}
 \int_{\bM^2} W(x;\rd y,\rd y_{\pi(1)},\ldots,\rd y_{\pi(k)})
 \\\times
 \,G(y-z)\, 
 U(z;\rd y_{\pi(k+1)},\ldots,\rd y_{\pi(m)})\,\rd z,
\end{multline}
where $\cP_m$ is the group of permutations of $\{1,\ldots,m\}$.
\end{dfn}
\begin{ex}\label{ex:functional_D}
	Let $V\in\cV^{1+k}$, $W\in\cV^{m-k}$ and $G\in\sS(\bM)$. Define the functionals $\tilde V,\tilde W,\tilde U$ by the relations
	\begin{equation}
		\langle \tilde V[\varphi],\psi\rangle
		:=\langle V,\psi\otimes\varphi^{\otimes (1+k)}\rangle,
		\qquad
		\langle \tilde W[\varphi],\psi\rangle
		:=\langle W,\psi\otimes\varphi^{\otimes (m-k)}\rangle
	\end{equation}
	and
	\begin{equation}
		\tilde U[\varphi]:=\rD \tilde V[\varphi] \cdot (G\ast \tilde W[\varphi]).
	\end{equation}
	Show that
	\begin{equation}
		\langle \tilde U[\varphi],\psi\rangle
		:=\langle U,\psi\otimes\varphi^{\otimes m}\rangle,
	\end{equation}
	where $U\in\cV^m$ satisfies
	\begin{equation}
		U = (1+k)\, \fB(G,V,W).
	\end{equation}
\end{ex}
\begin{ex}\label{ex:fB1_bound}
Prove that the map $\fB:\sS(\bM)\times \cV^{1+k}\times\cV^{m-k}\to\cV^m$ is well defined and
\begin{equation}\label{eq:fB1_bound}
 \|\fB(G,W,U)\|_{\cV^{m}}
 \leq
 \|G\|_\cK\,\|W\|_{\cV^{1+k}} \|U\|_{\cV^{m-k}}.
\end{equation}
\end{ex}

The basic idea behind the flow equation approach is a recursive construction of the effective force kernels $(F^{i,m}_{\kappa,\mu})_{i\in\{0,\ldots,i_\flat\}, m\in\bN_0}$ for all $\kappa\in(0,1]$, $\mu\in[0,1]$. For each $\kappa\in(0,1]$ we define the kernels in such a way that the map
\begin{equation}
 [0,1]\ni\mu\mapsto F^{i,m}_{\kappa,\mu} \in\cV^m
\end{equation}
is continuous and continuously differentiable for $\mu\in(0,1]$ using the following recursive algorithm:
\begin{enumerate}
 \item[(0)] We set $F^{0,0}_{\kappa,\mu}=\xi_\kappa$ and $F^{i,m}_{\kappa,\mu}= 0$ if $m>3i$,
 \item[(I)] Assuming that all $F^{i,m}_{\kappa,\mu}$ with $i<i_\circ$, or $i=i_\circ$ and $m>m_\circ$ were constructed we define $\dot F^{i,m}_{\kappa,\mu}$ with $i=i_\circ$ and $m=m_\circ$ to be the RHS of~\eqref{eq:flow_deterministic_i_m}.
 \item[(II)] Subsequently, $F^{i,m}_{\kappa,\mu}$ is defined by
 $
  F^{i,m}_{\kappa,\mu} = F^{i,m}_{\kappa} + \int_0^\mu \dot F^{i,m}_{\kappa,\eta}\,\rd\eta.
 $
\end{enumerate}

\begin{rem}
The recursive construction of the kernels is performed according to a fixed total ordering on the index set $(i,m)$. The ordering is lexicographic with priority given to $i$ in increasing order, and within each fixed $i$ the index $m$ is traversed in decreasing order. This ordering can be visualized as follows:
\[
\renewcommand{\arraystretch}{1.4}
\begin{array}{cc|ccccccc}
	&& &  &  & m &  &  & 
	\\
	&& 0 & 1 & 2 & 3 & 4 & 5 & 6
	\\
	\hline
	&0 
	& F^{0,0}_{\kappa,\mu}
	& 
	& 
	& 
	& 
	& 
	&
	\\
	i&1 
	& F^{1,0}_{\kappa,\mu}
	& F^{1,1}_{\kappa,\mu}
	& F^{1,2}_{\kappa,\mu}
	& F^{1,3}_{\kappa,\mu}
	& 
	& 
	& 
	\\
	&2 
	& F^{2,0}_{\kappa,\mu}
	& F^{2,1}_{\kappa,\mu}
	& F^{2,2}_{\kappa,\mu}
	& F^{2,3}_{\kappa,\mu}
	& F^{2,4}_{\kappa,\mu}
	& F^{2,5}_{\kappa,\mu}
	& F^{2,6}_{\kappa,\mu}
\end{array}
\]
Concretely, one first completes each row from right to left, and then moves to the last entry of the next row.
\end{rem}

\begin{dfn}
The finite list of the effective force kernels $(F^{i,m}_{\kappa,\Cdot})_{i\in\{0,\ldots,i_\flat\},m\in\{0,\ldots,3i\}}$ is called the {\it enhanced noise}. 
\end{dfn}

\begin{rem}
The above procedure cannot be used to construct directly $F^{i,m}_{\kappa,\mu}$ with $\kappa=0$. In fact, we expect that $F^{i,m}_{0,\mu}\notin\cV^m$. For example, $F^{0,0}_{0,\mu}=\xi\notin C(\bT)=\cV^0$. Instead, the kernels $F^{i,m}_{0,\mu}$ are defined probabilistically. The stochastic estimates for the enhanced noise are stated in Theorem~\ref{thm:stochastic_estimates} below.  
\end{rem}

\begin{rem}[$\spadesuit$]\label{rem:counterterms}
Note that the effective force kernels depend implicitly on the mass renormalization constants $(c^{(i)}_\kappa)_{i\in\{1,\ldots,i_\sharp\}}$. More specifically, the kernels $F^{i,m}_{\kappa,\mu}$ with $i\in\{1,\ldots,i_\flat\}$ and $m\in\{2,3,\ldots\}$ depend on $(c^{(j)}_\kappa)_{j\in\{1,\ldots,i_\sharp\wedge (i-1)\}}$ and the kernels $F^{i,m}_{\kappa,\mu}$ with $i\in\{1,\ldots,i_\flat\}$ and $m\in\{0,1\}$ depend on $(c^{(j)}_\kappa)_{j\in\{1,\ldots,i_\sharp\wedge i\}}$.
\end{rem}

\begin{ex}\label{ex:flow_consistent}
Check that the condition $F^{i,m}_{\kappa,\mu}= 0$ if $m>3i$ is consistent with the conditions stated in Items~(I) and~(II). Prove that $F^{i,m}_{\kappa,\mu}= 0$ for all $m>(2i+1)\wedge3i$.
\end{ex}

\begin{rem}\label{rem:explicit_coefficients}
Let us list some examples of effective force kernels. To this end, it will be convenient to use a diagrammatical notation. We view the diagrams as placeholders for certain multi-linear functionals of the noise. Since the diagrams do not play any role in the flow equation approach we refrain from defining precise rules that are used to draw them. Instead, for each diagram we provide an explicit expression that it represents. Note that the edges of the diagrams that are introduced below represent the fluctuation propagator $$(G-G_\mu)(x-y)=:\,x\,\<GB>\,y$$ and not the Green function $$G(x-y)=:\,x\,\<G>\,y$$. First note that
\begin{equation}
 F^{0,0}_{\kappa,\mu}(x)=\xi_\kappa(x) =: \<0>_\kappa(x),
\end{equation}
and
\begin{multline}
	F^{1,3}_{\kappa,\mu}(x;\rd y_1,\rd y_2,\rd y_3)
	=
	F^{1,3}_\kappa(x;\rd y_1,\rd y_2,\rd y_3)
	\\=
	\delta_x(\rd y_1)\delta_x(\rd y_2)\delta_x(\rd y_3)
	=:
	\<o>\,(x;\rd y_1,\rd y_2,\rd y_3).
\end{multline}
Let
\begin{equation}
 \<1b>_{\kappa,\mu}(x):=((G-G_\mu)\ast\<0>_\kappa)(x)\,,
 \qquad
 \<2b>_{\kappa,\mu}(x):=(\<1b>_{\kappa,\mu}(x))^2+c_{\kappa}^{(1)}/3.
\end{equation}
We have
\begin{equation}
\begin{gathered}
 F^{1,2}_{\kappa,\mu}(x;\rd y_1,\rd y_2)=3\,\<1b>_{\kappa,\mu}(x)\, \delta_x(\rd y_1)\delta_x(\rd y_2)=:3\,\<1o>_{\,\kappa,\mu}(x;\rd y_1,\rd y_2),
 \\
 F^{1,1}_{\kappa,\mu}(x;\rd y_1)=3\,\<2b>_{\kappa,\mu}(x)\, \delta_x(\rd y_1)=:3\,\<2o>_{\kappa,\mu}(x;\rd y_1),
 \\
 F^{1,0}_{\kappa,\mu}(x)=(\<1b>_{\kappa,\mu}(x))^3+c_{\kappa}^{(1)}\,\<1b>_{\kappa,\mu}(x) =: \<3b>_{\kappa,\mu}(x),
\end{gathered} 
\end{equation}
The kernel
\begin{equation}
 F^{2,5}_{\kappa,\mu}(x;\rd y_1,\ldots,\rd y_5)=:3\,\<oo>_{\,\mu}(x;\rd y_1,\ldots,\rd y_5)
\end{equation}
is defined as the symmetrization, with respect to permutations of $(y_1,\ldots,y_5)$, of the kernel
\begin{equation}
 3\,\delta_x(\rd y_1)\delta_x(\rd y_2)\,(G-G_\mu)(x-y_3)\, \delta_{y_3}(\rd y_4)\delta_{y_3}(\rd y_5)\,\rd y_3.
\end{equation}
The construction of kernels $F^{2,4}_{\kappa,\mu}$, $F^{2,3}_{\kappa,\mu}$, $F^{2,2}_{\kappa,\mu}$ is left as an exercise. Let
\begin{equation}
 \<30b>_{\!\!\kappa,\mu}(x):=
 ((G-G_\mu)\ast
 \<3b>_{\kappa,\mu})(x).
\end{equation}
The remaining second order kernels are given by
\begin{multline}
 F^{2,1}_{\kappa,\mu}(x;\rd y_1) = 
 9\,\<2b>_{\kappa,\mu}(x)\,
 (G-G_\mu)(x-y_1)\,
 \<2b>_{\kappa,\mu}(y_1)\,\rd y_1
 +
 c_{\kappa}^{(2)}\,\delta_x(\rd y_1)
 \\
 +
 6\,\<1b>_{\,\kappa,\mu}(x)\,\<30b>_{\!\!\kappa,\mu}(x)
 \,\delta_x(\rd y_1)
 =: 9\,\<22o>_{\kappa,\mu}(x;\rd y_1) + 6\,\<31o>_{\kappa,\mu}(x;\rd y_1) 
\end{multline}
and
\begin{equation}
 F^{2,0}_{\kappa,\mu}(x) = 
 3\,\<2b>_{\kappa,\mu}(x)
 \<30b>_{\!\!\kappa,\mu}(x)
 + c_{\kappa}^{(2)}\, \<1b>_{\kappa,\mu}(x)=: 3\,\<32b>_{\kappa,\mu}(x).
\end{equation}
\end{rem}

\begin{ex}
Convince yourself that the expressions given in Remark~\ref{rem:explicit_coefficients} satisfy the condition $F^{i,m}_{\kappa,0}=F^{i,m}_{\kappa}$, where $F^{i,m}_{\kappa}$ are the force kernels listed in Remark~\ref{rem:force_coefficients}. 
\end{ex}

\begin{ex}
Using the notation $\dot G_\mu(x-y):=\,x\,\<GR>\,y$ draw the diagrams representing the kernels $\partial_\mu F^{i,m}_{\kappa,\mu}$ for $i\in\{0,1,2\}$. Verify the formulas given in Remark~\ref{rem:explicit_coefficients} and write explicit expressions for the kernels $F^{2,4}_{\kappa,\mu}$, $F^{2,3}_{\kappa,\mu}$, $F^{2,2}_{\kappa,\mu}$. For simplicity, you can ignore numerical prefactors. For example, for $i=2$ and $m=0$, we obtain
\begin{equation}
\begin{aligned}
	\partial_\mu F^{2,0}_{\kappa,\mu} &= 3\partial_\mu \<32b>_{\kappa,\mu} =
	-9\,\<32b2>_{\kappa,\mu} - 6\,\<32b1>_{\kappa,\mu} -3\,\<32b3>_{\kappa,\mu}
	\\
	&
	=-9\,\fB(\dot G_\mu,\<22o>_{\kappa,\mu},\<0>_\kappa)
	-6\,\fB(\dot G_\mu,\<31o>_{\kappa,\mu},\<0>_\kappa)
	-3\,\fB(\dot G_\mu,\<2o>_{\kappa,\mu},\<3b>_{\kappa,\mu})
	\\
	&=
	-\fB(\dot G_\mu,F^{2,1}_{\kappa,\mu},F^{0,0}_{\kappa,\mu})
	-\fB(\dot G_\mu,F^{1,1}_{\kappa,\mu},F^{1,0}_{\kappa,\mu}),
\end{aligned}
\end{equation}
which is consistent with the flow equation~\eqref{eq:flow_deterministic_i_m}.
\end{ex}

\begin{dfn}\label{def:rho}
Let $\varepsilon\in[0,\infty)$ and $\alpha\equiv\alpha_\varepsilon:=\sigma-d/2-\varepsilon$, $\gamma\equiv\gamma_\varepsilon:=3\sigma-d-3\varepsilon$. For $i,m\in\bN_0$ we define
\begin{equation}
 \varrho_\varepsilon(i,m) := 
 \alpha_\varepsilon-\sigma
 - m\, \alpha_\varepsilon
 + i\, \gamma_\varepsilon \in \bR.
\end{equation}
We omit $\varepsilon$ if $\varepsilon=0$. Let $i_\flat,i_\sharp\in\bN_+$ be the smallest positive integers such that $\varrho(i_\flat+1,0)>0$, $\varrho(i_\sharp+1,1)>0$, respectively.
\end{dfn}

The quantity $\alpha_\varepsilon$ coincides with the regularity of the solution,
$\alpha_\varepsilon-\sigma$ is the regularity of the driving noise, and
$\gamma_\varepsilon$ plays the role of the gain due to the subcriticality of the
equation. The quantity \(\varrho_\varepsilon(i,m)\) should be interpreted as the \emph{order} of the kernel \(F^{i,m}\). More precisely, it plays the role of the scaling exponent governing the behavior of the effective force kernel \(F^{i,m}_{\kappa,\mu}\) as a function of the scale parameter \(\mu\). As will be shown later, these kernels satisfy estimates of the form
\[
\|F^{i,m}_{\kappa,\mu}\|_\mu
\lesssim
[\mu]^{\varrho_\varepsilon(i,m)},
\]
uniformly in \(\kappa\in(0,1]\). In particular, kernels of positive order become small at large scales, whereas kernels of non-positive order may diverge, but in a quantitatively controlled manner. See~\eqref{eq:stochastic_estimate} below for the precise statement.

\begin{rem}
	Actually, if the bounds of the above form are known for the \emph{relevant} effective force kernels, i.e. $F^{i,m}_{\kappa,\mu}$ such that $\varrho(i,m)\leq 0$, it can be easily proved deterministically for the \emph{irrelevant} kernels, i.e. $F^{i,m}_{\kappa,\mu}$ such that $\varrho(i,m)>0$ as discussed in Remark~\ref{rem:irrelevant} below. Thus only the relevant kernels have to be controlled directly, and only these kernels may require renormalization. Since the kernels $F^{i,1}_{\kappa,\mu}$ with $i\in\{i_\sharp+1,\ldots,i_\flat\}$ are irrelevant no renormalization should be necessary to bound them. This is an intuitive reason why only the counterterms $(c_\kappa^{(i)})_{i\in\{1,\ldots,i_\sharp\}}$ are included in the expression for the force~\eqref{eq:force}. See Remark~\ref{rem:symmetries} for more details.
\end{rem}

\begin{rem}
	In the case \(d=5\) and \(\sigma=2\), which closely resembles the elliptic stochastic quantization equation for the \(\Phi^4_3\) model, the renormalization problem is analogous to that of the dynamical \(\Phi^4_3\) model. The relevant effective force kernels and their corresponding orders are given by
	\[
	\renewcommand{\arraystretch}{1.3}
	\begin{array}{lll}
		F^{0,0}_{\kappa,\mu}
		&= \<0>_\kappa,
		&
		\varrho_\varepsilon(0,0)
		= -\frac52-\varepsilon,
		\\
		
		F^{1,3}_{\kappa,\mu}
		&= \<o>,
		&
		\varrho_\varepsilon(1,3)
		= -\varepsilon,
		\\
		
		F^{1,2}_{\kappa,\mu}
		&= 3\,\<1o>_{\,\kappa,\mu},
		&
		\varrho_\varepsilon(1,2)
		= -\frac12-2\varepsilon,
		\\
		
		F^{1,1}_{\kappa,\mu}
		&= 3\,\<2o>_{\kappa,\mu},
		&
		\varrho_\varepsilon(1,1)
		= -1-3\varepsilon,
		\\
		
		F^{1,0}_{\kappa,\mu}
		&= \<3b>_{\kappa,\mu},
		&
		\varrho_\varepsilon(1,0)
		= -\frac32-4\varepsilon,
		\\
		
		F^{2,1}_{\kappa,\mu}
		&= 9\,\<22o>_{\kappa,\mu}
		+ 6\,\<31o>_{\kappa,\mu},\quad\qquad
		&
		\varrho_\varepsilon(2,1)
		= -6\varepsilon,
		\\
		
		F^{2,0}_{\kappa,\mu}
		&= 3\,\<32b>_{\kappa,\mu},
		&
		\varrho_\varepsilon(2,0)
		= -\frac12-7\varepsilon.
	\end{array}
	\]
\end{rem}

\begin{rem}
Note that $\alpha_\varepsilon\leq0$ for all $\varepsilon\in[0,\infty)$. Moreover, $\gamma_\varepsilon>0$ for all $\varepsilon\in[0,\infty)$ in a sufficiently small neighborhood of $\varepsilon=0$ by the condition of subcriticality. In particular, $i_\flat,i_\sharp\in\bN_+$ are well defined and there are only finitely many $i,m\in\bN_0$ such that $m\leq3i$ and $\varrho(i,m)\leq 0$. Recall that if $m>3i$, then $F^{i,m}_{\kappa,\mu}$ vanishes identically. For arbitrary $\varepsilon\in(0,\infty)$ and $i,m\in\bN_0$ such that $m\leq 3i$ it holds $\varrho_\varepsilon(i,m)<\varrho(i,m)$.
\end{rem}

\begin{rem}[$\spadesuit$]\label{rem:rho}
We claim that there exists $\varepsilon_\diamond\in(0,\sigma)$ such that for all $\varepsilon\in(0,\varepsilon_\diamond)$ and all  $i,m,l\in\bN_0$ it holds \mbox{$\varrho_\varepsilon(i,m)+l>0$} if \mbox{$\varrho(i,m)+l>0$}. In what follows, we assume that $\varepsilon\in(0,\varepsilon_\diamond)$.
\end{rem}

\begin{dfn}
For $n\in\bN_+$ let $\cK^n\subset\sS'(\bM^n)$ be the space of signed measures on $\bM^n$ with finite total variation. We set $\|K\|_{\cK^n} = \int_{\bM^n} |K(\rd x_1\ldots\rd x_n)|$. Given $K\in\cK=\cK^1$ and $n\in\bN_+$ we set $K^{\otimes n}:=K\otimes\ldots\otimes K\in\cK^n$.
\end{dfn}

In the theorem below we state the {\it stochastic estimates} for the enhanced noise. Recall that the enhanced noise coincides with the following finite list of the effective force kernels $(F^{i,m}_{\kappa,\Cdot})_{i\in\{0,\ldots,i_\flat\},m\in\{0,\ldots,3i\}}.$ By the deterministic results established in Section~\ref{sec:effective_equation} and Corollary~\ref{cor:stochastic_estimates} these estimates imply that for every $\kappa\in(0,1]$ there exists $\varPhi_\kappa$ such that $\varPhi_\kappa = G\ast F_\kappa[\varPhi_\kappa]$ and $\bE\big(\sup_{\kappa\in(0,1]}\|\varPhi_\kappa\|^n_{\sC^\alpha(\bM)}\big)<\infty$ for all $n\in\bN_0$. For the proof of the convergence of $\varPhi_\kappa$ as $\kappa\searrow0$ see Lemma~\ref{lem:fixed_convergence}.

\begin{thm}\label{thm:stochastic_estimates}
There exists a choice of the mass renormalization constants $$(c^{(i)}_\kappa)_{i\in\{1,\ldots,i_\sharp\}}$$ in the expression~\eqref{eq:force} for the force $F_\kappa$ and a random variable $\tilde R\in[1,\infty]$ such that $\bE \tilde R^n<\infty$ for all $n\in\bN_+$ and it holds 
\begin{equation}\label{eq:stochastic_estimate}
 \|K_\mu^{\otimes(1+m)}\ast F^{i,m}_{\kappa,\mu}\|_{\cV^m} \leq \tilde R\,[\mu]^{\varrho_\varepsilon(i,m)}
\end{equation} 
for all $i\in\{0,\ldots,i_\flat\}$, $m\in\{0,\ldots,3i\}$, $\kappa,\mu\in(0,1]$.
\end{thm}
\begin{proof}
The theorem follows from the bounds for the cumulants of the effective force kernels established in Theorem~\ref{thm:cumulants} and Exercise~\ref{ex:cumulants} together with a Kolmogorov-type argument from Lemma~\ref{lem:probabilistic_bounds}. 
\end{proof}

\begin{rem}
	The exponent \(\varrho_\varepsilon(i,m)\) appearing in the bound of the
	previous theorem plays a role analogous to homogeneity in the theory of
	regularity structures. It should not, however, be identified with regularity.
	Let us illustrate this distinction on a simple example.
	
	Suppose, for concreteness, that \(d=5\) and \(\sigma=2\). Then
	\[
	\<1>_{\kappa=0}=G\ast\xi\in\sC^\alpha(\bM),
	\qquad
	\alpha=\sigma-\frac d2-\varepsilon=-\frac12-\varepsilon .
	\]
	The theorem implies, in particular, that
	\begin{equation}\label{eq:rem_tree_bound}
		\big\|K_\mu\ast\<32b>_{\kappa,\mu}\big\|
		\lesssim
		[\mu]^{-1/2-7\varepsilon}
	\end{equation}
	almost surely, uniformly in \(\kappa,\mu\in(0,1]\). The important point is that
	the tree \(\<32b>_{\kappa,\mu}\) itself depends on the scale \(\mu\). Thus
	\eqref{eq:rem_tree_bound} is not a statement about the regularity of a fixed
	distribution. Rather, it is a scale-dependent bound for a family of objects
	adapted to the flow equation.
	
	This distinction is essential. Consider instead the tree with standard,
	scale-independent edges \(x\,\<G>\,y=G(x-y)\),
	\begin{equation}
		\<32>_\kappa(x)
		:=
		\<2>_{\kappa}(x)\,\<30>_{\!\!\kappa}(x)
		+
		\frac{1}{3}c_{\kappa}^{(2)}\,\<1>_{\kappa}(x).
	\end{equation}
	One should not expect an estimate analogous to
	\eqref{eq:rem_tree_bound} for this object. Its regularity cannot be better than
	that of \(\<2>_\kappa\), which belongs to the Besov space \(\sC^{-1-\varepsilon}(\bM)\). Indeed, the
	natural bound is only
	\begin{equation}
		\big\|K_\mu\ast\<32>_{\kappa}\big\|
		\lesssim
		[\mu]^{-1-\varepsilon},
	\end{equation}
	uniformly in \(\kappa,\mu\in(0,1]\), and this is too singular for the estimates
	needed here.
	
	Different approaches overcome this obstruction in different ways. In the
	paracontrolled approach one isolates the resonant part and estimates the
	renormalized resonant product
	\begin{equation}
		\<32t>_\kappa
		:=
		\<2>_{\kappa}\odot \<30>_{\!\!\kappa}
		+
		\frac{1}{3}c_{\kappa}^{(2)}\,\<1>_{\kappa},
	\end{equation}
	where \(\odot\) denotes the resonant product. This object satisfies the improved
	bound
	\begin{equation}
		\big\|K_\mu\ast\<32t>_{\kappa}\big\|
		\lesssim
		[\mu]^{-1/2-5\varepsilon},
	\end{equation}
	uniformly in \(\kappa,\mu\in(0,1]\); see, for example,~\cite{MWX16}. In the
	regularity structure framework, one instead works with a recentered tree,
	\begin{equation}
		\<32s>_\kappa(x,y)
		:=
		\<2>_\kappa(y)\,
		\big(\<30>_{\!\!\kappa}(y)-\<30>_{\!\!\kappa}(x)\big)
		+
		\frac{1}{3}c_{\kappa}^{(2)}\,\<1>_{\kappa}(y),
	\end{equation}
	and proves the corresponding estimate
	\begin{equation}
		\sup_{x\in\bM}
		\left|
		\int K_\mu(x-y)\,\<32s>_{\kappa}(x,y)\,\rd y
		\right|
		\lesssim
		[\mu]^{-1/2-5\varepsilon},
	\end{equation}
	again uniformly in \(\kappa,\mu\in(0,1]\); see, for example~\cite{Hai15}.
	
	The flow equation approach achieves the same improvement in a different way.
	The trees representing the effective force kernels are not recentered; hence no
	structure group appears, and there is no positive renormalization. Instead, the
	improvement comes from the scale dependence of the edges. The kernels
	\[
	x\,\<GB>\,y=(G-G_\mu)(x-y)
	\]
	become small as \(\mu\searrow0\). For example,
	\[
	\|G-G_\mu\|_{L^1(\bM)}
	\leq
	\int_0^\mu \|\dot G_\eta\|_{L^1(\bM)}\,\rd\eta
	\lesssim
	[\mu]^\sigma .
	\]
	This smallness is what allows the scale-dependent tree
	\(\<32b>_{\kappa,\mu}\) to satisfy the desired bound
	\eqref{eq:rem_tree_bound}, even though the analogous fixed tree
	\(\<32>_\kappa\) has much worse regularity.
\end{rem}

\begin{lem}
	Assume that the bound~\eqref{eq:stochastic_estimate} holds true for all $F^{i,m}_{\kappa,\mu}$ with $i<i_\circ$, or $i=i_\circ$ and $m>m_\circ$. Then $$\|K_\mu^{\otimes(1+m)}\ast \partial_\mu F^{i,m}_{\kappa,\mu}\|_{\cV^m} \lesssim [\mu]^{\varrho_\varepsilon(i,m)-\sigma}$$ for $i=i_\circ$ and $m=m_\circ$, where $\dot F^{i,m}_{\kappa,\mu}$ denotes the RHS of~\eqref{eq:flow_deterministic_i_m}. 
\end{lem}
\begin{proof}
	By the flow equation~\eqref{eq:flow_deterministic_i_m}, we have
	\[
	\partial_\mu F_{\kappa,\mu}^{i,m}
	=
	-
	\sum_{j=0}^i
	\sum_{k=0}^m
	(1+k)\,
	\fB\bigl(
	\dot G_\mu,
	F^{j,1+k}_{\kappa,\mu},
	F^{i-j,m-k}_{\kappa,\mu}
	\bigr).
	\]
	It is therefore enough to estimate each term in this finite sum. Set
	\[
	\tilde G_\mu:=\fP_\mu^2\dot G_\mu,
	\qquad
	\tilde F^{i,m}_{\kappa,\mu}
	:=
	K_\mu^{\otimes(1+m)}
	\ast F^{i,m}_{\kappa,\mu}.
	\]
	Using the identity $\fP_\mu K_\mu=\delta_0$ and associativity of
	convolution, we obtain
	\begin{align}
		K_\mu^{\otimes(1+m)}
		\ast
		\fB\bigl(
		\dot G_\mu,
		F^{j,1+k}_{\kappa,\mu},
		F^{i-j,m-k}_{\kappa,\mu}
		\bigr)
		=
		\fB\bigl(
		\tilde G_\mu,
		\tilde F^{j,1+k}_{\kappa,\mu},
		\tilde F^{i-j,m-k}_{\kappa,\mu}
		\bigr).
	\end{align}
	Hence, by Exercise~\ref{ex:fB1_bound},
	\begin{align}
		&\bigl\|
		K_\mu^{\otimes(1+m)}
		\ast
		\fB\bigl(
		\dot G_\mu,
		F^{j,1+k}_{\kappa,\mu},
		F^{i-j,m-k}_{\kappa,\mu}
		\bigr)
		\bigr\|_{\cV^m}
		\nonumber \\
		&\qquad \leq
		\|\tilde G_\mu\|_{\cK}\,
		\|\tilde F^{j,1+k}_{\kappa,\mu}\|_{\cV^{1+k}}\,
		\|\tilde F^{i-j,m-k}_{\kappa,\mu}\|_{\cV^{m-k}} .
	\end{align}
	By the stochastic estimate~\eqref{eq:stochastic_estimate}, we have
	\[
	\|\tilde F^{j,1+k}_{\kappa,\mu}\|_{\cV^{1+k}}
	\lesssim
	[\mu]^{\varrho_\varepsilon(j,1+k)},
	\qquad
	\|\tilde F^{i-j,m-k}_{\kappa,\mu}\|_{\cV^{m-k}}
	\lesssim
	[\mu]^{\varrho_\varepsilon(i-j,m-k)}.
	\]
	Moreover, $\|\tilde G_\mu\|_{\cK}\lesssim1$. Therefore
	\begin{align}
		\bigl\|
		K_\mu^{\otimes(1+m)}
		\ast
		\fB\bigl(
		\dot G_\mu,
		F^{j,1+k}_{\kappa,\mu},
		F^{i-j,m-k}_{\kappa,\mu}
		\bigr)
		\bigr\|_{\cV^m}
		\lesssim
		[\mu]^{
			\varrho_\varepsilon(j,1+k)
			+
			\varrho_\varepsilon(i-j,m-k)
		}.
	\end{align}
	Using the identity
	\[
	\varrho_\varepsilon(j,1+k)
	+
	\varrho_\varepsilon(i-j,m-k)
	=
	\varrho_\varepsilon(i,m)-\sigma,
	\]
	we conclude that each summand in the flow equation satisfies
	\[
	\bigl\|
	K_\mu^{\otimes(1+m)}
	\ast
	\fB\bigl(
	\dot G_\mu,
	F^{j,1+k}_{\kappa,\mu},
	F^{i-j,m-k}_{\kappa,\mu}
	\bigr)
	\bigr\|_{\cV^m}
	\lesssim
	[\mu]^{\varrho_\varepsilon(i,m)-\sigma}.
	\]
	Summing over the finitely many indices $j$ and $k$ gives
	\[
	\bigl\|
	K_\mu^{\otimes(1+m)}
	\ast
	\partial_\mu F_{\kappa,\mu}^{i,m}
	\bigr\|_{\cV^m}
	\lesssim
	[\mu]^{\varrho_\varepsilon(i,m)-\sigma},
	\]
	which finishes the proof.
\end{proof}

\begin{rem}\label{rem:irrelevant}
	Suppose that the bound~\eqref{eq:stochastic_estimate} holds for all kernels $F^{i,m}_{\kappa,\mu}$ with either $i<i_\circ$, or $i=i_\circ$ and $m>m_\circ$. One then finds that the same bound automatically holds for the kernel $F^{i,m}_{\kappa,\mu}$ with $i=i_\circ$ and $m=m_\circ$ provided $\varrho(i,m)>0$, i.e.\ in the case where $F^{i,m}_{\kappa,\mu}$ is irrelevant. Indeed, in this situation the kernel vanishes at the initial scale, and therefore
	\begin{equation}
		F^{i,m}_{\kappa,\mu}
		=
		F^{i,m}_{\kappa}
		+
		\int_0^\mu \dot F^{i,m}_{\kappa,\eta}\,\rd\eta
		=
		\int_0^\mu \dot F^{i,m}_{\kappa,\eta}\,\rd\eta,
	\end{equation}
	where we used $F^{i,m}_{\kappa}=0$, cf.\ Remark~\ref{rem:force_coefficients}. Applying the bound from the previous lemma together with Exercise~\ref{ex:K_mu}~(3), we obtain
	\begin{equation}
		\|K_\mu^{\otimes(1+m)}\ast F^{i,m}_{\kappa,\mu}\|_{\cV^m}
		\leq
		\int_0^\mu
		\|K_\eta^{\otimes(1+m)}\ast \dot F^{i,m}_{\kappa,\eta}\|_{\cV^m}\,\rd\eta
		\lesssim
		\int_0^\mu [\eta]^{\varrho_\varepsilon(i,m)-\sigma}\,\rd\eta.
	\end{equation}
	If $\varrho_\varepsilon(i,m)>0$, then
	\begin{equation}
		\int_0^\mu [\eta]^{\varrho_\varepsilon(i,m)-\sigma}\,\rd\eta
		\lesssim
		[\mu]^{\varrho_\varepsilon(i,m)}.
	\end{equation}
	Hence, for irrelevant kernels the bound~\eqref{eq:stochastic_estimate} propagates directly by induction. This mechanism breaks down for relevant kernels, since when $\varrho_\varepsilon(i_\circ,m_\circ)<0$ the integrand $[\eta]^{\varrho_\varepsilon(i_\circ,m_\circ)-\sigma}$ fails to be integrable at $\eta=0$.
\end{rem}

\begin{cor}\label{cor:stochastic_estimates}
	There exists a deterministic $c\in(0,\infty)$ such that for all
	$\lambda\in[-1,1]$ the functionals $\tilde F_{\kappa,\Cdot}$ and
	$\tilde H_{\kappa,\Cdot}$ defined in terms of the enhanced noise
	by~\eqref{eq:intro_ansatz},~\eqref{eq:H_F} and~\eqref{eq:tilde_F_H_2}
	satisfy the assumptions of Lemma~\ref{lem:lift} with
	\[
	R=c\,\tilde R,
	\qquad
	\alpha=\sigma-d/2-\varepsilon,
	\qquad
	\beta=\varrho_\varepsilon(i_\flat+1,0),
	\qquad
	m_\flat=6i_\flat-1
	\]
	for all $\kappa\in(0,1]$, where $\tilde R\in[1,\infty]$ is the random
	variable introduced in Theorem~\ref{thm:stochastic_estimates}.
\end{cor}
\begin{proof}[Proof sketch]
We have to verify the bounds~\eqref{eq:bound_F},~\eqref{eq:bound_H}. Recall that
$$
 \tilde F_{\kappa,\mu}[\varphi]= K_\mu\ast F_{\kappa,\mu}[K_\mu\ast\varphi],\qquad
 \tilde H_{\kappa,\mu}[\varphi]=
 K_\mu\ast H_{\kappa,\mu}[K_\mu\ast\varphi]$$ 
and define 
$$\tilde F^{i,m}_{\kappa,\mu}:=K_\mu^{\otimes(1+m)}\ast F^{i,m}_{\kappa,\mu},
\qquad
\tilde H^{i,m}_{\kappa,\mu}:=K_\mu^{\otimes(1+m)}\ast H^{i,m}_{\kappa,\mu}.$$ 
Then $\tilde F^{i,m}_{\kappa,\mu}$ and $\tilde H^{i,m}_{\kappa,\mu}$ are related to $\tilde F_{\kappa,\mu}$ and $\tilde H_{\kappa,\mu}$ by formulas analogous to~\eqref{eq:intro_ansatz} and~\eqref{eq:intro_ansatz_H}. By Theorem~\ref{thm:stochastic_estimates} we have
\begin{equation}
 \|\tilde F^{i,m}_{\kappa,\mu}\|_{\cV^m} \leq \tilde R\,[\mu]^{\varrho_\varepsilon(i,m)}.
\end{equation} 
Verification of the bound~\eqref{eq:bound_F} for $\tilde F_{\kappa,\Cdot}$ is straightforward. Let us prove the bound~\eqref{eq:bound_H} for $\tilde H_{\kappa,\Cdot}$. Using~\eqref{eq:H_F},~\eqref{eq:intro_ansatz_H} as well as the fact that the effective force kernels satisfy the flow equation~\eqref{eq:flow_deterministic_i_m} we obtain
\begin{equation}
 H^{i,m}_{\kappa,\mu}=\sum_{j=i-i_\flat}^{ i_\flat}\sum_{k=0}^m
 \,(1+k)\, \fB(\dot G_\mu,F^{j,1+k}_{\kappa,\mu},F^{i-j,m-k}_{\kappa,\mu})
\end{equation}
for $i\in\{i_\flat+1,\ldots,2i_\flat\}$ and $m\in\bN_0$. Since $\fP_\mu K_\mu=\delta_0$ this implies that
\begin{equation}
 \tilde H^{i,m}_{\kappa,\mu}=\sum_{j=i-i_\flat}^{ i_\flat}\sum_{k=0}^m
 \,(1+k)\, \fB(\tilde G_\mu,\tilde F^{j,1+k}_{\kappa,\mu},\tilde F^{i-j,m-k}_{\kappa,\mu}),
\end{equation}
where $\tilde G_\mu=\fP_\mu^2\dot G_\mu$. Consequently, by the bounds for $\tilde F^{i,m}_{\kappa,\mu}$, the estimate stated in Exercise~\ref{ex:fB1_bound} and the identity
\begin{equation}
 \varrho_\varepsilon(j,1+k) + \varrho_\varepsilon(i-j,m-k)
 =
 \varrho_\varepsilon(i,m)-\sigma
\end{equation} 
we have
\begin{equation}
 \|\tilde H^{i,m}_{\kappa,\mu}\|_{\cV^m} \lesssim \tilde R^2\,[\mu]^{\varrho_\varepsilon(i,m)-\sigma}.
\end{equation} 
Since
\begin{equation}
 \langle \tilde H_{\kappa,\mu}[\varphi],\psi\rangle
 :=\sum_{i=i_\flat+1}^{2i_\flat} \sum_{m=0}^{3i} \lambda^i\,\langle \tilde H^{i,m}_{\kappa,\mu},\psi\otimes\varphi^{\otimes m}\rangle
\end{equation}
we obtain the bound~\eqref{eq:bound_H} with $\beta=\varrho_\varepsilon(i_\flat+1,0)>0$. 
\end{proof}

\begin{lem}\label{lem:support}
For all $i\in\{0,\ldots,i_\flat\}$ and $m\in\bN_0$ there exists $c\in\bR_+$ such that for all $s\in\{0,1\}$, $\kappa\in(0,1]$ and $\mu\in(0,1/2]$ it holds
\begin{equation}
 \supp\,\partial_\mu^s F^{i,m}_{\kappa,\mu}\subset\{(x,y_1,\ldots,y_m)\in\bM^{1+m}\,|\,|x-y_1|\vee\ldots\vee|x-y_m|\leq c\,[\mu]\}.
\end{equation}
\end{lem}
\begin{rem}
To prove the lemma it is enough to use the graphical representation for $F^{i,m}_{\kappa,\mu}$ introduced in Remark~\ref{rem:explicit_coefficients} and observe that for $\mu\in(0,1/2]$ the fluctuation propagator $(G-G_\mu)(x-y)$ represented by edges of the graphs vanishes identically if $|x-y|>[\mu]$ by Remark~\ref{rem:support_G}. Since we did not introduce precise rules for drawing diagrams in the exercise below we suggest proving this result using the flow equation. Note that the above support property is not true for $\mu$ close to one.
\end{rem}

\begin{ex}[$\spadesuit$]
Prove the above lemma by induction using the flow equation~\eqref{eq:flow_deterministic_i_m}. Hint: Observe that $F^{i,m}_{\kappa,0}=F^{i,m}_\kappa$ is local, and consequently it satisfies the above support property. Note that for $\mu\in(0,1/2]$ the kernel $\dot G_\mu$ is supported in a ball of radius $[\mu]$.
\end{ex}

\begin{rem}[$\spadesuit$]
Since we would like to use Lemma~\ref{lem:support} in the proof of the stochastic estimates for the enhanced noise we will actually establish the bound stated in Theorem~\ref{thm:stochastic_estimates} only for $\mu\in(0,1/2]$. The bound for $\mu\in[1/2,1]$ can then be easily proved deterministically. Alternatively, one can set
\begin{equation}\label{eq:stopped_eff_force}
 \langle F_{\kappa,\mu}[\varphi],\psi\rangle
 :=\sum_{i=0}^{i_\flat} \sum_{m=0}^{3i} \lambda^i\,\langle F^{i,m}_{\kappa,\mu\wedge1/2},\psi\otimes\varphi^{\otimes m}\rangle.
\end{equation}
Then $H_{\kappa,\mu}=O(\lambda^{i_\flat+1})$ is true for $\mu\in(0,1/2]$, which is sufficient to show the bound~\eqref{eq:bound_H} for some $\beta>0$.
\end{rem}

\section{Cumulants of effective force kernels}\label{sec:cumulants_estimates}

We now prepare the ground for the probabilistic part of the argument. In the
previous sections, the effective force kernels were constructed by means of
deterministic flow equations. The enhanced noise is the finite family
\[
(F^{i,m}_{\kappa,\Cdot})_
{i\in\{0,\ldots,i_\flat\},\,m\in\{0,\ldots,3i\}} .
\]
To prove the stochastic estimates on this family, stated in
Theorem~\ref{thm:stochastic_estimates}, we shall not estimate moments directly.
Instead, we first estimate the joint cumulants of the effective force kernels.

The reason for passing to cumulants is that they retain only the connected part
of a correlation function. In the present setting this connectedness gives an
additional spatial localization factor, which is invisible at the level of
ordinary moments. This localization will be measured in the spaces
\(\cE^\MM\), introduced below. The corresponding norm integrates over relative
positions and is adapted to translation-invariant cumulants.

The purpose of this section is therefore preparatory. We first recall the
definition of joint cumulants and fix the notation for cumulants of random
distributions. We then introduce the cumulants
\[
E^\vI_{\kappa,\mu}
\]
of the effective force kernels, together with the bookkeeping notation for
lists of indices. After that, we define the spaces \(\cE^\MM\) in which the
cumulant estimates will be formulated, and we explain heuristically why one
expects an improvement of order \([\mu]^{d(n-1)}\) for an \(n\)-point cumulant.
Finally, we derive the flow equation for cumulants. This is the key structural
input for the next section: it shows that a scale derivative of a cumulant can
be written in terms of lower-order cumulants through two deterministic
multilinear operations, denoted by \(\fA\) and \(\fB\). The remaining technical
lemmas in this section establish the basic bounds on these operations.

\begin{dfn}\label{def:cumulants}
	Let $p \in \mathbb{N}_+$ and set $I := \{1,\ldots,p\}$. Let $(\zeta_q)_{q\in I} = (\zeta_1,\ldots,\zeta_p)$ be a list of random variables with moments of all orders. The joint cumulant of $(\zeta_q)_{q\in I}$ is defined by
	\begin{equation}
		\mathbb{E}(\zeta_1,\ldots,\zeta_p)
		\equiv
		\mathbb{E}(\zeta_q)_{q\in I}
		=
		(-\ri)^p \partial_{t_1}\cdots\partial_{t_p}
		\log \mathbb{E} \exp\!\big(\ri t_1 \zeta_1 + \cdots + \ri t_p \zeta_p\big)
		\Big|_{t_1=\cdots=t_p=0}.
	\end{equation}
	In particular, $\bE(\zeta_1,\zeta_2)=\bE(\zeta_1\zeta_2)-\bE \zeta_1\,\bE\zeta_2$.
\end{dfn}

\begin{lem}\label{lem:cumulants}
	Let $p\in\bN_+$, $I=\{1,\ldots,p\}$ and $\zeta_1,\ldots,\zeta_p,\Phi,\Psi$ be random variables. It holds 
	\begin{equation}\label{eq:expectation_cumulants}
		\bE(\zeta_1\ldots\zeta_p)
		=\sum_{r=1}^p\frac{1}{r!}
		\sum_{\substack{I_1,\ldots,I_r\subset I,\\I_1\sqcup\ldots\sqcup I_r=I\\I_1,\ldots,I_r\neq \emptyset}}
		\bE(\zeta_q)_{q\in I_1}
		\ldots
		\bE(\zeta_q)_{q\in I_r},
	\end{equation} 
	\begin{equation}\label{eq:cumulants_product}
		\bE((\zeta_q)_{q\in I},\Phi\Psi)
		=
		\bE((\zeta_q)_{q\in I},\Phi,\Psi)
		+
		\sum_{\substack{I_1,I_2\subset I\\I_1\sqcup I_2= I}}
		\bE((\zeta_q)_{q\in I_1},\Phi)
		~
		\bE((\zeta_q)_{q\in I_2},\Psi).
	\end{equation}
	Here \(I_1\sqcup\cdots\sqcup I_n=I\) means that
	\(I_1,\ldots,I_n\) are pairwise disjoint subsets of \(I\) whose union is \(I\) and we used the notation for cumulants introduced in Definition~\ref{def:cumulants}.
\end{lem}
\begin{proof}
	See e.g. Proposition~3.2.1 in~\cite{peccati2011wiener}. 
\end{proof}

\begin{ex}\label{ex:gaussian_cumulant_pairings}
	Use the notation
	\begin{equation}\label{eq:covariance_edge}
		x\,\<C>\,y
		:=
		\bE\bigl(\xi_\kappa(x)\xi_\kappa(y)\bigr)
	\end{equation}
	for the covariance of the regularized noise. Consider the joint cumulant of
	\[
	X=
	\xi_\kappa(x_1)\xi_\kappa(x_2)\xi_\kappa(x_3),
	\qquad
	Y=
	\xi_\kappa(y_1)\xi_\kappa(y_2),
	\qquad
	Z=
	\xi_\kappa(z).
	\]
	By Wick's theorem, its expansion is a sum over pairings of the six noises
	located at
	\[
	x_1,x_2,x_3,y_1,y_2,z.
	\]
	Each pairing can be represented by a graph whose edges are covariance edges of
	the form~\eqref{eq:covariance_edge}. Prove that a pairing contributes to the
	cumulant only if this graph connects the three blocks
	\[
	\{x_1,x_2,x_3\},
	\qquad
	\{y_1,y_2\},
	\qquad
	\{z\}.
	\]
	Equivalently, only pairings whose induced graph on these three blocks is
	connected survive in the cumulant. Deduce that
	\begin{equation}
		\begin{aligned}
			\bE\bigl(
			\xi_\kappa(x_1)&\xi_\kappa(x_2)\xi_\kappa(x_3),
			\xi_\kappa(y_1)\xi_\kappa(y_2),
			\xi_\kappa(z)
			\bigr)
			\\
			={}&
			\sum_{\sigma\in \cP_3}
			x_{\sigma(1)}\,\<C>\,y_1
			~
			x_{\sigma(2)}\,\<C>\,y_2
			~
			x_{\sigma(3)}\,\<C>\,z
			\\
			&+
			\frac12
			\sum_{\sigma\in \cP_3}
			\sum_{\pi\in \cP_2}
			x_{\sigma(1)}\,\<C>\,x_{\sigma(2)}
			~
			x_{\sigma(3)}\,\<C>\,y_{\pi(1)}
			~
			y_{\pi(2)}\,\<C>\,z .
		\end{aligned}
	\end{equation}
\end{ex}

\begin{dfn}\label{dfn:notation_cumulants_distributions}
	Let $n \in \mathbb{N}_+$, set $I := \{1,\ldots,n\}$, and let $m_1,\ldots,m_n \in \mathbb{N}_0$. For $q \in I$, let $\zeta_q \in \sS'(\bM^{1+m_q})$ be random distributions with moments of all orders. The deterministic distribution
	\begin{equation}
		\bE(\zeta_q)_{q\in I}
		\equiv
		\bE(\zeta_1,\ldots,\zeta_n)
		\in 
		\sS'(\bM^n\times\bM^{m_1+\ldots+m_n})
	\end{equation}
	is defined by the relation
	\begin{equation}
	 \langle\bE(\zeta_1,\ldots,\zeta_n),\psi_1\otimes\ldots\otimes\psi_n\otimes\varphi_1\otimes\ldots\otimes\varphi_n\rangle
	 :=
	 \bE
	 (\langle\zeta_1,\psi_1\otimes\varphi_1\rangle,\ldots,\langle\zeta_n,\psi_n\otimes\varphi_n\rangle)
	\end{equation}
	for all $\psi_q\in \sS(\bM)$, $\varphi_q\in \sS(\bM^{m_q})$, $q\in I$.
\end{dfn}

\begin{dfn}\label{dfn:cumulants_eff_force}
	An \emph{index} is a quadruple $(i,m,s,r)$, where $i \in \{0,\ldots,i_\flat\}$, $m \in \mathbb{N}_0$, and $s,r \in \{0,1\}$. For $n \in \mathbb{N}_+$, we call
	\begin{equation}\label{eq:list_indices}
		\vI = \big((i_1,m_1,s_1,r_1),\ldots,(i_n,m_n,s_n,r_n)\big)
	\end{equation}
	a \emph{list of indices}. We define
	\[
	n(\vI) := n, 
	\qquad
	i(\vI) := i_1 + \cdots + i_n,
	\qquad
	\MM(\vI) := (m_1,\ldots,m_n),
	\]
	\[
	m(\vI) := m_1 + \cdots + m_n,
	\qquad
	s(\vI) := s_1 + \cdots + s_n,
	\qquad
	r(\vI) := r_1 + \cdots + r_n.
	\]
	The effective force kernels are multilinear functionals of the white noise and therefore admit moments of all orders. We denote their joint cumulants by
	\begin{equation}
		E^\vI_{\kappa,\mu}
		:=
		\mathbb{E}\big(
		\partial_{\mu}^{s_1}\partial_{\kappa}^{r_1}F^{i_1,m_1}_{\kappa,\mu}, \ldots,
		\partial_{\mu}^{s_n}\partial_{\kappa}^{r_n}F^{i_n,m_n}_{\kappa,\mu}
		\big)
		\in \sS'\big(\bM^{n(\vI)+m(\vI)}\big).
	\end{equation}
\end{dfn}
\begin{rem}
In order to prove the convergence of the enhanced noise as $\kappa\searrow0$ one has to study cumulants $E^\vI_{\kappa,\mu}$ with $r(\vI)\neq 0$. In what follows, for simplicity, we restrict attention to cumulants $E^\vI_{\kappa,\mu}$ with $\vI=((i_1,m_1,s_1,0),\ldots,(i_n,m_n,s_n,0))$.
\end{rem}

\begin{dfn}\label{dfn:sVMM}
Let $n\in\bN_+$, \mbox{$\MM=(m_1,\ldots,m_n)\in\bN_0^n$} and $m=m_1+\ldots+m_n$. The vector space $\cE^\MM$ consists of maps $V\in C(\bM^n\times\bM^m)$ such that
\begin{itemize} 
\item[(1)]
the function 
$$
 (x_1,\ldots,x_n)\mapsto V(x_1,\ldots,x_n;\mathrm y_1+x_1,\ldots,\mathrm y_n+x_n)
$$
is $2\pi$ periodic in all variables for every
$$
 (y_1,\ldots,y_m)=(\mathrm y_1,\ldots,\mathrm y_n)\in\bM^{m_1}\times\ldots\times\bM^{m_n}=\bM^m,
$$
where 
$$
\ry+x:=(y_1+x,\ldots,y_m+x)\in\bM^m
$$
for arbitrary $m\in\bN_0$, $x\in\bM$, $\ry=(y_1,\ldots,y_m)\in\bM^m$,
\item[(2)] it holds 
$$
V(x_1,\ldots,x_n;y_1,\ldots,y_m)=V(x_1+z,\ldots,x_n+z;y_1+z,\ldots,y_m+z)
$$
for all $x_1,\ldots,x_n,y_1,\ldots,y_m,z\in\bM$, 
\item[(3)] the norm
 \begin{equation}
 \|V\|_{\cE^\MM}
 :=
 \sup_{x_1\in\bM} \int_{\bT^{n-1}\times\bM^m}
 |V(x_1,\ldots,x_n;y_1,\ldots,y_m)|\,\rd x_2\ldots\rd x_n\,\rd y_1\ldots\rd y_m
\end{equation}
is finite.
\end{itemize}
\end{dfn}

\begin{rem}\label{rem:norms}
Let $m\in\bN_0$. Note that for $n=1$, $\MM=(m)\in\bN^n$ and $V\in\cE^\MM$ it holds $\|V\|_{\cE^\MM}=\|V\|_{\cV^m}$.
\end{rem}

\begin{rem}\label{rem:translations}
Using the condition of translational invariance stated in Item~(2) of the above definition, the expression for the norm $\|V\|_{\cE^\MM}$ can be rewritten in a more symmetric form
\begin{equation}
 \|V\|_{\cE^{\MM}}
 =
 \frac{1}{(2\pi)^{d}}\int_{\bT^n\times\bM^m}
 |V(x_1,\ldots,x_n;y_1,\ldots,y_m)|\,\rd x_1\ldots\rd x_n \rd y_1\ldots\rd y_m.
\end{equation} 
\end{rem}

\begin{dfn}\label{dfn:varrho_I}
For $\varepsilon\in[0,\infty)$ and a list of indices $\vI$ of the form~\eqref{eq:list_indices} we define 
\begin{equation}
 \varrho_\varepsilon(\vI) := \varrho_\varepsilon(i_1,m_1) +\ldots +\varrho_\varepsilon(i_n,m_n)\in\bR.
\end{equation}
We also set $\varrho(\vI):=\varrho_0(\vI)$.
\end{dfn}

\begin{rem}
Using the fact that the law of the noise~$\xi_\kappa$ is invariant under spatial translations, one proves that the same is true for the effective force kernels $F^{i,m}_{\kappa,\mu}$. Recall the regularizing kernel $\tilde K_\mu$ introduced in Definition~\ref{def:K}. Since $\tilde K_\mu\in C(\bM)\cap L^1(\bM)$ and $\partial_\kappa^r\partial_\mu^s F^{i,m}_{\kappa,\mu}\in\cV^m$ for all $\kappa,\mu\in(0,1]$, one easily shows that
$$
\tilde K_\mu^{\otimes (n+m)}\ast E^\vI_{\kappa,\mu} \in \cE^{\MM}
$$ 
for all $\kappa,\mu\in(0,1]$, where $n=n(\vI)$, $m=m(\vI)$ and $\MM=\MM(\vI)$. In Theorem~\ref{thm:cumulants} stated in the next section we prove that for an appropriate choice of the counterterms the following bound
\begin{equation}\label{eq:rem_cumulants1}
 \|\tilde K_\mu^{\otimes (n+m)}\ast E^\vI_{\kappa,\mu}\|_{\cE^{\MM}}\lesssim
 [\mu]^{\varrho_\varepsilon(\vI)-\sigma s(\vI)+d(n-1)}
\end{equation}
holds uniformly in $\kappa\in(0,1]$ and $\mu\in(0,1/2]$. In order to see that, at least at a heuristic level, the above estimate is compatible with the bound stated in Theorem~\ref{thm:stochastic_estimates} suppose that $$\bE\|\tilde K_\mu^{\otimes(1+m)}\ast \partial_\mu^s F^{i,m}_{\kappa,\mu}\|_{\cV^m}^n\lesssim [\mu]^{n(\varrho_\varepsilon(i,m)-\sigma s)}.$$ The last bound is almost what Theorem~\ref{thm:stochastic_estimates} says. Observe that this bound implies that
\begin{equation}\label{eq:rem_cumulants2}
 \|\tilde K_\mu^{\otimes (n+m)}\ast E^\vI_{\kappa,\mu}\|_{{\tilde\cV}_\rt^{\MM}}\lesssim
 \,[\mu]^{\varrho_\varepsilon(\vI)-\sigma s(\vI)},
\end{equation}
where 
\begin{equation}
 \|V\|_{{\tilde\cV}_\rt^\MM}
 :=
 \sup_{x_1,\ldots,x_n\in\bM}
 \int |V(x_1,\ldots,x_n;y_1,\ldots,y_m)|\,\rd y_1\ldots\rd y_m.
\end{equation}
Actually, the bound~\eqref{eq:rem_cumulants2} would be also true if $E^\vI_{\kappa,\mu}$ was defined to be the expected value of a product of the effective force kernels and not their joint cumulant. The presence of the extra factor $[\mu]^{d(n-1)}$ appearing in the bound~\eqref{eq:rem_cumulants1} can be understood at least for $\mu\geq\kappa$ by noting the following support property of the cumulants
\begin{multline}
 \supp\, E^\vI_{\kappa,\mu}\subset \{(x_1,\ldots,x_n;y_1,\ldots,y_m)\in\bM^{n+m}\,|\,\\
 |x_2-x_1|\vee\ldots\vee|x_n-x_1|\vee |y_1-x_1|\vee\ldots\vee|y_m-x_1|\leq c\,[\kappa\vee\mu]\}
\end{multline}
for all $\kappa,\mu\in(0,1]$ and some $c\in(0,\infty)$ depending only on $\vI$. Note that the norm $\|\Cdot\|_{\cV^\MM_\rt}$ is weaker than $\|\Cdot\|_{{\tilde\cV}_\rt^\MM}$. Because of the extra factor $[\mu]^{d(n-1)}$ the bound~\eqref{eq:rem_cumulants1} is easier to establish than the bound~\eqref{eq:rem_cumulants2}.
\end{rem}

\begin{rem}[$\spadesuit$]\label{rem:rho2}
For $\varepsilon\in(0,\infty)$ and any list of indices $\vI$ such that $m(\vI)\leq 3i(\vI)$ it holds $\varrho_\varepsilon(\vI)<\varrho(\vI)$. Moreover, \mbox{$\varrho_\varepsilon(\vI)+(n(\vI)-1)d>0$} for $\varepsilon\in(0,\varepsilon_\diamond)$ and lists of indices $\vI$ such that $\varrho(\vI)+(n(\vI)-1)d>0$.
\end{rem}

We now introduce two multilinear maps which encode the two elementary
operations on cumulants produced by differentiating the flow equation. When
$\partial_\mu$ hits one effective force kernel, the deterministic flow
equation inserts one distinguished propagator $\dot G_\mu$ and produces a
product of two effective force kernels. At the level of cumulants, this product
is expanded using~\eqref{eq:cumulants_product}. The first term on the
right-hand side of~\eqref{eq:cumulants_product} keeps all factors in a single
joint cumulant; this operation is represented by the map $\fA$. Analytically,
$\fA$ contracts this enlarged cumulant with the kernel $\dot G_\mu$, integrating
over the two newly created variables. The remaining terms in
\eqref{eq:cumulants_product} split the cumulant into a product of two
cumulants; this operation is represented by the map $\fB$. In this case the
two cumulants are joined by one $\dot G_\mu$-edge, and the symmetrization in
the definition accounts for all possible distributions of the external
variables between the two factors. Thus $\fA$ and $\fB$ provide a compact
analytic notation for the two graphical mechanisms appearing in the
differentiated cumulant expansion: contraction of one connected cumulant and
joining of two cumulants by a $\dot G_\mu$-edge. These two operations are
illustrated explicitly in Exercise~\ref{ex:cumulant_diagrams} below.

\begin{dfn}\label{dfn:maps_A_B}
Fix $n\in\bN_+$, $\hat n\in\{1,\ldots,n\}$, $m_1,\ldots,m_{n+1}\in\bN_0$. Let
\begin{equation}
\begin{gathered}
 \mathsf{m}=(m_1+m_{n+1},m_2,\ldots,m_n)\in\bN_0^n,
 \qquad
 \tilde{\mathsf{m}}=(1+m_1,m_2,\ldots,m_{n+1})\in\bN_0^{n+1},
 \\
 \hat{\mathsf{m}}=(1+m_1,m_2,\ldots,m_{\hat n})\in\bN_0^{\hat n},
 \qquad
 \check{\mathsf{m}}=(m_{\hat n+1},\ldots,m_{n+1})\in\bN_0^{n-\hat n+1}
\end{gathered} 
\end{equation}
and $\underline m=m_1+m_{n+1}$. The bilinear map $\fA\,:\,\sS(\bM) \times\cE^{\tilde{\mathsf{m}}}\to\cE^{\mathsf{m}}$ is defined by
\begin{multline}\label{eq:fA_dfn}
 \fA(G,V)(x_1,\ldots,x_n;\ry_1,\ry_{n+1},\ry_2,\ldots,\ry_n)
 \\
 :=
 \int_{\bM^2} V(x_1,\ldots,x_{n+1}; y,\ry_1,\ldots,\ry_{n+1})\,G(y-x_{n+1})\,\rd y\rd x_{n+1}.
\end{multline}
The trilinear map $\fB\,:\,\sS(\bM)\times \cE^{\hat{\mathsf{m}}}\times\cE^{\check{\mathsf{m}}}\to\cE^{\mathsf{m}}$ is defined by 
\begin{multline}\label{eq:fB_dfn}
 \fB(G,W,U)(x_1,\ldots,x_{n};\ry_1,\ry_{n+1},\ry_2,\ldots,\ry_{n})
 \\
 :=
 \frac{1}{\underline m!}\sum_{\pi\in\cP_{\underline m}} 
 \int_{\bM^2} W(x_1,\ldots,x_{\hat n};y,y_{\pi(1)},\ldots,y_{\pi(m_1)},\ry_2,\ldots,\ry_{\hat n})
 \,G(y-x_{n+1})
 \\
 \times\,
 U(x_{n+1},x_{\hat n+1},\ldots,x_{n};y_{\pi(m_1+1)},\ldots,y_{\pi(\underline m)},\ry_{\hat n+1},\ldots,\ry_{n})\,\rd y\rd x_{n+1}.
\end{multline}
In the above equations $\ry_j\in\bM^{m_j}$, $j\in\{1,\ldots,n+1\}$.
\end{dfn}

\begin{ex}\label{ex:cumulant_diagrams}
	Draw the graphs representing cumulants $E^\vI_{\kappa,\mu}$ with $i(\vI)\leq1$, $m(\vI)\leq3$, $s(\vI)\in\{0,1\}$, $r(\vI)=0$. Use the graphical representation of the effective force kernels in terms of trees introduced in Section~\ref{sec:effective_force} and the property of the joint cumulant of products of Gaussian random variables established in Exercise~\ref{ex:gaussian_cumulant_pairings}. Recall that 
	\begin{equation}
		x\,\<GB>\,y\,=(G-G_\mu)(x-y),
		\qquad
		x\,\<GR>\,y\,=\dot G_\mu(x-y),
		\qquad
		x\,\<C>\,y
		=
		\bE\bigl(\xi_\kappa(x)\xi_\kappa(y)\bigr).
	\end{equation}
	For example, let 
	\begin{equation}
		\vI = ((1,0,0,0),(0,0,0,0)).
	\end{equation}
	Then using the graphical notation introduced in Remark~\ref{rem:explicit_coefficients} and for simplicity omitting the dependence on $\kappa$ and $\mu$ of the diagrams, we obtain
	\begin{equation}
		E^\vI_{\kappa,\mu}=\bE(F^{1,0}_{\kappa,\mu},F^{0,0}_{\kappa,\mu})=\bE\big(\,\<3b>\,,\,\<0>\,\big)=3~\<cumulant1>~.
	\end{equation}
	Since
	\begin{equation}
		\begin{aligned}
			\partial_\mu F_{\kappa,\mu}^{1,0}=\partial_\mu \<3b>= -3~\<3b1>~=
			-3\,\fB(\dot G_\mu,\<2o>_{\kappa,\mu},\<0>_\kappa)
			=
			-\fB(\dot G_\mu,F^{1,1}_{\kappa,\mu},F^{0,0}_{\kappa,\mu}),
		\end{aligned}
	\end{equation} 
	we obtain
	\begin{equation}
		\begin{aligned}
			\partial_\mu E_{\kappa,\mu}^\vI
			=&
			\,3\,\partial_\mu \<cumulant1> 
			\\
			=& 
			- 6~\<cumulant12>-3~\<cumulant11> 
			\\
			=&
			-6\,\fA\Big(\dot G_\mu,~\<cumulant21>~\Big) - 3\,\fB\Big(\dot G_\mu,~\<cumulant22>~,~\<C>~\Big)
		\end{aligned}
	\end{equation}
	Let
	\begin{equation}
		\begin{gathered}
		\vK =  ((0,0,0,0),(1,1,0,0),(0,0,0,0)),
		\\
		\vL=((1,1,0,0))
		\qquad
		\vM=((0,0,0,0),(0,0,0,0)).
		\end{gathered}
	\end{equation}
	Observe that
	\begin{equation}
		\begin{aligned}
			E_{\kappa,\mu}^\vK &= \bE(F^{0,0}_{\kappa,\mu},F^{1,1}_{\kappa,\mu},F^{0,0}_{\kappa,\mu})
			=
			3\,\bE\big(\,\<0>\,,\,\<2o>\,,\,\<0>\,\big) = 6~\<cumulant21>~,
			\\
			E_{\kappa,\mu}^\vL &=\bE(F^{1,1}_{\kappa,\mu})=3\,\bE(\,\<2o>\,)= 3~\<cumulant22>~,
			\\
			E_{\kappa,\mu}^\vM &=\bE(F^{0,0}_{\kappa,\mu},F^{0,0}_{\kappa,\mu})=\bE(\,\<0>\,,\,\<0>\,)=~ \<C>~.
		\end{aligned}
	\end{equation}
	As a result,
	\begin{equation}
		\partial_\mu E_{\kappa,\mu}^\vI = 
		-\fA\big(\dot G_\mu,E_{\kappa,\mu}^\vK\big)
		- \fB\big(\dot G_\mu,E_{\kappa,\mu}^\vL,E_{\kappa,\mu}^\vM\big)\,.
	\end{equation}
\end{ex}

The following two lemmas make precise the observation, illustrated in the preceding exercise, about the scale derivative of the cumulants.

\begin{lem}\label{lem:flow_E_general}
Let $n\in\bN_+$, $i_1\in\bN_0$, $m_1,\ldots,m_n\in\bN_0$ and $I\equiv \{2,\ldots,n\}$. For any random  distributions $\zeta_q\in\sS'(\bM^{1+m_q})$, $q\in I$, the cumulant
\begin{equation}
 \bE(
 \partial_\mu  F^{i_1,m_1}_{\kappa,\mu},
 (\zeta_q)_{q\in I}) \in \sS'(\bM^{n+m_1+\ldots+m_n})
\end{equation}
is a linear combination of the expressions
\begin{equation}\label{eq:flow_E_A}
 \fA\big(\dot G_\mu,
 \bE(
  F^{j,1+k}_{\kappa,\mu},
 (\zeta_q)_{q\in I},
  F^{i_1-j,m_1-k}_{\kappa,\mu})\big)
\end{equation}
or
\begin{equation}\label{eq:flow_E_B}
 \fB\big(\dot G_\mu,
 \bE(F^{j,1+k}_{\kappa,\mu},(\zeta_q)_{q\in I_1})
 ,
 \bE(F^{i_1-j,m_1-k}_{\kappa,\mu},(\zeta_q)_{q\in I_2})\big),
\end{equation}
where $j\in\{1,\ldots,i_1\}$, $k\in\{0,\ldots,m_1\}$ and the subsets $I_1,I_2\subset I$ are such that $I_1\cup I_2= I$ and $I_1\cap I_2=\emptyset$. The coefficients of the above linear combination do not depend on $\kappa,\mu\in(0,1]$. We used the notation introduced in Definition~\ref{dfn:notation_cumulants_distributions}.
\end{lem}
\begin{proof}
The statement follows immediately from the flow equation~\eqref{eq:flow_deterministic_i_m} and~\eqref{eq:cumulants_product}.
\end{proof}

\begin{lem}\label{lem:flow_E_form_bound}
Let $\vJ\equiv (\vJ_1,\ldots,\vJ_n)=((i_1,m_1,s_1,0),\ldots,(i_n,m_n,s_n,0))$ be a list of indices such that $s_1=1$.
\begin{enumerate}
\item[(A)]
The distribution $E^\vJ_{\kappa,\mu}$ can be expressed as a linear combination of distributions of the form
\begin{equation}
 \fA\big(\dot G_\mu,E^{\vK}_{\kappa,\mu}\big)
 \qquad
 \textrm{or}
 \qquad
 \fB\big(\dot G_\mu,
 E^{\vL}_{\kappa,\mu},
 E^{\vM}_{\kappa,\mu}
 \big),
\end{equation}
where the lists of indices $\vK$, $\vL$, $\vM$ satisfy the following conditions
\begin{equation} 
\begin{split}
&n(\vK)=n(\vJ)+1,\\
&i(\vK)=i(\vJ),\\
&m(\vK)=m(\vJ)+1,\\
&s(\vK)=s(\vJ)-1,\\
&\varrho_\varepsilon(\vJ) -\sigma = \varrho_\varepsilon(\vK),
\end{split}
\qquad~\textrm{or}\qquad\quad
\begin{split}
&n(\vL)+n(\vM)=n(\vJ)+1,\\
&i(\vL)+i(\vM)=i(\vJ),\\
&m(\vL)+m(\vM)=m(\vJ)+1,\\ 
&s(\vL)+s(\vM)=s(\vJ)-1,\\
&\varrho_\varepsilon(\vJ) -\sigma = \varrho_\varepsilon(\vL)+\varrho_\varepsilon(\vM).
\end{split}
\end{equation}

\item[(B)] Suppose that the bound
\begin{equation}
 \|\tilde K^{\otimes(n(\vI)+m(\vI))}_\mu\ast E^\vI_{\kappa,\mu}\|_{\cE^{\MM(\vI)}}
 \lesssim 
 [\mu]^{\varrho_\varepsilon(\vI)-\sigma s(\vI)+(n(\vI)-1)d}
\end{equation}
holds uniformly in $\kappa\in(0,1]$, $\mu\in(0,1/2]$ for all lists of indices $\vI$ such that $i(\vI)<i(\vJ)$, or $i(\vI)=i(\vJ)$ and $m(\vI)>m(\vJ)$. Then the above bound holds uniformly in $\kappa\in(0,1]$, $\mu\in(0,1/2]$ for $\vI=\vJ$.
\end{enumerate}
\end{lem}
\begin{rem}
Let $\vJ\equiv (\vJ_1,\ldots,\vJ_n)$ be a list of indices. For a permutation $\pi\in\cP_n$ we set $\pi(\vJ):=(\vJ_{\pi(1)},\ldots,\vJ_{\pi(n)})$. By Remark~\ref{rem:translations} it holds 
$$
\|\tilde K_\mu^{\otimes(n+m(\vJ))}\ast E^\vJ_{\kappa,\mu}\|_{\cE^{\MM(\vJ)}}
=
\|\tilde K_\mu^{\otimes(n+m(\vJ))}\ast E^{\pi(\vJ)}_{\kappa,\mu}\|_{\cE^{\MM(\pi(\vJ))}}.$$ 
Hence, the above lemma is true for all lists of indices $\vJ$ such that $s(\vJ)\neq 0$.
\end{rem}
\begin{proof}[Proof $\mathrm{(}\spadesuit\mathrm{)}$.]
Part~(A) of the lemma follows immediately from Lemma~\ref{lem:flow_E_general} applied with
\begin{equation}
 \zeta_q\equiv \partial_\mu^{s_q}F^{i_q,m_q}_{\kappa,\mu}, 
 \qquad q\in\{2,\ldots,n\}.
\end{equation}
It holds
\begin{equation}
\begin{gathered}
 \vK=((j,k+1,0,0),\,\vJ_2,\ldots,\vJ_n,\,(i_1-j,m_1-k,0,0)),
 \\
 \vL= (j,k+1,0,0)\sqcup (\vJ_q)_{q\in I_1},
 \qquad
 \vM =(i_1-j,m_1-k,0,0) \sqcup (\vJ_q)_{q\in I_2},
\end{gathered} 
\end{equation}
where $\sqcup$ denotes the concatenation of lists, $I_1\cup I_2=I=\{2,\ldots,n\}$, $I_1\cap I_2=\emptyset$ and $j\in\{1,\ldots,i_1\}$, $k\in\{0,\ldots,m_1\}$ coincide with the respective objects in~\eqref{eq:flow_E_A} and~\eqref{eq:flow_E_B}. This together with Definition~\ref{dfn:varrho_I} implies that the lists $\vK$, $\vL$, $\vM$ satisfy the conditions stated in Part~(A). To prove Part~(B) we use Part~(A) and Remark~\ref{lem:fA_fB_Ks}.
\end{proof}

In the remaining part of this section, we collect the technical estimates used
in the proof of the preceding lemma. They serve two purposes: first, to control
the multilinear maps $\fA$ and $\fB$ in the norms $\cE^\MM$, and second, to
record the elementary scale bounds on the kernels $\tilde K_\mu$ and their
periodizations. Since the kernels are defined on the covering space $\bM$ while
the spatial variables are periodic, we begin by recalling the periodization
operation.

\begin{dfn}[$\spadesuit$]\label{dfn:periodization}
Recall $\bT:=\bM/(2\pi\bZ)^d$. For $K\in L^1(\bM)$ we define $\fT K\in L^1(\bT)$ by
\begin{equation}
 \fT K(x):= \sum_{y\in(2\pi\bZ)^d} K(x+y).
\end{equation}
\end{dfn}

\begin{lem}[$\spadesuit$]\label{lem:fA_fB_bounds}
The maps $\fA:\sS(\bM) \times\cE^{\tilde\MM}\to\cE^\MM$,
\mbox{$\fB:\sS(\bM)\times \cE^{\hat\MM}\times\cE^{\check\MM}\to\cE^\MM$} are well defined. It holds
\begin{equation}\label{eq:fA_ieq}
 \|\fA(G,V)\|_{\cE^\MM}
 \leq
 \|\fT |G|\|_{L^\infty(\bM)}\,
 \|V\|_{\cE^{\tilde\MM}},
\end{equation}
\begin{equation}\label{eq:fB_ieq}
 \|\fB(G,W,U)\|_{\cE^\MM}
 \leq
 \|G\|_{L^1(\bM)}\,
 \|W\|_{\cE^{\hat\MM}}\, 
 \|U\|_{\cE^{\check\MM}}.
\end{equation}
\end{lem}
\begin{proof}[Proof sketch]
	We begin with the proof of the second bound. Observe that
	\begin{multline}
		\|\fB(G,W,U)\|_{\cE^{\mathsf{m}}}
		\leq
		\sup_{x_1\in\bT}\int_{\bT^{n-1}\times\bM^{m+2}} |W(x_1,\ldots,x_{\hat n};y,\ry_1,\ldots,\ry_{\hat n})|
		\,|G(y-z)|\,
		\\\times
		|U(z,x_{\hat n+1},\ldots,x_{n};\ry_{n+1},\ry_{\hat n+1},\ldots,\ry_{n})|\,\rd x_2\ldots\rd x_{n}\rd z\rd y\rd\ry_1\ldots\rd\ry_{n+1}.
	\end{multline}
	Consequently,
	\begin{multline}
		\|\fB(G,W,U)\|_{\cE^{\mathsf{m}}}
		\\
		\leq
		\sup_{x_1\in\bT}\int_{\bT^{\hat n-1}\times\bM^{\hat m}} |W(x_1,\ldots,x_{\hat n};y,\ry_1,\ldots,\ry_{\hat n})|
		\,\rd x_2\ldots\rd x_{\hat n}\rd y\rd\ry_1\ldots\rd\ry_{\hat n}
		\\\times
		\sup_{y\in\bM}\int_\bM |G(y-z)|\,\rd z
		~\|U\|_{\cE^{\check{\mathsf{m}}}} 
		\leq
		\|G\|_{L^1(\bM)}\,
		\|W\|_{\cE^{\hat\MM}}\, 
		\|U\|_{\cE^{\check\MM}},
	\end{multline}
	where $\hat m:=1+m_1+\ldots+m_{\hat n}$.
	
	We now turn to the first bound. For simplicity, we present the argument in the infinite-volume setting; the periodic case requires only minor modifications (see~\cite[Lemma~14.10]{Du22}). Under the simplifying assumption \(\bT = \bM\), we have
	\begin{multline}
		\|\fA(G,V)\|_{\cE^{\mathsf{m}}}
		\leq 
		\sup_{x_1\in\bT}\int_{\bM^{n+m+1}} 
		|G(y-x_{n+1})|
		\\
		\times 
		|V(x_1,\ldots,x_{n+1}; y,\ry_1,\ldots,\ry_{n+1})|\,\rd y\rd x_{n+1}\rd x_2\ldots\rd x_n\rd\ry_1\ldots\rd\ry_{n+1}.
	\end{multline}
	It follows that
	\begin{equation}
		\|\fA(G,V)\|_{\cE^{\mathsf{m}}}\leq
		\|G\|_{L^\infty(\bM)}\, \|V\|_{\cE^{\tilde{\mathsf{m}}}},
	\end{equation}
	which establishes the infinite-volume analogue of the first bound.
\end{proof}
\begin{rem}\label{lem:fA_fB_Ks}
	We now apply Lemma~\ref{lem:fA_fB_bounds} to the particular kernel
	$G=\dot G_\mu$ appearing in the flow equation. The following observation shows
	how the smoothing by $\tilde K_\mu$ can be moved through the maps $\fA$ and
	$\fB$, at the cost of replacing $\dot G_\mu$ by $\tilde \fP^2_\mu \dot G_\mu$.
Note that Lemmas~\ref{lem:kernel_G} and~\ref{lem:kernel_simple_fact}~(D) imply that $\|\tilde\fP_\mu^2\dot G_\mu\|_{L^1(\bM)} \lesssim 1$ and
\begin{equation}
 \|\fT |\tilde \fP^2_\mu \dot G_\mu|\|_{L^\infty(\bM)}\leq \|\fT \tilde K_\mu\|_{L^\infty(\bM)} \,\|\tilde \fP^3_\mu \dot G_\mu\|_{L^1(\bM)}\lesssim [\mu]^{-d}\,\|\tilde \fP^3_\mu \dot G_\mu\|_{L^1(\bM)},
\end{equation}
uniformly in $\mu\in(0,1]$. Moreover, using the fact that $\tilde \fP_\mu \tilde K_\mu=\delta_0$ one shows that for all $\mu\in(0,1]$ it holds
\begin{equation}
 \tilde K_\mu^{\otimes(n+m)}\ast\fA(G,V)=
 \fA\big(\tilde \fP^2_\mu G,
 \tilde K_\mu^{\otimes (n+m+2)}\ast V\big)
\end{equation}
and
\begin{equation}
 \tilde K_\mu^{\otimes(n+m)}\ast\fB(G,W,U)=
 \fB\big(\tilde \fP^2_\mu G,
 \tilde K_\mu^{\otimes(\hat n+\hat m+1)}\ast W,
 \tilde K_\mu^{\otimes(n-\hat n+m-\hat m+1)}\ast U\big),
\end{equation}
where \mbox{$m=m_1+\ldots+m_{n+1}$} and \mbox{$\hat m=m_1+\ldots+m_{\hat n}$}. Consequently, by the above lemma we have
\begin{equation}
 \|\tilde K_\mu^{\otimes(n+m)}\ast\fA(\dot G_\mu,V)\|_{\cE^\MM}
 \lesssim
 [\mu]^{-d}\,
 \|\tilde K_\mu^{\otimes(n+m+2)}\ast V\|_{\cE^{\tilde\MM}}
\end{equation}
and
\begin{equation}
 \|\tilde K_\mu^{\otimes(n+m)}\ast \fB(\dot G_\mu,W,U)\|_{\cE^\MM}
 \lesssim 
 \|\tilde K_\mu^{\otimes(\hat n+\hat m+1)}\ast W\|_{\cE^{\hat\MM}}\, 
 \|\tilde K_\mu^{\otimes(n-\hat n+m-\hat m+1)}\ast U\|_{\cE^{\check\MM}}
\end{equation}
uniformly in $\mu\in(0,1]$ and $V\in\cE^{\tilde\MM}$,
$W\in\cE^{\hat\MM}$, $U\in\cE^{\check\MM}$.
\end{rem}

We finish with a lemma stating standard auxiliary estimates on the smoothing kernels~$\tilde K_\mu$, whose proof is left as an exercise.
\begin{lem}[$\spadesuit$]\label{lem:kernel_simple_fact}
Let $a\in\bN_0^d$ and $p\in[1,\infty]$. The following is true:
\begin{enumerate}
\item[(A)]
If $|a|\leq d$, then $\|\partial^a \tilde K_\mu\|_\cK
 \lesssim [\mu]^{-|a|}$ uniformly in $\mu\in(0,1]$. 
\item[(B)]
It holds $\|\tilde\fP_\mu\partial_\mu \tilde K_\mu\|_{\cK} \lesssim [\mu]^{-\sigma}$ uniformly in $\mu\in(0,1]$.
\item[(C)] 
$\|\tilde K_\mu\|_{L^p(\bM)}\lesssim [\mu]^{-d (p-1)/p}$ uniformly in $\mu\in(0,1]$.
\item[(D)] 
$\|\fT \tilde K_\mu\|_{L^p(\bT)}\lesssim [\mu]^{-d (p-1)/p}$ uniformly in $\mu\in(0,1]$.
\end{enumerate}
\end{lem}
\begin{ex}[$\spadesuit$]
	Prove the preceding lemma. For Item~{\rm(D)}, identify $\bT$ with
	\mbox{$[-\pi,\pi)^d\subset\bR^d$}. Write
	\[
	\fT\tilde K_\mu
	=
	\tilde K_\mu
	+
	\bigl(\fT\tilde K_\mu-\tilde K_\mu\bigr).
	\]
	Using the bound $\tilde K_1(x)\lesssim\exp(-|x|)$, uniformly in
	$x\in\bR^d$, show that for every $p\in[1,\infty]$ and every $q>0$,
	\[
	\bigl\|\fT\tilde K_\mu-\tilde K_\mu\bigr\|_{L^p(\bT)}
	\lesssim
	[\mu]^q .
	\]
\end{ex}

\section{Uniform bounds for cumulants}\label{sec:cumulants_uniform_bounds}

We now prove the main cumulant estimate, which is the core of the flow equation
approach. After choosing the local counterterms in the force appropriately, we
show that the cumulants of the effective force kernels satisfy the uniform bound
\[
\|\tilde K^{\otimes(n+m)}_{\mu}\ast E^\vI_{\kappa,\mu}\|_{\cE^\MM}
\lesssim
[\mu]^{\varrho_\varepsilon(\vI)-\sigma s(\vI)+d(n-1)}
\]
for all lists of indices \(\vI\) under consideration, uniformly in the
ultraviolet cutoff \(\kappa\in(0,1]\) and in the scale
\(\mu\in(0,1/2]\). The exponent
\(\varrho_\varepsilon(\vI)-\sigma s(\vI)\) is dictated by power counting,
whereas the additional factor \([\mu]^{d(n-1)}\) reflects the spatial
localization of connected correlations.

The proof proceeds by induction through the hierarchy of effective force
kernels. The cumulant flow equation, derived in
Lemma~\ref{lem:flow_E_form_bound}, provides the required bounds for cumulants
containing at least one scale derivative. Cumulants without scale derivatives
are then recovered by integration in the scale parameter. For irrelevant
cumulants this integration starts at \(\mu=0\) and is convergent by power
counting.

The delicate point is the treatment of expectations of relevant kernels, namely
those \(F^{i,m}_{\kappa,\mu}\) with \(\varrho(i,m)\leq0\). For such kernels the
scale derivative has the correct formal degree, but the corresponding bound is
not integrable at \(\mu=0\). This is precisely where renormalization enters. We
separate the expectation of a relevant kernel into a local part and a nonlocal
remainder. The local part is absorbed into the counterterm, while the nonlocal
remainder gains two powers of the scale and hence becomes integrable at the
origin. In the present model, the symmetries of the equation reduce the possible
local counterterms to the mass counterterms \(c_\kappa^{(i)}\varphi\).

Let us explain this point more concretely. Suppose that
\(F^{i,m}_{\kappa,\mu}\) is relevant. The flow equation gives a bound of the
form
\begin{equation}
	\|\tilde K^{\otimes(1+m)}_{\eta}\ast \bE \partial_\eta F^{i,m}_{\kappa,\eta}\|_{\cE^m}
	\lesssim
	[\eta]^{\varrho_\varepsilon(i,m)-\sigma}.
\end{equation}
Together with the identity
\begin{equation}
	F^{i,m}_{\kappa,\mu}
	=
	F^{i,m}_{\kappa,0}
	+
	\int_0^\mu \partial_\eta F^{i,m}_{\kappa,\eta}\,\rd\eta,
\end{equation}
this estimate does not directly imply
\begin{equation}
	\|\tilde K^{\otimes(1+m)}_{\mu}\ast \bE F^{i,m}_{\kappa,\mu}\|_{\cE^m}
	\lesssim
	[\mu]^{\varrho_\varepsilon(i,m)}.
\end{equation}
Indeed, the function
\([\eta]^{\varrho_\varepsilon(i,m)-\sigma}\) is not integrable at
\(\eta=0\) whenever \(\varrho_\varepsilon(i,m)<0\).

One possible way around this obstruction would be to impose a boundary
condition at a positive scale, for instance \(F^{i,m}_{\kappa,1/2}=0\), and use
the representation
\begin{equation}
	F^{i,m}_{\kappa,\mu}
	=
	F^{i,m}_{\kappa,1/2}
	-
	\int_\mu^{1/2}
	\partial_\eta F^{i,m}_{\kappa,\eta}\,\rd\eta.
\end{equation}
This would yield the desired bounds for the expectations of the relevant
kernels. Its drawback is that the kernels defining the original force
functional would then have to be modified by nonlocal correction terms:
\begin{equation}
	F^{i,m}_{\kappa}
	=
	F^{i,m}_{\kappa,0}
	=
	F^{i,m}_{\kappa,1/2}
	-
	\int_0^{1/2}
	\partial_\eta F^{i,m}_{\kappa,\eta}\,\rd\eta.
\end{equation}
Such a modification is undesirable, since it would destroy the local form of
the counterterms.

We therefore proceed differently. We decompose each relevant kernel
\[
F^{i,m}_{\kappa,\mu}(x;\rd y_1,\ldots,\rd y_m)
\]
into a local part, supported on the diagonal
\[
\{x=y_1=\cdots=y_m\},
\]
and a nonlocal remainder, which behaves like an irrelevant kernel. This
decomposition keeps the additional counterterms local. Using the symmetries and
the specific form of the equation, the analysis reduces to the family
\[
(F^{i,1}_{\kappa,\mu})_{i\in\{1,\ldots,i_\sharp\}}.
\]
Thus, in order to bound \(\bE F^{i,1}_{\kappa,\mu}\), we decompose
\[
\bE F^{i,1}_{\kappa,\mu}(x;\rd y)
\]
into a component proportional to \(\delta_x(\rd y)\) and a remainder. The maps
used for this decomposition are introduced next.
\begin{dfn}
For $m\in\bN_+$ we define $\delta^{[m]}\in\sS'(\bM^{1+m})$ by the equality
\begin{equation}
 \langle \delta^{[m]},\psi\otimes\varphi_1\otimes\ldots\otimes\varphi_m\rangle:=\int_\bM\psi(x)\varphi_1(x)\ldots\varphi_m(x)\,\rd x
\end{equation} 
for all $\psi,\varphi_1,\ldots,\varphi_m\in\sS(\bM)$. Let $\cX^a(x;y):=(x-y)^a$ for $a\in\bN_0^d$ and $x,y\in\bM$ and let $V\in\cV_\rt^1$ be such that $V(x;y)=V(x-y;0)$, $V(x;y)=V(-x;-y)$ and $\cX^a V\in\cV_\rt^1$ for all $a\in\bN_0^d$. We define $\fI V := \int_{\bM} V(x;y)\rd y\in\bR$. For $a\in\bN_0^d$ we define $\fR^a V\in\cV_\rt^1$ by the equality
\begin{equation}
 (\fR^a V)(x;y):=\frac{|a|}{a!}\int_0^1 (1-\tau)^{|a|-1}/\tau^d~ (\cX^a V)(x;x+(y-x)/\tau)\,\rd\tau
\end{equation}
for all $x,y\in\bM$.
\end{dfn}

\begin{ex}\label{ex:taylor}
Prove that the following equality $V = (\fI V)\,\delta^{[1]} + \sum_{|a|=2}\partial^a\fR^a V$ holds in $\sS'(\bM^2)$, where the sum is over $a\in\bN_0^d$ and $\partial^a$ denotes the derivative with respect to the second argument. Show that $|\fI V|\leq \|V\|_{\cV_\rt^1}$ and $\|\fR^a V\|_{\cV_\rt^1}\leq \|\cX^a V\|_{\cV_\rt^1}$. Hint: Use the integral form of the Taylor remainder.
\end{ex}

\begin{rem}
	The proof of the following theorem is inspired by Polchinski's flow-equation
	argument for perturbative renormalizability in quantum field theory
	\cite{polchinski1984}; see~\cite{muller2003} for a review. The basic strategy is
	to use the renormalization group flow equation to propagate scale-dependent
	bounds, while fixing the relevant local parts by renormalization conditions.
	For non-perturbative uses of related flow-equation methods, see
	\cite{brydges1987Mayer,bauerschmidt2021}.
\end{rem}

\begin{rem}\label{rem:symmetries}
	The equation
\begin{equation}
	\varPhi_\kappa=G\ast F_\kappa[\varPhi_\kappa],
	\qquad
	F_\kappa[\varphi]:= \xi_\kappa 
	+\lambda\, \varphi^3
	+\sum_{i=1}^{i_\sharp}\lambda^i\, c_\kappa^{(i)} \varphi,
\end{equation}
	is invariant under the transformations
	\[
	(\varPhi_\kappa,\xi_\kappa)\mapsto -(\varPhi_\kappa,\xi_\kappa)
	\qquad\text{and}\qquad
	(\varPhi_\kappa,\xi_\kappa)\mapsto
	(\varPhi_\kappa(-\,\Cdot),\xi_\kappa(-\,\Cdot)).
	\]
	Since the law of the noise \(\xi_\kappa\) is itself invariant under
	\[
	\xi_\kappa\mapsto -\xi_\kappa
	\qquad\text{and}\qquad
	\xi_\kappa\mapsto \xi_\kappa(-\,\Cdot),
	\]
	it follows that
	\[
	E^\vI_{\kappa,\mu}=0
	\qquad\text{unless}\qquad
	n(\vI)+m(\vI)\in2\bN_0,
	\]
	and that the cumulants \(E^\vI_{\kappa,\mu}\) are invariant under inversion through the origin in \(\bM^{n(\vI)+m(\vI)}\).
	
	These symmetry properties play a crucial role in determining the possible counterterms required for renormalization. In general, the relevant local parts of the effective force kernels are proportional to distributions of the form
	\begin{equation}
		\partial_{y_1}^{a_1}\delta_x(\rd y_1)\cdots
		\partial_{y_m}^{a_m}\delta_x(\rd y_m),
		\qquad
		|a_1|+\cdots+|a_m|+\varrho(1,m)\le0,
	\end{equation}
	where we used that
	\[
	\varrho(i,m)\ge \varrho(1,m),
	\qquad
	i\in\bN_+.
	\]
	Consequently, the corresponding local counterterms that may appear in the equation are of the form
	\begin{equation}
		(\partial^{a_1}\varphi)\cdots(\partial^{a_m}\varphi),
		\qquad
		|a_1|+\cdots+|a_m|+\varrho(1,m)\le0.
	\end{equation}
	In the case \(d\in\{2,\ldots,6\}\) and \(\sigma\in(d/3,d/2]\), we have
	\begin{equation}
		\varrho(1,m)
		=
		-\frac d2
		+
		m\Bigl(\frac d2-\sigma\Bigr)
		+
		(3\sigma-d).
	\end{equation}
	A straightforward analysis then shows that the only possible counterterms are
	\[
	\varphi^3,
	\qquad
	\varphi^2,
	\qquad
	\varphi,
	\qquad
	\partial\varphi,
	\qquad
	1.
	\]
	The above symmetry considerations imply that the counterterms proportional to
	\[
	\varphi^2,
	\qquad
	\partial\varphi,
	\qquad
	1
	\]
	vanish identically. Moreover, it turns out that the cubic term \(\varphi^3\) does not require renormalization. As a consequence, the only remaining counterterms are the mass renormalization terms proportional to \(\varphi\).
\end{rem}

\begin{thm}\label{thm:cumulants}
	There exists a choice of the mass renormalization constants $$(c^{(i)}_\kappa)_{i\in\{1,\ldots,i_\sharp\}}$$ in~\eqref{eq:force} such that for all lists of indices
	$\vI=((i_1,m_1,s_1,0),\ldots,(i_n,m_n,s_n,0))$ the bound
	\begin{equation}\label{eq:thm_cumulants}
		\|\tilde K^{\otimes(n+m)}_{\mu}\ast E^\vI_{\kappa,\mu}\|_{\cE^\MM}
		\lesssim 
		[\mu]^{\varrho_\varepsilon(\vI)-\sigma s(\vI)+d(n-1)}
	\end{equation}
	holds uniformly in $\kappa\in(0,1]$, $\mu\in(0,1/2]$, where $\MM=(m_1,\ldots,m_n)$, $m=m_1+\ldots+m_n$.
\end{thm}

\begin{proof}
We first note that the theorem is trivially true for all lists of indices $\vI$ such that $m(\vI)>3i(\vI)$ since then $E^\vI_{\kappa,\mu}=0$. The rest of the proof is by induction.

\textit{The base case:} 
Consider a list of indices $\vI$ such that \mbox{$i(\vI)=0$}. In this case the cumulants $E^\vI_{\kappa,\mu}$ coincide with the cumulants of the white noise $\xi_\kappa$. The only non-vanishing cumulant is the covariance corresponding to $n(\vI)=2$, $\MM(\vI)=(0,0)$, $m(\vI)=0$ and $s(\vI)=0$. It holds
\begin{equation}
 \|\bE(\tilde K_\mu\ast\xi_\kappa,\tilde K_\mu\ast\xi_\kappa)\|_{\cE^\MM} \leq \sup_{x_1\in\bT}\int_{\bT}|\bE(\xi(x_1)\xi(\rd x_2))| = 1.
\end{equation}
This finishes the proof of the base case.

\textit{Induction step:}
Fix $i\in\bN_+$ and $m\in\bN_0$. Assume that the theorem is true for all lists of indices $\vI$ such that either $i(\vI)<i$, or $i(\vI)=i$ and $m(\vI)>m$. We shall prove the theorem for all $\vI$ such that $i(\vI)=i$ and $m(\vI)=m$. 

Consider the case $s(\vI)> 0$. Then we use the flow equation for cumulants introduced in the previous section. More precisely, the bound~\eqref{eq:thm_cumulants} follows from the inductive assumption and Lemma~\ref{lem:flow_E_form_bound}~(B). 

It remains to prove the statement for lists of indices $$\vI = ((i_1,m_1,0,0),\ldots,(i_n,m_n,0,0))$$ such that $s(\vI)=0$. It follows from Definition~\ref{dfn:cumulants_eff_force} of the cumulants $E^\vI_{\kappa,\mu}$ that 
\begin{equation}\label{eq:thm_cumulants_ind_step}
 E^\vI_{\kappa,\mu}
 = 
 E^\vI_{\kappa,0} + \sum_{q=1}^n \int_0^\mu E^{\vI_q}_{\kappa,\eta}\,\rd\eta,
\end{equation}
where
\begin{equation}
 \vI_q = ((i_1,m_1,0,0),\ldots, (i_q,m_q,1,0),\ldots,(i_n,m_n,0,0)).
\end{equation}
Note that $s(\vI_q)=1$, hence the bound~\eqref{eq:thm_cumulants} has already been established for $E^{\vI_q}_{\kappa,\eta}$.

First, let us analyze the {\it irrelevant cumulants}, i.e. those with $\vI$ such that
\begin{equation}
 \varrho(\vI)+(n(\vI)-1)d>0.
\end{equation}
Let us recall the non-zero force kernels 
\begin{equation}
 F^{0,0}_\kappa=\xi_\kappa,
 \quad
 F^{1,3}_\kappa=\delta^{[3]},
 \quad
 F^{i,1}_\kappa=c_\kappa^{(i)}\,\delta^{[1]}, 
 \quad
 i\in\{1,\ldots,i_\sharp\}.
\end{equation}
In particular $F^{i,m}_{\kappa,0}=F^{i,m}_\kappa$ is deterministic if $i\in\bN_+$. If $n(\vI)>1$, then $E^\vI_{\kappa,0}$ is a joint cumulant of a list of at least two random distributions. Since $i(\vI)=i>0$ one of these distributions is deterministic and the cumulant vanishes. If $n(\vI)=1$ and $\varrho(\vI)>0$, then $E^\vI_{\kappa,0}$ coincides with $F^{i,m}_\kappa$ for some $i,m\in\bN_0$ such that $\varrho(i,m)>0$. However, $F^{i,m}_\kappa$ vanishes for $i,m\in\bN_0$ such that $\varrho(i,m)>0$. Hence, we conclude that $E^\vI_{\kappa,0}=0$ for all irrelevant cumulants. Using this fact and~\eqref{eq:thm_cumulants_ind_step} we arrive at
\begin{equation}\label{eq:K_bound_s_u}
 \|\tilde K_\mu^{\otimes(n+m)}\ast E^\vI_{\kappa,\mu}\|_{\cE^\MM}\leq 
 \sum_{q=1}^n\int_0^\mu \|\tilde K_\eta^{\otimes(n+m)}\ast E^{\vI_q}_{\kappa,\eta}\|_{\cE^\MM}\,\rd\eta,
\end{equation}
which follows from the estimate proved in Exercise~\ref{ex:kernel_u_v_cumulants}. Using the induction hypothesis we arrive at
\begin{equation}
 \|\tilde K_\mu^{\otimes(n+m)}\ast E^\vI_{\kappa,\mu}\|_{\cE^\MM} 
 \lesssim \int_0^\mu [\eta]^{\varrho_\varepsilon(\vI)-\sigma+d(n-1)} 
 \lesssim [\mu]^{\varrho_\varepsilon(\vI)+d(n-1)}.
\end{equation}
Note that to get the last bound we crucially used the fact that $\varrho_\varepsilon(\vI)+d(n-1)>0$ for sufficiently small $\varepsilon\in(0,\infty)$. This finishes the proof of the induction step for the irrelevant cumulants.

Next, let us analyze the {\it relevant cumulants}, i.e. those with $\vI$ such that
\begin{equation}
 \varrho(\vI)+(n(\vI)-1)d\leq0.
\end{equation}
Note that for $i(\vI)\geq 1$ the above inequality implies that \mbox{$n(\vI)=1$}. Consequently, by Remark~\ref{rem:symmetries} the non-trivial cases are 
\[
\vI=((1,3,0,0))
\qquad\text{or}\qquad
\vI=((i,1,0,0))
\quad\text{with } i\in\{1,\ldots,i_\sharp\}.
\] 
Thus, we have to prove bounds for the following cumulants
\begin{equation}
 \bE F^{1,3}_{\kappa,\mu},
 \qquad\text{and}\qquad
 \bE F^{i,1}_{\kappa,\mu},
 \quad
 i\in\{1,\ldots,i_\sharp\}.
\end{equation}
Recall that $F^{i,m}_{\kappa,0}=F^{i,m}_\kappa$. Since $\partial_\mu F^{1,3}_{\kappa,\mu}=0$ we have $\bE F^{1,3}_{\kappa,\mu}=\delta^{[3]}$. Consequently,
\begin{equation}
 \|\tilde K_\mu^{\otimes4} \ast \bE F^{1,3}_{\kappa,\mu}\|_{\cV_\rt^3} =1 \leq[\mu]^{-\varepsilon} = [\mu]^{\varrho_\varepsilon(1,3)}.
\end{equation}
Let us now study the cumulants $\bE F^{i,1}_{\kappa,\mu}$ with $ i\in\{1,\ldots,i_\sharp\}$. By Exercise~\ref{ex:taylor} the bound for the cumulant 
$$
E^{\vI_1}_{\kappa,\mu}=\bE \partial_\mu F^{i,1}_{\kappa,\mu}=:\dot E^i_{\kappa,\mu},
$$
which follows from the induction hypothesis, implies that
\begin{equation}\label{eq:cumulants_bound_I}
 \|\tilde K_\mu^{\otimes2}\ast\dot E^i_{\kappa,\mu}\|_{\cV_\rt^1} \lesssim [\mu]^{\varrho_\varepsilon(i,1)-\sigma},
 \qquad
 |\fI(\tilde K_\eta^{\otimes2}\ast\dot E^i_{\kappa,\eta})|\lesssim 
 [\eta]^{\varrho_\varepsilon(i,1)-\sigma}.
\end{equation}
Using Exercise~\ref{ex:taylor} we introduce the following decomposition
\begin{equation}\label{eq:E_decomposition}
 \bE F^{i,1}_{\kappa,\mu} = 
 \hat{E}^i_{\kappa,\mu}\,\delta^{[1]}
 +
 \check E^i_{\kappa,\mu}
 +
 \tilde E^i_{\kappa,\mu}
 \in\sS'(\bM^2),
\end{equation}
where $\hat{E}^i_{\kappa,\mu}\in\bR$ and $\check E^i_{\kappa,\mu},\tilde E^i_{\kappa,\mu}\in\sS'(\bM^2)$ are defined by the equalities
\begin{equation}
\begin{split}
 \partial_\eta \hat E^i_{\kappa,\eta}
 &:=
 \fI(\tilde K_\eta^{\otimes2}\ast\dot E^i_{\kappa,\eta}),
 \\
 \partial_\eta \check E^i_{\kappa,\eta}
 &:=
 \textstyle\sum_{|a|=2}\partial^a\fR^a(\tilde K_\eta^{\otimes2}\ast\dot E^i_{\kappa,\eta}),
 \\
 \partial_\eta \tilde E^i_{\kappa,\eta}
 &:=
 \dot E^i_{\kappa,\eta}- \tilde K_\eta^{\otimes2}\ast\dot E^i_{\kappa,\eta},
\end{split}
\qquad\qquad
\begin{split}
\hat E^i_{\kappa,0}&:=c_\kappa^{(i)}
\\
\check E^i_{\kappa,0}&:=0,
\\
\tilde E^i_{\kappa,0}&:=0
\end{split}
\end{equation}
The motivation behind the above decomposition is that, as we shall see below, $\partial_\eta \check E^i_{\kappa,\eta}$ and $\partial_\eta \tilde E^i_{\kappa,\eta}$ satisfy the following bound
\begin{multline}
 \|K_\mu^{\otimes2}\ast\partial_\eta \check E^i_{\kappa,\eta}\|_{\cV_\rt^1}\vee
 \|K_\mu^{\otimes2}\ast\partial_\eta \tilde E^i_{\kappa,\eta}\|_{\cV_\rt^1}
 \\
 \lesssim [\eta/\mu]^2\,\|\tilde K_\eta^{\otimes2}\ast\dot E^i_{\kappa,\eta}\|_{\cV_\rt^1}
 \lesssim
  [\mu]^{-2}\,[\eta]^{\varrho_\varepsilon(i,1)-\sigma+2}.
\end{multline}
It turns out that for $d\in\{2,\ldots,6\}$, $\sigma\in(d/3,d/2]$ and sufficiently small $\varepsilon\in(0,\infty)$ it holds $\varrho_\varepsilon(i,1)+2>0$ for all $i\in\bN_+$. Consequently, the RHS of the above estimate is integrable in $\eta$ at $\eta=0$. On the other hand, the bound for $\partial_\eta\hat{E}^i_{\kappa,\eta}\,\delta^{[1]}$, given in~\eqref{eq:cumulants_bound_I}, is not integrable in $\eta$ at $\eta=0$ but this contribution is local and, as we will prove in a moment, $\hat{E}^i_{\kappa,\mu}\in\bR$ can be bounded by making a suitable choice of the counterterm $c_\kappa^{(i)}\in\bR$.

Let us first study the local term $\hat{E}^i_{\kappa,\mu}\,\delta^{[1]}$. We start by fixing the counterterm to be
\begin{equation}
 c_\kappa^{(i)} := -\int_0^{1/2} \fI(\tilde K_\eta^{\otimes2}\ast\dot E^i_{\kappa,\eta})\,\rd \eta.
\end{equation}
Then it holds
\begin{equation}
 \hat E^i_{\kappa,\mu} = c_\kappa^{(i)}+\int_0^\mu  \fI(\tilde K_\eta^{\otimes2}\ast\dot E^i_{\kappa,\eta})\,\rd\eta
 =
 -\int_\mu^{1/2} \fI(\tilde K_\eta^{\otimes2}\ast\dot E^i_{\kappa,\eta})\,\rd\eta.
\end{equation}
Consequently, since $\varrho_\varepsilon(i,1)<0$ for $i\in\{1,\ldots,i_\sharp\}$ we obtain 
\begin{equation}
 |\hat E^i_{\kappa,\mu}|\lesssim \int_\mu^{1/2}\,[\eta]^{\varrho_\varepsilon(i,1)-\sigma}\,\rd\eta \lesssim 
 [\mu]^{\varrho_\varepsilon(i,1)}.
\end{equation}
This implies the desired bound
\begin{equation}
 \|\tilde K_\mu^{\otimes2}\ast(\hat E^i_{\kappa,\mu}\,\delta^{[1]})\|_{\cV_\rt^1} \lesssim [\mu]^{\varrho_\varepsilon(i,1)}.
\end{equation}

Next, let us proceed to the estimates for the non-local terms $\check E^i_{\kappa,\eta}$ and $\tilde E^i_{\kappa,\eta}$. Using the boundary condition $\check E^i_{\kappa,0}=0$ and the bound $\|\partial^a \tilde K_\mu\|_\cK\lesssim [\mu]^{-|a|}$ proved in Lemma~\ref{lem:kernel_simple_fact}~(A) we estimate
\begin{equation}
 \|\tilde K_\mu^{\otimes2}\ast\check E^i_{\kappa,\mu}\|_{\cV_\rt^1} 
 \lesssim
 [\mu]^{-2}\sup_{|a|=2}\int_0^\mu
 \|\fR^a(\tilde K_\eta^{\otimes2}\ast \dot E^i_{\kappa,\eta})\|_{\cV_\rt^1}\,\rd\eta.
\end{equation}
By Exercise~\ref{ex:taylor} the RHS is bounded up to a constant by
\begin{equation}
 \sup_{|a|=2}~[\mu]^{-2}\int_0^\mu
 \|\cX^a(\tilde K_\eta^{\otimes2}\ast \dot E^i_{\kappa,\eta})\|_{\cV_\rt^1}\,\rd\eta
\end{equation}
Recall that by Lemma~\ref{lem:support} there exists $c\in(0,\infty)$ such that for all $\eta\in(0,1/2]$ the kernel $\partial_\eta F^{i,1}_{\kappa,\eta}$ as well as its expected value $\dot E^i_{\kappa,\eta}=\bE\partial_\eta F^{i,1}_{\kappa,\eta}$ are supported in the set
\begin{equation}
 \{(x,y)\in\bM^2\,|\,|x-y|\leq c\,[\eta]\}.
\end{equation}
Pretending that there exists $c\in(0,\infty)$ such that for all $\eta\in(0,1/2]$ the kernel $\tilde K_\eta$ is supported~in 
$$
\{x\in\bM\,|\,|x|<c\,[\eta]\}
$$
we obtain that $\tilde K_\eta^{\otimes2}\ast \dot E^i_{\kappa,\eta}$ has the same support property as $\dot E^i_{\kappa,\eta}$. Consequently,
\begin{equation}
 \|\cX^a(\tilde K_\eta^{\otimes2}\ast \dot E^i_{\kappa,\eta})\|_{\cV_\rt^1}\lesssim
 [\eta]^2\,\|\tilde K_\eta^{\otimes2}\ast \dot E^i_{\kappa,\eta}\|_{\cV_\rt^1}.
\end{equation}
Since the kernel $\tilde K_\eta$ does not have the above-mentioned support property the rigorous proof of the above estimate is slightly more complicated and is given in Lemma~\ref{lem:support_V_Kmu} below. Taking into account the induction hypothesis, we obtain the bound
\begin{equation}
 \|\tilde K_\mu^{\otimes2}\ast\check E^i_{\kappa,\mu}\|_{\cV_\rt^1} 
 \lesssim
 [\mu]^{-2}\int_0^\mu
 [\eta]^2\,\|\tilde K_\eta^{\otimes2}\ast \dot E^i_{\kappa,\eta}\|_{\cV_\rt^1}\,\rd\eta
 \lesssim 
 [\mu]^{-2}\int_0^\mu
 [\eta]^{\varrho_\varepsilon(i,1)-\sigma+2}\,\rd\eta.
\end{equation}
Similarly, by Exercise~\ref{ex:K_cumulants} we have 
\begin{equation}
 \|\tilde K_\mu^{\otimes2}\ast\tilde E^i_{\kappa,\mu}\|_{\cV_\rt^1} 
 \lesssim
 [\mu]^{-2}\int_0^\mu [\eta]^2\,
 \|\tilde K_\eta^{\otimes2}\ast \dot E^i_{\kappa,\eta}\|_{\cV_\rt^1}\,\rd\eta
 \lesssim 
 [\mu]^{-2}\int_0^\mu
 [\eta]^{\varrho_\varepsilon(i,1)-\sigma+2}\,\rd\eta.
\end{equation}
Finally, we use the fact that $\varrho_\varepsilon(i,1)+2>0$ to show
\begin{equation}
 [\mu]^{-2}\int_0^\mu
 [\eta]^{\varrho_\varepsilon(i,1)-\sigma+2}\,\rd\eta
 \lesssim
 [\mu]^{\varrho_\varepsilon(i,1)},
\end{equation}
which concludes the proof.
\end{proof}

\begin{ex}
Verify~\eqref{eq:E_decomposition} using Remark~\ref{rem:symmetries}.
\end{ex}

\begin{ex}
Check that in the Da~Prato--Debussche regime~\cite{daprato2003} corresponding to $\sigma\in(2d/5,d/2]$ the above proof simplifies as the use of the maps $\fI$ and $\fR$ is redundant.
\end{ex}

\begin{rem}[$\spadesuit$]
Observe that the counterterms $(c^{(i)}_\kappa)_{i\in\{1,\ldots,i_\sharp\}}$ are implicitly fixed by the following renormalization conditions
\begin{equation}
 \int_\bM \bE F^{i,1}_{\kappa,\mu=1/2}(x;\rd y) = 0,
 \qquad\quad
 i\in\{1,\ldots,i_\sharp\}.
\end{equation}
The procedure of fixing the counterterms works because of the property of the effective force kernels mentioned in Remark~\ref{rem:counterterms}.
\end{rem}

We conclude this section by recording the auxiliary estimates that were used in
the proof of Theorem~\ref{thm:cumulants}. The first exercise states the
monotonicity of the \(\cE^\MM\)-norm under additional smoothing: passing from a
finer scale \(\eta\) to a coarser scale \(\mu\) cannot increase the norm. The
second exercise gives a quantitative comparison between smoothing at scale
\(\mu\) and smoothing first at scale \(\eta\) and then at scale \(\mu\); the
gain \([\eta/\mu]^2\) is the analytic input behind the improved bound for the
nonlocal remainder. The following lemma formalizes the corresponding spatial
localization estimate: if a kernel is supported within distance \(O([\mu])\) of
the diagonal, then multiplication by powers of the relative coordinate costs
the expected factor \([\mu]^{|a|}\), even after convolution with the
non-compactly supported kernel \(\tilde K_\mu\). Finally, the last exercise
extends Theorem~\ref{thm:cumulants} to cumulants containing derivatives with
respect to the ultraviolet cutoff \(\kappa\), which is the version needed for
the convergence of the enhanced noise as \(\kappa\searrow0\).

\begin{ex}[$\spadesuit$]\label{ex:kernel_u_v_cumulants}
Let $n\in\bN_+$, $\MM=(m_1,\ldots,m_n)\in\bN_0^n$ and $m=m_1+\ldots+m_n$. Using Exercise~\ref{ex:K_mu}~(3) prove that for all $0\leq\eta\leq\mu\leq1$ and $V\in\cE^\MM$ it holds
\begin{equation}
 \|\tilde K_\mu^{\otimes(n+m)}\ast V\|_{\cE^\MM}\leq \|\tilde K_\eta^{\otimes(n+m)}\ast V\|_{\cE^\MM}.
\end{equation}
\end{ex}

\begin{ex}[$\spadesuit$]\label{ex:K_cumulants}
Show that it holds
\begin{equation}
 \|(\tilde K_\mu^{\otimes2}-\tilde K_\mu^{\otimes2}\ast \tilde K_\eta^{\otimes2})\ast V\|_{\cV_\rt^1}\lesssim [\eta/\mu]^2~\|\tilde K_\eta^{\otimes2}\ast V\|_{\cV_\rt^1}
\end{equation}
uniformly over $\mu,\eta\in(0,1]$ and $V\in\cV_\rt^1$. Hint: Let $\hat K_\mu\in\cK$ be the solution of \mbox{$\hat\fP_\mu \hat K_\mu=\delta_0$}, where \mbox{$\hat\fP_\mu:=(1-[\mu]^2\Delta)$}. Verify that $\hat K_\mu-\hat K_\mu \ast \hat K_\eta = [\eta/\mu]^2~(\hat K_\eta - \hat K_\eta\ast \hat K_\mu)$.
\end{ex}

\begin{lem}[$\spadesuit$]\label{lem:support_V_Kmu}
Fix some $m\in\bN_+$ and $c\in\bR$. There exists $C\in(0,\infty)$ such that if for some $\mu\in(0,1/2]$ and $V\in\sS'(\bT\times\bM^m)$ it holds
\begin{equation}
 \supp\,V\subset\{(x,y_1,\ldots,y_m)\,|\,|x-y_1|\vee\ldots\vee|x-y_m|\leq c\,[\mu]\}
\end{equation}
then 
\begin{equation}
 \|\cX^a (\tilde K_\mu^{\otimes(1+m)}\ast V)\|_{\cV_\rt^m}\leq C\,[\mu]^{|a|}\, \|\tilde K_\mu^{\otimes(1+m)}\ast V\|_{\cV_\rt^m}.
\end{equation}
\end{lem}
\begin{rem}
Note that the lemma would be obvious if $\tilde K_\mu\in C(\bM)$ was compactly supported in a ball $\{x\in\bM\,|\,|x|<c\,[\mu]\}$ for some $c\in(0,\infty)$ independent of $\mu\in(0,1/2]$. Even though this is not exactly true, one can easily prove the above result by leveraging the fact that the kernel $\tilde K_\mu$ decays exponentially with the rate $1/[\mu]$.
\end{rem}
\begin{proof}[Proof $\mathrm{(}\spadesuit\mathrm{)}$.]
A simple rescaling reduces the proof to the case $\mu=1$. Let $v\in C^\infty(\bM)$ be such that $\supp\, v\subset\{x\in\bM\,|\,|x|<1\}$ and $v=1$ on $\{x\in\bM\,|\,|x|\leq 1/2\}$. For $\tau\in[1,\infty)$ let $ L_{\tau}(x):= \tilde K_1(x)\,v(x/\tau)$. It holds
\begin{equation}
\begin{aligned}
 \|\cX^a (\tilde K_1^{\otimes(1+m)}\ast V)\|_{\cV_\rt^m}
 &\leq
 \|\cX^a (L_1^{\otimes(1+m)}\ast V)\|_{\cV_\rt^m}
 +
 \int_1^\infty \|\cX^a \partial_\tau ( L_{\tau}^{\otimes(1+m)}\ast V)\|_{\cV_\rt^m}\,\rd\tau
 \\
 &\lesssim
 \|L_1^{\otimes(1+m)}\ast V\|_{\cV_\rt^m}
 +
 \int_1^\infty \tau^{|a|}\,\|\partial_\tau ( L_{\tau}^{\otimes(1+m)}\ast V)\|_{\cV_\rt^m}\,\rd\tau,
\end{aligned}
\end{equation}
where to get the last estimate we used the fact that $\supp\, L_{\tau}\subset \{x\in\bM\,|\,|x|<\tau\}$. Next, we observe that
\begin{equation}
  L_{\tau} = \tilde\fP_1 L_{\tau}\ast \tilde K_1,
 \qquad
 \partial_\tau L_{\tau} = \tilde\fP_1\partial_\tau L_{\tau}\ast \tilde K_1
\end{equation}
and $\|\tilde\fP_1 L_{\tau}\|_\cK \lesssim 1$ and $\|\tilde\fP_1\partial_\tau L_{\tau}\|_\cK\lesssim \tau^{-N}$ uniformly in $\tau\in[1,\infty)$ for any $N\in\bN_+$ because of the exponential decay of the kernel $\tilde K_1$. Consequently, we have
\begin{equation}
 \|\cX^a (\tilde K_1^{\otimes(1+m)}\ast V)\|_{\cV_\rt^m}
 \lesssim
 \|\tilde K_1^{\otimes(1+m)}\ast V\|_{\cV_\rt^m}
 +
 \int_1^\infty \tau^{|a|-N}\,\|\tilde K_1^{\otimes(1+m)}\ast V\|_{\cV_\rt^m}\,\rd\tau,
\end{equation}
which finishes the proof.
\end{proof}

\begin{ex}[$\spadesuit$]\label{ex:cumulants}
Prove that with the choice of the counterterms $(c^{(i)}_\kappa)_{i\in\{1,\ldots,i_\sharp\}}$ made in the proof of Theorem~\ref{thm:cumulants} for all lists of indices
$\vI=((i_1,m_1,s_1,r_1),\ldots,(i_n,m_n,s_n,r_n))$ the bound
\begin{equation}
 \|\tilde K^{\otimes(n+m)}_{\mu}\ast E^\vI_{\kappa,\mu}\|_{\cE^\MM}
 \lesssim 
 [\kappa]^{(\varepsilon-\sigma)r(\vI)}\,
 [\mu]^{\varrho_\varepsilon(\vI)-\sigma s(\vI)+d(n-1)}
\end{equation}
holds uniformly in $\kappa\in(0,1]$, $\mu\in(0,1/2]$. Hint: First generalize appropriately Lemma~\ref{lem:flow_E_general}. Then follow the proof of Theorem~\ref{thm:cumulants}. To prove the base case of the induction verify that
\begin{equation}
 \sup_{x\in\bT}\int_\bT|\bE((\tilde K_\mu\ast\partial_\kappa^{r_1}\xi_\kappa)(x)\,(\tilde K_\mu\ast\partial_\kappa^{r_2}\xi_\kappa)(y))|\,\rd y \lesssim [\kappa]^{(\varepsilon-\sigma)(r_1+r_2)} \,[\mu]^{-2\varepsilon}.
\end{equation}
\end{ex}

\section{Kolmogorov-type argument}\label{sec:probabilistic}

In this section we complete the passage from deterministic cumulant estimates
to almost sure bounds on the enhanced noise. The input is the family of
cumulant bounds proved in Theorem~\ref{thm:cumulants}, together with their
extension to \(\kappa\)-derivatives stated in Exercise~\ref{ex:cumulants}. These
bounds control the connected correlations of the effective force kernels at
each fixed scale. The goal is to turn this information into uniform random
bounds, simultaneously in the ultraviolet cutoff \(\kappa\) and in the scale
parameter \(\mu\).

The argument has three steps. First, we use the moment--cumulant expansion to
deduce moment bounds for the regularized effective force kernels from the
cumulant estimates. Since the cumulant bounds are formulated in the
\(\cE^\MM\)-norm, we also need a simple deterministic estimate converting this
norm into pointwise and \(L^n\)-type bounds after convolution with the kernels
\(\tilde K_\mu\). Second, we use Young's inequality and the smoothing properties
of \(\tilde K_\mu\) to upgrade the resulting \(L^n\)-bounds to bounds in the
kernel norm \(\cV^m\). Finally, we apply a Kolmogorov-type argument in the two
parameters \((\kappa,\mu)\). This produces a single random variable with finite
moments which controls the whole family of kernels uniformly over
\(\kappa\in(0,1]\) and \(\mu\in(0,1/2]\).

The main statement is Lemma~\ref{lem:probabilistic_bounds}. It gives the
stochastic estimate for a fixed effective force kernel from the corresponding
cumulant bounds. Exercise
\ref{ex:kolmogorov_norms} explains how to pass from the \(\cE^\MM\)-norm to a
pointwise estimate, Lemma~\ref{lem:expectation_sup} records the smoothing
estimate used to control suprema, and Lemma~\ref{lem:probabilistic_estimate}
contains the abstract two-parameter Kolmogorov argument. The final exercise then
uses the same estimates to obtain convergence as \(\kappa\searrow0\), thereby
identifying the limiting enhanced noise.

\begin{lem}\label{lem:probabilistic_bounds}
	Fix $n\in2\bN_+$ such that $d/n<\varepsilon$ and $i,m\in\bN_0$ such that $\varrho(i,m)\leq 0$. For $k\in\{1,\ldots,n\}$, $s,r\in\{0,1\}$ we define the list of indices $\vI\equiv\vI(k,s,r)=((i,m,s,r),\ldots,(i,m,s,r))$, $n(\vI)=k$. Assume that for all $k\in\{1,\ldots,n\}$ and $s,r\in\{0,1\}$ the bound
	\begin{equation}\label{eq:probabilistic_thm_assumption}
		\|\tilde K^{\otimes(k+km)}_\mu\ast E^{\vI}_{\kappa,\mu}\|_{\cE^{\MM(\vI)}}
		\lesssim
		[\kappa]^{k(\varepsilon-\sigma)r}
		[\mu]^{\varrho_\varepsilon(\vI)-\sigma s(\vI)+d(k-1)}
	\end{equation}
	holds uniformly in $\kappa\in(0,1]$ and $\mu\in(0,1/2]$. Then there exists a random variable $R\in[1,\infty]$ such that $\bE R^n<\infty$ and the following bound
	\begin{equation}
		\|
		K_\mu^{\otimes(1+m)}
		\ast F^{i,m}_{\kappa,\mu}\|_{\cV^m} \leq R\,[\mu]^{\varrho_{3\varepsilon}(i,m)}
	\end{equation}
	holds for all $\kappa\in(0,1]$ and $\mu\in(0,1/2]$.
\end{lem}
\begin{proof}
	Since \(F^{i,m}_{\kappa,\mu}=0\) for \(m>3i\), we assume that \(m\leq 3i\). By the moment--cumulant formula~\eqref{eq:expectation_cumulants}, the
	\(n\)-th moment is a sum over partitions of \(\{1,\ldots,n\}\). For a block
	\(B\) of cardinality \(k\), we use the assumption with the list
	\(\vI(k,s,r)\). Since
	\[
	\varrho_\varepsilon(\vI(k,s,r))=k\varrho_\varepsilon(i,m),
	\qquad
	s(\vI(k,s,r))=ks,
	\]
	Exercise~\ref{ex:kolmogorov_norms} gives a contribution bounded by
	\[
	[\kappa]^{k(\varepsilon-\sigma)r}
	[\mu]^{k(\varrho_\varepsilon(i,m)-\sigma s-dm)} .
	\]
	Multiplying over the blocks of the partition and summing over the finitely many
	partitions, we obtain
	\begin{equation}
		\bE\Bigl[
		\bigl(
		\tilde K^{\otimes(1+m)}_\mu\ast \tilde K^{\otimes(1+m)}_\mu \ast \partial_\mu^s\partial_\kappa^r F^{i,m}_{\kappa,\mu}(x,y_1,\ldots,y_m)
		\bigr)^n
		\Bigr]
		\lesssim [\kappa]^{n(\varepsilon-\sigma)r}[\mu]^{n(\varrho_\varepsilon(i,m)-\sigma s-dm)}
	\end{equation}
	uniformly in $\kappa\in(0,1]$ and $\mu\in(0,1/2]$ and $x,y_1,\ldots,y_m\in\bM$. Using Fubini's theorem and the argument from the proof of Lemma~\ref{lem:support_V_Kmu} one shows that
	\begin{equation}
		\bE\|\tilde K^{\otimes(1+m)}_\mu\ast \tilde K^{\otimes(1+m)}_\mu \ast \partial_\mu^s\partial_\kappa^r F^{i,m}_{\kappa,\mu}\|^n_{L^n(\bT\times\bM^m)}
		\lesssim [\kappa]^{n(\varepsilon-\sigma)r}[\mu]^{n(\varrho_\varepsilon(i,m)-\sigma s-dm)+d m}.
	\end{equation}
	Taking into account the fact that $K_\mu=\tilde K_\mu\ast\tilde K_\mu\ast\tilde K_\mu$ we conclude by Lemma~\ref{lem:expectation_sup} that
	\begin{equation}
		\bE\|K^{\otimes(1+m)}_\mu \ast \partial_\mu^s\partial_\kappa^r  F^{i,m}_{\kappa,\mu}\|^n_{L^\infty(\bT\times\bM^m)}
		\lesssim [\kappa]^{n(\varepsilon-\sigma)r}[\mu]^{n(\varrho_\varepsilon(i,m)-\sigma s-dm)-d}.
	\end{equation}
	Employing again the strategy from the proof of Lemma~\ref{lem:support_V_Kmu} we obtain
	\begin{equation}
		\bE\|K^{\otimes(1+m)}_\mu \ast \partial_\mu^s\partial_\kappa^r  F^{i,m}_{\kappa,\mu}\|^n_{\cV^m}
		\lesssim [\kappa]^{n(\varepsilon-\sigma)r}[\mu]^{n(\varrho_\varepsilon(i,m)-\sigma s)-d}.
	\end{equation}
	Next, we show that
	\begin{equation}
		\bE\|\partial_\mu^s\partial_\kappa^r(K^{\otimes(1+m)}_\mu \ast F^{i,m}_{\kappa,\mu})\|^n_{\cV^m}
		\lesssim [\kappa]^{n(\varepsilon-\sigma)r}[\mu]^{n(\varrho_\varepsilon(i,m)-\sigma s)-d}.
	\end{equation}
	For \(s=1\), the derivative
	\[
	\partial_\mu\bigl(K_\mu^{\otimes(1+m)}\ast F^{i,m}_{\kappa,\mu}\bigr)
	\]
	is a finite sum of terms in which either \(\partial_\mu\) falls on
	\(F^{i,m}_{\kappa,\mu}\), or on one of the \(K_\mu\)-factors. The first type
	has already been estimated. The second type is bounded in the same way using
	Lemma~\ref{lem:kernel_simple_fact}~(B), which gives the additional factor
	\([\mu]^{-\sigma}\).
	Finally, we apply Lemma~\ref{lem:probabilistic_estimate} to the family
	\[
	\zeta_{\kappa,2\mu}
	:=
	K_\mu^{\otimes(1+m)}\ast F^{i,m}_{\kappa,\mu},
	\qquad \mu\in(0,1/2],
	\]
	with
	\[
	\rho:=\varrho_\varepsilon(i,m)-2\varepsilon<0.
	\]
	The preceding estimate gives, after taking \(n\)-th roots,
	\[
	\mathbb E
	\Bigl[
	\left\|
	\partial_\mu^s\partial_\kappa^r
	\left(
	K_\mu^{\otimes(1+m)}
	\ast F^{i,m}_{\kappa,\mu}
	\right)
	\right\|_{\cV^m}^n
	\Bigr]^{1/n}
	\lesssim
	[\kappa]^{(\varepsilon-\sigma)r}
	[\mu]^{\varrho_\varepsilon(i,m)-\sigma s-d/n}.
	\]
	Since \(d/n<\varepsilon\), this is bounded by
	\[
	[\kappa]^{(\varepsilon-\sigma)r}
	[\mu]^{\rho-\sigma s+\varepsilon s},
	\]
	for \(s,r\in\{0,1\}\). Thus the assumptions of
	Lemma~\ref{lem:probabilistic_estimate} are satisfied. We obtain a random
	variable \(R\) with \(\mathbb E R^n<\infty\) such that
	\[
	\left\|
	K_\mu^{\otimes(1+m)}
	\ast F^{i,m}_{\kappa,\mu}
	\right\|_{\cV^m}
	\leq
	R\,[\mu]^\rho .
	\]
	Recall that \(m\leq 3i\). Hence,
	\[
	\rho-\varrho_{3\varepsilon}(i,m)
	=
	2\varepsilon(3i-m)
	\geq0.
	\]
	Since \([\mu]\leq1\), the desired bound with exponent
	\(\varrho_{3\varepsilon}(i,m)\) follows.
\end{proof}

\begin{ex}[$\spadesuit$]\label{ex:kolmogorov_norms}
	Let $n\in\bN_+$, $\MM=(m_1,\ldots,m_n)\in\bN_0^n$ and $m=m_1+\ldots+m_n$. Show that it holds
	\begin{equation}
		\|\tilde K^{\otimes(n+m)}_\mu\ast V\|_{L^\infty(\bM^{n+m})}
		\lesssim 
		[\mu]^{-d(n+m-1)}\,\|V\|_{\cE^\MM}
	\end{equation}
	uniformly in $\mu\in(0,1]$ and $V\in\cE^\MM$. Hint: Use Lemma~\ref{lem:kernel_simple_fact}~(C) and~(D) with $p=\infty$.
\end{ex}

\begin{lem}[$\spadesuit$]\label{lem:expectation_sup}
	Let $n\in2\bN_+$, $m\in\bN_0$. There exists a constant $C>0$ such that for all random fields $\zeta\in L^\infty(\bT\times\bM^m)$ and $\mu\in(0,1]$ it holds
	\begin{equation}
		\bE
		\|\tilde K_\mu^{\otimes(1+m)} \ast \zeta\|^n_{L^\infty(\bT\times\bM^m)}
		\leq C\, [\mu]^{-d(1+m)}\,
		\bE\|\zeta\|^n_{L^n(\bT\times\bM^m)}.
	\end{equation}
\end{lem}
\begin{proof}
	Note that 
	\begin{equation}
		\tilde K_\mu^{\otimes(1+m)} \ast \zeta
		=(\fT \tilde K_\mu\otimes\tilde K_\mu^{\otimes m})\star \zeta
	\end{equation}
	where $\star$ is the convolution in $\bT\times\bM^m$ and $\fT \tilde K_\mu$ is the periodization of $\tilde K_\mu$ (see Definition~\ref{dfn:periodization}). Using Young's inequality for convolutions, we obtain
	\begin{equation}
		\bE\|\tilde K_\mu^{\otimes(1+m)}\ast \zeta\|^n_{L^\infty(\bT\times\bM^m)}
		\leq
		\|\fT \tilde K_\mu\|_{L^{n/(n-1)}(\bT)}^n\,
		\|\tilde K_\mu\|_{L^{n/(n-1)}(\bM)}^{mn}\,
		\bE\|\zeta\|^n_{L^n(\bT\times\bM^m)}.
	\end{equation} 
	The lemma follows now from Lemma~\ref{lem:kernel_simple_fact}~(C),~(D).
\end{proof}

\begin{lem}\label{lem:probabilistic_estimate}
	Fix \( n \in \mathbb{N}_+ \), \( m \in \mathbb{N}_0 \) and \(\varepsilon>0\). There exists a constant \( c > 0 \) such that the following holds. Let \( \zeta : (0,1]^2 \to \mathcal{V}^m \) be a random function, differentiable in both variables. Suppose that there exist constants \( C > 0 \) and \( \rho \le 0 \) such that for all \( r,s \in \{0,1\} \) and all \( \kappa, \mu \in (0,1] \),
	\[
	\mathbb{E}\Bigl[\bigl\| \partial_{\mu}^s \partial_{\kappa}^r \zeta_{\kappa,\mu} \bigr\|_{\mathcal{V}^m}^n\Bigr]^{1/n}
	\le C \, [\kappa]^{-\sigma r + \varepsilon r} \, [\mu]^{\rho - \sigma s + \varepsilon s}.
	\]
	Then,
	\[
	\mathbb{E}\Biggl[ \sup_{\kappa,\mu \in (0,1]} 
	[\mu]^{-n\rho} \, \|\zeta_{\kappa,\mu}\|_{\mathcal{V}^m}^n \Biggr]^{1/n}
	\le c \, C.
	\]
\end{lem}
\begin{proof}
	We write
	\begin{equation}
		\zeta_{\kappa,\mu}=
		\zeta_{1,1}
		-
		\int_\mu^1 \partial_\eta\zeta_{1,\eta}\,\rd\eta
		-
		\int_\kappa^1 \partial_\nu\zeta_{\nu,1}\,\rd\nu
		+
		\int_\mu^1 \int_\kappa^1 \partial_\eta\partial_\nu \zeta_{\nu,\eta} \,\rd\nu\rd\eta.
	\end{equation}
	Using the fact that $\rho\leq0$, we obtain
	\begin{multline}
		[\mu]^{-\rho}\,\|\zeta_{\kappa,\mu}\|_{\cV^m}
		\leq
		\|\zeta_{1,1}\|_{\cV^m}
		+
		\int_\mu^1 [\eta]^{-\rho}\,\|\partial_\eta\zeta_{1,\eta}\|_{\cV^m}\,\rd\eta
		+
		\int_\kappa^1 \|\partial_\nu\zeta_{\nu,1}\|_{\cV^m}\,\rd\nu
		\\
		+
		\int_\mu^1 \int_\kappa^1 [\eta]^{-\rho}\,\|\partial_\eta\partial_\nu \zeta_{\nu,\eta}\|_{\cV^m} \,\rd\nu\rd\eta.
	\end{multline}
	Applying Minkowski's inequality to the $L^n$ norm over the probability space and using the assumed bound for $\mathbb{E}\bigl[\bigl\| \partial_{\mu}^s \partial_{\kappa}^r \zeta_{\kappa,\mu} \bigr\|_{\mathcal{V}^m}^n\bigr]^{1/n}$, we obtain
	\begin{multline}
		\mathbb{E}\Biggl[ \sup_{\kappa,\mu \in (0,1]} 
		[\mu]^{-n\rho} \, \|\zeta_{\kappa,\mu}\|_{\mathcal{V}^m}^n \Biggr]^{1/n}
		\leq C
		+
		\int_0^1 C\,[\eta]^{-\sigma+\varepsilon}\,\rd\eta
		+
		\int_0^1 C\,[\nu]^{-\sigma+\varepsilon}\,\rd\nu
		\\
		+
		\int_0^1 \int_0^1 C\,[\nu]^{-\sigma+\varepsilon}
		[\eta]^{-\sigma+\varepsilon}
		\,\rd\nu\rd\eta.
	\end{multline}
	This proves the desired bound.
\end{proof}

\begin{ex}[$\spadesuit$]\label{ex:convergence}
	Under the assumptions of Lemma~\ref{lem:probabilistic_bounds}, prove that
	there exists a random family
	\[
	(F^{i,m}_{0,\mu})_{\mu\in(0,1/2]}
	\subset \sS'(\bM^{1+m})
	\]
	such that
	\begin{equation}
		\lim_{\kappa\searrow0}
		\sup_{\mu\in(0,1/2]}
		[\mu]^{-\varrho_{3\varepsilon}(i,m)}
		\|K_\mu^{\otimes(1+m)}
		\ast (F^{i,m}_{0,\mu}-F^{i,m}_{\kappa,\mu})\|_{\cV^m}
		=0
	\end{equation}
	almost surely. Hint: Use the argument from the proof of Lemma~\ref{lem:probabilistic_bounds}.
\end{ex}

\begin{rem}[$\spadesuit$]
	Using the result stated in the above exercise and the fact that, for every
	\(\kappa\in(0,1]\), the families of kernels
	\[
	\Bigl(
	(0,1/2]\ni\mu\longmapsto F^{i,m}_{\kappa,\mu}\in\cV^m
	\Bigr)_{
		i\in\{0,\ldots,i_\flat\},\,
		m\in\bN_0
	}
	\]
	satisfy the flow equation~\eqref{eq:flow_deterministic_i_m}, one verifies that
	the limiting families
	\[
	\Bigl(
	(0,1/2]\ni\mu\longmapsto F^{i,m}_{0,\mu}\in\sS'(\bM^{1+m})
	\Bigr)_{
		i\in\{0,\ldots,i_\flat\},\,
		m\in\bN_0
	}
	\]
	are almost surely continuously differentiable in \(\mu\) and satisfy the same
	flow equation.
\end{rem}

\section{Relation to original equation and convergence}\label{sec:app}

The purpose of this section is to collect several auxiliary results needed
to pass from the scale-dependent fixed point constructed in Section~\ref{sec:effective_equation} to a solution of the original equation, and then to identify its
limit as $\kappa\searrow0$. We first show that, in the presence of the
regularization, any bounded fixed point of the scale-dependent effective
equation gives rise to a classical solution of the original regularized equation. We then
verify the required boundedness property for the fixed point of the
regularized equation. Finally, we prove a stability statement for the fixed
points with respect to the regularization parameter $\kappa$, which yields
convergence of the corresponding solutions in $\sC^\alpha(\bM)$ and completes
the proof of Theorem~\ref{thm:main}.

The lemma below specifies the conditions under which a fixed point of the map $\fQ$, introduced in Lemma~\ref{lem:lift}, corresponds to a solution of $\varPhi_\kappa=G\ast F_\kappa[\varPhi_\kappa]$ for $\kappa\in(0,1]$.
\begin{lem}[$\spadesuit$]\label{lem:classical}
	Fix $\kappa\in(0,1]$. Assume that the family of functionals $(F_{\kappa,\mu})_{\mu\in[0,1]}$ of polynomial type depends continuously on $\mu\in[0,1]$ and is piecewise continuously differentiable in $\mu$ for $\mu\in(0,1]$. Moreover, suppose that $F_{\kappa,0}=F_\kappa$ and for some $R\in[1,\infty)$ and $m_\flat\in\bN_0$ it holds
	\begin{equation}\label{eq:F_uv_bounds}
		\| \rD^k (\partial_\mu^s F_{\kappa,\mu})[\varphi]\cdot \psi^{\otimes k}\,\|
		\leq 
		R\,\|\psi\|^k\,(1+\|\varphi\|)^{m_\flat}
	\end{equation}
	for all $k\in\bN_0$, $s\in\{0,1\}$, $\mu\in[0,1]$, $\varphi,\psi\in C(\bT)$. Let the family of functionals $(H_{\kappa,\mu})_{\mu\in(0,1]}$ be defined by~\eqref{eq:def_H}. If a continuous and bounded function
	\begin{equation}\label{eq:classical_fixed_point}
		(0,1]\ni\mu\mapsto(\tilde\varPhi_{\kappa,\mu},\tilde\zeta_{\kappa,\mu})\in C(\bT)\times C(\bT)
	\end{equation}
	is the fixed point of the map $\fQ_\kappa$ defined by~\eqref{eq:map_Q} in terms of 
	$\tilde F_{\mu}[\varphi]:= K_\mu\ast F_{\kappa,\mu}[K_\mu\ast\varphi]$ and $\tilde H_{\mu}[\varphi]:=
	K_\mu\ast H_{\kappa,\mu}[K_\mu\ast\varphi]$, then the limit $\lim_{\mu\searrow0}K_\mu\ast\tilde\varPhi_{\kappa,\mu}=:\varPhi_\kappa$ exists in $C(\bT)$ and satisfies the equation $\varPhi_\kappa=G\ast F_\kappa[\varPhi_\kappa]$.
\end{lem}
\begin{proof}[Proof sketch]
	Noting that $\fP_\mu K_\mu=\delta_0$ and $K_{\eta,\mu}\ast K_\mu=K_\eta$ we show that
	\begin{equation}
		(0,1]\ni\mu\mapsto(\varPhi_{\kappa,\mu},\zeta_{\kappa,\mu}):=(K_\mu\ast\tilde \varPhi_{\kappa,\mu},\fP_\mu\tilde\zeta_{\kappa,\mu})\in C(\bT)\times \sS'(\bM)
	\end{equation}
	satisfies the system of equations
	\begin{equation}\label{eq:effective_proof}
		\begin{cases}\displaystyle
			\varPhi_{\kappa,\mu} = -\int_\mu^1 \dot G_\eta \ast (F_{\kappa,\eta}[\varPhi_{\kappa,\eta}]+\zeta_{\kappa,\eta})\,\rd \eta
			\\
			\displaystyle
			\zeta_{\kappa,\mu} = -\int_0^\mu (H_{\kappa,\eta}[\varPhi_{\kappa,\eta}] + \rD F_{\kappa,\eta}[\varPhi_{\kappa,\eta}] \cdot (\dot G_\eta\ast\zeta_{\kappa,\eta}))\,\rd \eta.
		\end{cases}
	\end{equation}
	Using the assumptions about the effective force and the fixed point we show that the integrands above are continuous and bounded. Hence, we conclude that
	\begin{equation}
		(0,1]\ni\mu\mapsto(\varPhi_{\kappa,\mu},\zeta_{\kappa,\mu})\in C(\bT)\times C(\bT)
	\end{equation}
	and the above function is bounded, differentiable and has a limit at $\mu=0$. Next, we show that $\partial_\eta (F_{\kappa,\eta}[\varPhi_{\kappa,\eta}]+\zeta_{\kappa,\eta})=0$ by following the argument from the beginning of Section~\ref{sec:effective_equation}. As a~result, we obtain $F_{\kappa,\eta}[\varPhi_{\kappa,\eta}]+\zeta_{\kappa,\eta}=F_\kappa[\varPhi_\kappa]$. Consequently, the first of the equations~\eqref{eq:effective_proof} implies that $\varPhi_\kappa$ satisfies the equation $\varPhi_\kappa=G\ast F_\kappa[\varPhi_\kappa]$.
\end{proof}

In order to use the above lemma one has to verify that for $\kappa\in(0,1]$ the unique fixed point $(\tilde\varPhi_{\kappa,\Cdot},\tilde\zeta_{\kappa,\Cdot})$ of the map $\fQ_\kappa\,:\,\sB_R\to\sB_R$ constructed in Lemma~\ref{lem:lift} is such that the map~\eqref{eq:classical_fixed_point} is bounded. The proof of this fact is the subject of the following lemma.

\begin{lem}[$\spadesuit$]\label{lem:solution_bounded}
	Fix $\kappa\in(0,1]$ and assume that
	\[
	\|\xi_\kappa\|\lesssim 1,
	\qquad
	|c^{(i)}_\kappa|\lesssim 1,
	\qquad i\in\{1,\ldots,i_\sharp\}.
	\]
	Then the coefficients generated by the flow equation
	\eqref{eq:flow_deterministic_i_m} satisfy
	\begin{equation}
		\label{eq:F_im_uniform_bound}
		\|F^{i,m}_{\kappa,\mu}\|_{\cV^m}
		\lesssim
		1\wedge [\mu]^{\sigma(m-3)/2},
	\end{equation}
	uniformly in $\mu\in(0,1]$, for every
	$i\in\{0,\ldots,i_\flat\}$ and
	$m\in\{0,\ldots,3i\}$. Consequently, there exists $R\in[1,\infty)$ such that, for every
	$\delta\in[-\sigma/2,0]$ and every $\varphi\in C(\bT)$,
	\begin{equation}
		\label{eq:F_gamma}
		[\mu]^{-3\delta}
		\|F_{\kappa,\mu}[\varphi]\|
		\leq
		R
		\bigl(
		1+[\mu]^{-\delta}\|\varphi\|
		)^{3i_\flat},
	\end{equation}
	where $F_{\kappa,\Cdot}$ denotes the effective force defined in
	\eqref{eq:stopped_eff_force}.
	
	Let $(\tilde\varPhi_{\kappa,\Cdot},\tilde\zeta_{\kappa,\Cdot})\in\sB_R$
	be the unique fixed point of $\fQ_\kappa$ constructed in
	Lemma~\ref{lem:lift}, applied with
	$\alpha\in(-\sigma/2,\sigma-d/2)$. Then the map
	\[
	\tilde\varPhi_{\kappa,\Cdot}:(0,1]\to C(\bT)
	\]
	is bounded.
\end{lem}

\begin{proof}
	We first prove \eqref{eq:F_im_uniform_bound}. The argument proceeds by
	induction along the hierarchy generated by the flow equation
	\eqref{eq:flow_deterministic_i_m}. The initial terms are bounded by the
	assumptions on $\xi_\kappa$ and on the constants $c^{(i)}_\kappa$. Assume that the desired estimate has already been obtained for all
	lower-order terms appearing on the right-hand side of
	\eqref{eq:flow_deterministic_i_m}. Using the bound from
	Exercise~\ref{ex:fB1_bound}, we obtain
	\begin{align}
		\|\partial_\mu F^{i,m}_{\kappa,\mu}\|_{\cV^m}
		&\lesssim
		\sum_{k=0}^m
		\bigl(
		1\wedge [\mu]^{\sigma(k+1-3)/2}
		)
		\bigl(
		1\wedge [\mu]^{\sigma(m-k-3)/2}
		)
		\nonumber \\
		&\lesssim
		1\wedge [\mu]^{\sigma(m-3)/2-\sigma}.
	\end{align}
	Integrating this estimate in the scale variable gives
	\begin{align}
		\|F^{i,m}_{\kappa,\mu}\|_{\cV^m}
		&\leq
		\|F^{i,m}_{\kappa,0}\|_{\cV^m}
		+
		\int_0^\mu
		\|\partial_\eta F^{i,m}_{\kappa,\eta}\|_{\cV^m}
		\,\rd\eta
		\nonumber \\
		&\lesssim
		1\wedge [\mu]^{\sigma(m-3)/2}.
	\end{align}
	This proves \eqref{eq:F_im_uniform_bound} uniformly in
	$\mu\in(0,1]$.
	
	The estimate \eqref{eq:F_gamma} follows by inserting
	\eqref{eq:F_im_uniform_bound} into the definition
	\eqref{eq:stopped_eff_force} of the effective force. Since only
	finitely many coefficients occur, with
	$i\leq i_\flat$ and $m\leq 3i$, all constants may be absorbed into a
	single constant $R\in[1,\infty)$.
	
	It remains to prove boundedness of the fixed point. Since
	$(\tilde\varPhi_{\kappa,\Cdot},\tilde\zeta_{\kappa,\Cdot})\in\sB_R$, we have
	\begin{equation}
		\|\tilde\varPhi_{\kappa,\mu}\|
		\lesssim [\mu]^\alpha,
		\qquad
		\|\tilde\zeta_{\kappa,\mu}\|
		\lesssim [\mu]^\beta,
	\end{equation}
	with $\alpha\in(-\sigma/2,\sigma-d/2)$ and $\beta>0$. Using the
	fixed-point equation, together with \eqref{eq:F_gamma} applied with
	$\delta=\alpha$, yields
	\begin{equation}
		\|\tilde\varPhi_{\kappa,\mu}\|
		\lesssim
		1+
		\int_\mu^1
		[\eta]^{3\alpha}
		\,\rd\eta
		\lesssim
		1\vee [\mu]^{\sigma+3\alpha}=1\vee[\mu]^{\alpha_1},
	\end{equation}
	where
	\[
	\alpha_0:=\alpha,
	\qquad
	\alpha_1:=\sigma+3\alpha_0.
	\]
	Then $\alpha_1>\alpha_0$. If $\alpha_1\geq0$, the preceding estimate
	already gives a uniform bound on
	$\|\tilde\varPhi_{\kappa,\mu}\|$ for $\mu\in(0,1]$. If $\alpha_1<0$, we repeat the same argument, now using
	\eqref{eq:F_gamma} with $\delta=\alpha_1$. This gives
	\[
	\|\tilde\varPhi_{\kappa,\mu}\|
	\lesssim
	1\vee [\mu]^{\alpha_2},
	\qquad
	\alpha_2:=\sigma+3\alpha_1.
	\]
	Iterating this procedure produces a sequence
	\[
	\alpha_{n+1}:=\sigma+3\alpha_n,
	\qquad n\geq0,
	\]
	as long as $\alpha_n<0$. Since each step improves the exponent, after
	finitely many iterations one obtains an exponent $\alpha_n\geq0$.
	Hence
	\[
	\sup_{\mu\in(0,1]}
	\|\tilde\varPhi_{\kappa,\mu}\|_{C(\bT)}
	<\infty,
	\]
	which finishes the proof.
\end{proof}

To conclude the proof of Theorem~\ref{thm:main}, we use the following lemma, which establishes the existence of the limit as $\kappa\searrow0$.

\begin{lem}[$\spadesuit$]\label{lem:fixed_convergence}
	For $\kappa\in[0,1]$ let 
	\begin{equation}\label{eq:ex_convergence}
		(\tilde F_\mu)_{\mu\in(0,1]}\equiv (\tilde F_{\kappa,\mu})_{\mu\in(0,1]},\qquad (\tilde H_\mu)_{\mu\in(0,1]}\equiv (\tilde H_{\kappa,\mu})_{\mu\in(0,1]}
	\end{equation}
	be families of functionals such that:
	\begin{itemize}
		\item[(1)] for all $\kappa\in(0,1]$ the assumptions of Lemma~\ref{lem:classical} are satisfied,
		\item[(2)] for all $\kappa\in[0,1]$ the assumptions of Lemma~\ref{lem:lift} are satisfied,
		\item[(3)] for all $\kappa\in[0,1]$ there exists $r_\kappa\in\bR$ such that
		\begin{equation}
			\begin{aligned}
				[\mu]^{\sigma-\alpha}\,\|\rD^k (\tilde F_{0,\mu}&-\tilde F_{\kappa,\mu})[\varphi]\cdot \psi^{\otimes k}\|
				\\&\leq 
				r_\kappa\,(|\lambda|^{1/3}[\mu]^{-\alpha}\,\|\psi\|)^k\,(1/2+|\lambda|^{1/3}[\mu]^{-\alpha}\,\|\varphi\|)^{m_\flat},
			\end{aligned} 
		\end{equation}
		and
		\begin{equation}
			[\mu]^{\sigma-\beta}\,\|(\tilde H_{0,\mu}-\tilde H_{\kappa,\mu})[\varphi]\|
			\leq r_\kappa\, (1/2+|\lambda|^{1/3}[\mu]^{-\alpha}\,\|\varphi\|)^{m_\flat}
		\end{equation}
		for all $k\in\{0,1\}$, $\kappa,\mu\in(0,1]$, $\varphi,\psi\in C(\bT)$, $\lambda\in[-1,1]$ and $\lim_{\kappa\searrow0}r_\kappa=0$. 
	\end{itemize}
	
	For $\kappa\in[0,1]$, let $\fQ_\kappa$ be the map defined by
	\eqref{eq:map_Q} in terms of the functionals in
	\eqref{eq:ex_convergence}. By Assumption~(2) and
	Lemma~\ref{lem:lift}, for every
	$\lambda\in[-\lambda_\star,\lambda_\star]$ and every $\kappa\in[0,1]$,
	the map
	\[
	\fQ_\kappa:\sB_R\to\sB_R
	\]
	is well defined and is a contraction with Lipschitz constant strictly
	smaller than $1/2$. Denote its unique fixed point by
	\[
	(\tilde\varPhi_{\kappa,\Cdot},\tilde\zeta_{\kappa,\Cdot})\in\sB_R .
	\]
	For $\kappa\in(0,1]$, let
	\[
	\varPhi_\kappa
	:=
	\lim_{\mu\searrow0}\tilde\varPhi_{\kappa,\mu}
	\in C_\rb(\bM),
	\]
	whose existence follows from Assumption~(1),
	Lemma~\ref{lem:classical} and Lemma~\ref{lem:solution_bounded}.
	Then $\varPhi_\kappa$ solves
	\[
	\varPhi_\kappa=G\ast F_\kappa[\varPhi_\kappa].
	\]
	Moreover, the following conclusions hold.
	\begin{itemize}
		\item[(i)] The fixed points converge in $\sB_R$, namely
		\[
		\lim_{\kappa\searrow0}
		(\tilde\varPhi_{\kappa,\Cdot},\tilde\zeta_{\kappa,\Cdot})
		=
		(\tilde\varPhi_{0,\Cdot},\tilde\zeta_{0,\Cdot}) .
		\]
		
		\item[(ii)] There exists $\varPhi_0\in\sS'(\bM)$ such that
		\begin{equation}
			\lim_{\kappa\searrow0}
			\|\varPhi_\kappa-\varPhi_0\|_{\sC^\alpha(\bM)}
			=
			\lim_{\kappa\searrow0}
			\sup_{\mu\in(0,1]}
			[\mu]^{-\alpha}
			\,
			\|K_\mu\ast(\varPhi_\kappa-\varPhi_0)\|
			=
			0 .
		\end{equation}
	\end{itemize}
\end{lem}

\begin{proof}[Proof sketch]
	We first prove that for every
	$(\tilde\varPhi_{\Cdot},\tilde\zeta_{\Cdot})\in\sB_R$,
	\[
	\lim_{\kappa\searrow0}
	\fQ_\kappa[\tilde\varPhi_{\Cdot},\tilde\zeta_{\Cdot}]
	=
	\fQ_0[\tilde\varPhi_{\Cdot},\tilde\zeta_{\Cdot}]
	\]
	in $\sB_R$. Let
	$(\tilde\varPhi_{\Cdot},\tilde\zeta_{\Cdot})\in\sB_R$ be fixed.
	The definition of $\fQ_\kappa$ in \eqref{eq:map_Q} involves only the
	functionals $\tilde F_{\kappa,\mu}$ and $\tilde H_{\kappa,\mu}$, together
	with scale integrations and convolution operators which are independent
	of $\kappa$. Therefore the difference
	\[
	\fQ_\kappa[\tilde\varPhi_{\Cdot},\tilde\zeta_{\Cdot}]
	-
	\fQ_0[\tilde\varPhi_{\Cdot},\tilde\zeta_{\Cdot}]
	\]
	is controlled by the differences
	$\tilde F_{\kappa,\mu}-\tilde F_{0,\mu}$ and
	$\tilde H_{\kappa,\mu}-\tilde H_{0,\mu}$ evaluated along
	$\tilde\varPhi_{\Cdot}$. Since
	$(\tilde\varPhi_{\Cdot},\tilde\zeta_{\Cdot})\in\sB_R$, the norms of
	$\tilde\varPhi_{\mu}$ and $\tilde\zeta_{\mu}$ have the scale bounds
	prescribed in the definition of $\sB_R$. Inserting these bounds into estimates from Assumption~(3) gives
	\begin{equation}\label{eq:bound_difference_r}
	\bigl\|
	\fQ_\kappa[\tilde\varPhi_{\Cdot},\tilde\zeta_{\Cdot}]
	-
	\fQ_0[\tilde\varPhi_{\Cdot},\tilde\zeta_{\Cdot}]
	\bigr\|_{\sB_R}
	\lesssim r_\kappa
	\end{equation}
	and proves the claim.
	
	We next show convergence of the fixed points. Since
	$(\tilde\varPhi_{\kappa,\Cdot},\tilde\zeta_{\kappa,\Cdot})$ and
	$(\tilde\varPhi_{0,\Cdot},\tilde\zeta_{0,\Cdot})$ are fixed points of
	$\fQ_\kappa$ and $\fQ_0$, respectively, we have
	\begin{align}
		&
		\bigl\|
		(\tilde\varPhi_{\kappa,\Cdot},\tilde\zeta_{\kappa,\Cdot})
		-
		(\tilde\varPhi_{0,\Cdot},\tilde\zeta_{0,\Cdot})
		\bigr\|_{\sB_R}
		\nonumber \\
		&\qquad\leq
		\bigl\|
		\fQ_\kappa[
		\tilde\varPhi_{\kappa,\Cdot},\tilde\zeta_{\kappa,\Cdot}
		]
		-
		\fQ_\kappa[
		\tilde\varPhi_{0,\Cdot},\tilde\zeta_{0,\Cdot}
		]
		\bigr\|_{\sB_R}
		\nonumber \\
		&\qquad\quad+
		\bigl\|
		\fQ_\kappa[
		\tilde\varPhi_{0,\Cdot},\tilde\zeta_{0,\Cdot}
		]
		-
		\fQ_0[
		\tilde\varPhi_{0,\Cdot},\tilde\zeta_{0,\Cdot}
		]
		\bigr\|_{\sB_R}.
	\end{align}
	The first term is bounded by one half of the left-hand side, because
	$\fQ_\kappa$ is a contraction with Lipschitz constant smaller than
	$1/2$. Hence
	\[
	\bigl\|
	(\tilde\varPhi_{\kappa,\Cdot},\tilde\zeta_{\kappa,\Cdot})
	-
	(\tilde\varPhi_{0,\Cdot},\tilde\zeta_{0,\Cdot})
	\bigr\|_{\sB_R}
	\lesssim
	\bigl\|
	\fQ_\kappa[
	\tilde\varPhi_{0,\Cdot},\tilde\zeta_{0,\Cdot}
	]
	-
	\fQ_0[
	\tilde\varPhi_{0,\Cdot},\tilde\zeta_{0,\Cdot}
	]
	\bigr\|_{\sB_R}.
	\]
	The right-hand side tends to zero by~\eqref{eq:bound_difference_r}, proving item~(i).

	It remains to construct the limiting distribution $\varPhi_0$ and to prove convergence in $\sC^\alpha(\bM)$. A simple computation, together with arguments similar to those in Section~\ref{sec:effective_equation}, yields
	\begin{equation}
		\label{eq:regularized_solution_identity}
		K_\mu\ast\varPhi_\kappa
		=
		K_\mu\ast K_\mu\ast\tilde\varPhi_{\kappa,\mu}
		+
		(G-G_\mu)
		\ast
		\bigl(
		\tilde F_{\kappa,\mu}[\tilde\varPhi_{\kappa,\mu}]
		+
		\tilde\zeta_{\kappa,\mu}
		\bigr)
	\end{equation}
	for all $\kappa,\mu\in(0,1]$. We define $\varPhi_0$ by postulating that
	\eqref{eq:regularized_solution_identity} holds for $\kappa=0$, replacing
	$(\tilde\varPhi_{\kappa,\Cdot},\tilde\zeta_{\kappa,\Cdot})$ by the limiting
	fixed point
	$(\tilde\varPhi_{0,\Cdot},\tilde\zeta_{0,\Cdot})$. That is, for each
	$\mu\in(0,1]$, we set
	\[
	K_\mu\ast\varPhi_0
	:=
	K_\mu\ast K_\mu\ast\tilde\varPhi_{0,\mu}
	+
	(G-G_\mu)
	\ast
	\bigl(
	\tilde F_{0,\mu}[\tilde\varPhi_{0,\mu}]
	+
	\tilde\zeta_{0,\mu}
	\bigr).
	\]
	The compatibility of this family in $\mu$ follows from the fixed-point
	equation for $\fQ_0$, hence it defines an element
	$\varPhi_0\in\sS'(\bM)$.
	
	Subtracting the identities for $K_\mu\ast\varPhi_\kappa$ and
	$K_\mu\ast\varPhi_0$ gives two types of terms. The first comes from
	\[
	K_\mu\ast K_\mu
	\ast
	(
	\tilde\varPhi_{\kappa,\mu}
	-
	\tilde\varPhi_{0,\mu}
	),
	\]
	which is controlled directly by the convergence of the fixed points in
	$\sB_R$. The second comes from the remainder term
	\[
	(G-G_\mu)\ast
	\Bigl[
	\tilde F_{\kappa,\mu}[\tilde\varPhi_{\kappa,\mu}]
	-
	\tilde F_{0,\mu}[\tilde\varPhi_{0,\mu}]
	+
	\tilde\zeta_{\kappa,\mu}
	-
	\tilde\zeta_{0,\mu}
	\Bigr].
	\]
	Using
	\[
	\|G-G_\mu\|\lesssim [\mu]^\sigma,
	\]
	Assumptions~(3),
	and the convergence of the fixed points from item~(i), this term is
	also bounded by a quantity tending to zero after multiplication by
	$[\mu]^{-\alpha}$ and taking the supremum over $\mu\in(0,1]$.
	Therefore
	\[
	\lim_{\kappa\searrow0}
	\sup_{\mu\in(0,1]}
	[\mu]^{-\alpha}
	\,
	\|K_\mu\ast(\varPhi_\kappa-\varPhi_0)\|
	=0.
	\]
	This is precisely the convergence
	$\varPhi_\kappa\to\varPhi_0$ in $\sC^\alpha(\bM)$ and proves~(ii).
\end{proof}

\section{Symbolic index}\label{sec:symbolic_index}

{\small
	\begin{center}
		\renewcommand{\arraystretch}{1.1}
		\begin{longtable}{@{}p{0.24\textwidth}p{0.56\textwidth}p{0.12\textwidth}@{}}
			\toprule
			Symbol & Meaning & Page\\
			\midrule
			\endfirsthead
			\toprule
			Symbol & Meaning & Page\\
			\midrule
			\endhead
			\bottomrule
			\endfoot
			\bottomrule
			\endlastfoot
			
			$d$ & Spatial dimension & \pageref{eq:singular_elliptic}\\
			
			$\sigma$ & Order of the fractional elliptic operator & \pageref{eq:singular_elliptic}\\
			
			$\xi$ & Periodisation of the white noise on $\bM$& \pageref{eq:singular_elliptic} \\
			
			$[\kappa]=\kappa^{1/\sigma}$ & Ultraviolet regularization length associated with $\kappa$ & \pageref{eq:mollifier}\\
			
			$\vartheta,\vartheta_\kappa$ & Mollifier and its rescaling at scale $\kappa$ & \pageref{eq:mollifier}\\
			
			$\xi_\kappa$ & Regularized noise $\vartheta_\kappa\ast\xi$ & \pageref{eq:mollifier}\\
			
			$\lambda$ & Coupling constant multiplying the nonlinearity & \pageref{eq:force}\\
			
			$F_\kappa$ & Regularized force functional & \pageref{eq:force}\\
			
			$c_\kappa^{(i)}$ & Mass renormalization constants & \pageref{eq:force}\\
			
			$i_\sharp$ & Largest order of mass counterterms in the original force & \pageref{eq:force}\\
			
			$\lambda_\star$ & Maximal admissible size of $\lambda$ & \pageref{thm:main}\\
			
			$\bM$ & Euclidean spatial domain, identified with $\bR^d$ & \pageref{dfn:basic}\\
			
			$\bT$ & Periodic spatial domain $\bM/(2\pi\bZ)^d$ & \pageref{dfn:basic}\\
			
			$\varPhi_{\kappa}$ & Solution of the regularized equation & \pageref{eq:spde_F_mild}\\
			
			$F_{\kappa,\mu}$ & Effective force at scale $\mu$ & \pageref{eq:effective_force_intro}\\
			
			$\rD F$ & Fr\'echet derivative of a functional $F$ & \pageref{dfn:functional}\\
			
			$\varPhi_{\kappa,\mu}$ & Coarse-grained process at scale $\mu$ & \pageref{eq:coarse_grained}\\
			
			$\zeta_{\kappa,\mu}$ & Remainder in the effective equation & \pageref{eq:def_zeta}\\
			
			$H_{\kappa,\mu}$ & Defect of the effective force flow equation & \pageref{eq:def_H}\\
			
			$[\mu]=\mu^{1/\sigma}$ & Spatial scale associated with the scale parameter $\mu$ & \pageref{rem:scale}\\
			
			$\cK$ & Space of signed measures on $\bM$ with finite total variation & \pageref{def:K_space}\\
			
			$\delta_x$ & Dirac delta measure at $x$ & \pageref{def:K_space}\\
			
			$\tilde\fP_\mu$ & Operator $(1-[\mu]^2\Delta)^{d+2}$ & \pageref{def:K}\\
			
			$\fP_\mu$ & Operator $\tilde\fP_\mu^3$& \pageref{def:K}\\
			
			$\tilde K_\mu$ & Regularizing kernel solving $\tilde\fP_\mu\tilde K_\mu=\delta_0$ & \pageref{def:K}\\
			
			$K_\mu$ & Regularizing kernel $\tilde K_\mu\ast\tilde K_\mu\ast\tilde K_\mu$ & \pageref{def:K}\\
			
			$G$ & Green function of $(1-\Delta)^{\sigma/2}$ & \pageref{rem:G}\\
			
			$G_\mu$ & Scale-decomposed Green kernel & \pageref{dfn:kernel_G}\\
			
			$\dot G_\mu$ & Scale derivative $\partial_\mu G_\mu$ & \pageref{dfn:kernel_G}\\
			
			$\chi,\chi_\mu$ & Cutoff functions in scale decomposition of $G$ & \pageref{dfn:kernel_G}\\
			
			$\tilde F_{\kappa,\mu}$ & Lifted effective force acting on regularized fields & \pageref{eq:tilde_F_H}\\
			
			$\tilde H_{\kappa,\mu}$ & Lifted defect functional associated with $H_{\kappa,\mu}$ & \pageref{eq:tilde_F_H}\\
			
			$\tilde G_\mu$ & Kernel $\fP_\mu^2 \dot G_\mu$ entering the fixed point map & \pageref{eq:tilde_G}\\
			
			$K_{\eta,\mu}$ & Kernel satisfying $K_\mu=K_{\mu,\eta}\ast K_\eta$ & \pageref{eq:inter_ref_kernel}\\
			
			$\tilde\varPhi_{\kappa,\mu}$ & Coarse-grained type field $\fP_\mu\ast\varPhi_{\kappa,\mu}$ & \pageref{eq:tilde_Phi_zeta}\\
			
			$\tilde\zeta_{\kappa,\mu}$ & Regularized remainder $K_\mu\ast\zeta_{\kappa,\mu}$ & \pageref{eq:tilde_Phi_zeta}\\
			
			$m_\flat$ & Maximal polynomial degree retained in the estimates & \pageref{eq:bound_F1}\\
			
			$\fQ$ & Fixed point map for the regularized effective equation & \pageref{eq:map_Q}\\
			
			$F^{i,m}_{\kappa,\mu}$ & Kernels of the effective force $F_{\kappa,\mu}$ & \pageref{eq:intro_ansatz}\\
			
			$H^{i,m}_{\kappa,\mu}$ & Kernels in the polynomial expansion of $H_{\kappa,\mu}$ & \pageref{eq:intro_ansatz_H}\\
			
			$F^{i,m}_{\kappa}$ & Kernels of the initial force $F_\kappa$ & \pageref{rem:force_coefficients}\\
			
			$\cV^m$ & Banach space of kernels with $m+1$ field variables & \pageref{dfn:sVm}\\
			
			$\|\cdot\|_{\cV^m}$ & Total-variation type norm on $\cV^m$ & \pageref{dfn:sVm}\\
			
			$\fB$ on $\cV$ & Trilinear map & \pageref{dfn:map_B}\\
			
			$\cP_m$ & Permutation group of $\{1,\ldots,m\}$ & \pageref{dfn:map_B}\\
			
			$i_\flat$ & Maximal order retained in the effective force expansion & \pageref{def:rho}\\
			
			$\alpha_\varepsilon$ & Regularity of the solution $\sigma-d/2-\varepsilon$ & \pageref{def:rho}\\
			
			$\gamma_\varepsilon$ & Gain exponent $3\sigma-d-3\varepsilon$ due to subcriticality & \pageref{def:rho}\\
			
			$\varrho_\varepsilon(i,m)$ & Power-counting exponent of the kernel $F^{i,m}$ & \pageref{def:rho}\\
			
			$\varrho(i,m)$ & Unperturbed exponent $\varrho_0(i,m)$ & \pageref{def:rho}\\
			
			$\varepsilon_\diamond$ & Small upper bound on the admissible values of $\varepsilon$ & \pageref{rem:rho}\\
			
			$\bE(\zeta_1,\ldots,\zeta_p)$ & Joint cumulant of random variables & \pageref{def:cumulants}\\
			
			$\bE(\zeta_q)_{q\in I}$ & Compact notation for joint cumulant & \pageref{dfn:notation_cumulants_distributions}\\
			
			$\vI$ & List of indices for cumulants of effective force kernels & \pageref{dfn:cumulants_eff_force}\\
			
			$(i,m,s,r)$ & Index recording order, number of field variables, scale derivative, and regularization derivative & \pageref{dfn:cumulants_eff_force}\\
			
			$n(\vI)$ & Length of a list of indices & \pageref{dfn:cumulants_eff_force}\\
			
			$i(\vI),m(\vI),s(\vI),r(\vI)$ & Aggregate parameters associated with a list of indices & \pageref{dfn:cumulants_eff_force}\\
			
			$\MM(\vI)$ & list of field multiplicities attached to a list of indices & \pageref{dfn:cumulants_eff_force}\\
			
			$E^\vI_{\kappa,\mu}$ & Cumulant of the corresponding effective force kernels & \pageref{dfn:cumulants_eff_force}\\
			
			$\cE^\MM$ & Space of translation-invariant kernels & \pageref{dfn:sVMM}\\
			
			$\|\cdot\|_{\cE^\MM}$ & Norm on the cumulant-kernel space $\cE^\MM$ & \pageref{dfn:sVMM}\\
			
			$\varrho_\varepsilon(\vI)$ & Sum of power-counting exponents over a list of indices & \pageref{dfn:varrho_I}\\
			
			$\fA$ on $\cE$ & Bilinear map appearing in the cumulant flow equation & \pageref{dfn:maps_A_B}\\
			
			$\fB$ on $\cE$ & Trilinear map appearing in the cumulant flow equation & \pageref{dfn:maps_A_B}\\
			
			$\mathsf m,\tilde{\mathsf m},\hat{\mathsf m},\check{\mathsf m}$ & Lists of field multiplicities in the definitions of $\fA$ and $\fB$ & \pageref{dfn:maps_A_B}\\
			
			$\fT K$ & Periodization of an integrable kernel $K$ from $\bM$ to $\bT$ & \pageref{dfn:periodization}\\
			
			$\delta^{[m]}$ & Dirac delta distribution on the diagonal in $\bM^{1+m}$ & \pageref{thm:cumulants}\\
			
			$\cX^a$ & Monomial kernel $(x-y)^a$ used in Taylor remainders & \pageref{thm:cumulants}\\
			
			$\fI V$ & Integral of a translation-invariant two-point kernel & \pageref{thm:cumulants}\\
			
			$\fR^a V$ & Remainder associated with Taylor expansion of $V$ & \pageref{thm:cumulants}\\
			
			$\hat E^i_{\kappa,\mu}$ & Scalar local part in the decomposition of cumulants & \pageref{eq:E_decomposition}\\
			
			$\check E^i_{\kappa,\mu},\tilde E^i_{\kappa,\mu}$ & Non-local components in the cumulant decomposition & \pageref{eq:E_decomposition}\\
			
		\end{longtable}
	\end{center}
}

\begin{acknowledgement}
I would like to thank the referees for their careful reading of these lecture notes. Their numerous insightful comments and suggestions greatly helped improve the clarity, accuracy, and overall quality of the exposition.
\end{acknowledgement}


\begin{thebibliography}{99.}%

\bibitem{BCD11}
H. Bahouri, J.-Y. Chemin, R. Danchin, \emph{Fourier Analysis and Nonlinear Partial Differential Equations}, Springer (2011)

\bibitem{BH23}
I. Bailleul and M. Hoshino,
\emph{Random models on regularity-integrability structures},
[arXiv:2310.10202]

\bibitem{bauerschmidt2021}
R. Bauerschmidt, T. Bodineau,
\emph{Log-Sobolev Inequality for the Continuum Sine-Gordon},
Model. Comm. Pure Appl. Math. {\bf 74}, 2064--2113 (2021)
[arXiv:1907.12308]

\bibitem{BBD23}
R. Bauerschmidt, T. Bodineau, B. Dagallier,
\emph{Stochastic dynamics and the Polchinski equation: an introduction},
Probab. Surveys {\bf 21}, 200--290 (2024)
[arXiv:2307.07619]

\bibitem{BOS25}
L. Broux, F. Otto, and R. Steele,
\emph{Multi-index based solution theory to the $\Phi^4$ equation in the full subcritical regime},
[arXiv:2503.01621]

\bibitem{BOT25}
L. Broux, F. Otto, and M. Tempelmayr,
\emph{Lecture notes on Malliavin calculus in regularity structures},
Stochastics and Partial Differential Equations: Analysis and Computations, 1--78 (2025)
[arXiv:2401.05935]

\bibitem{bruned2021renormalising}
Y. Bruned, A. Chandra, I. Chevyrev, M. Hairer, 
\emph{Renormalizing SPDEs in regularity structures}, 
J. Eur. Math. Soc. {\bf 23}(3), 869--947 (2021)
[arXiv:1711.10239]

\bibitem{bruned2019algebraic}
Y. Bruned, M. Hairer, L. Zambotti, 
\emph{Algebraic renormalization of regularity structures},
Invent. Math. {\bf 215}(3), 1039--1156 (2019)
[arXiv:1610.08468]

\bibitem{BM25}
Y. Bruned, A. Minguella,
\emph{Renormalization in the flow approach for singular SPDEs},
[arXiv:2504.04885]

\bibitem{brydges1987Mayer}
D. Brydges and T. Kennedy, 
\emph{Mayer expansions and the Hamilton--Jacobi equation},
J. Stat. Phys. {\bf 48}(1), 19--49 (1987).


\bibitem{CF24a}
A. Chandra, L. Ferdinand,
\emph{A flow approach to the generalized KPZ equation},
[arXiv:2402.03101]

\bibitem{CF24b}
A. Chandra, L. Ferdinand,
\emph{Rough differential equations in the flow approach},
[arXiv:2411.07157]

\bibitem{chandra2016bphz}
A. Chandra, M. Hairer, 
\emph{An analytic BPHZ theorem for regularity structures},
[arXiv:1612.08138]

\bibitem{daprato2003}
G. Da~Prato, A. Debussche, 
\emph{Strong solutions to the stochastic quantization equations},
Ann. Probab. {\bf 31}(4), 1900--1916 (2003)

\bibitem{DFG25}
F. De Vecchi, L. Fresta, M. Gubinelli,
\emph{A stochastic analysis of subcritical Euclidean fermionic field theories},
The Annals of Probability {\bf 53}(3), 906--966 (2025)
[arXiv:2210.15047]

\bibitem{Du24a}
P. Duch, Construction of Gross--Neveu model using Polchinski flow equation, 
[arXiv:2403.18562]

\bibitem{Du24b}
P. Duch,
\emph{Construction of fractional $\Phi^4_3$ model of Euclidean QFT using flow equation approach to singular SPDEs},
\url{https://pawelduch.github.io/fractional_phi43.pdf}



\bibitem{Du21}
P. Duch,
\emph{Flow equation approach to singular stochastic PDEs},
{Probab. Math. Phys.} \textbf{6}, 327--437 (2025)
[arXiv:2109.11380]

\bibitem{Du22}
P. Duch,
\emph{Renormalization of singular elliptic stochastic PDEs using flow equation},
{Probab. Math. Phys.} \textbf{6}, 111--138 (2025), 
[arXiv:2201.05031]

\bibitem{GR23}
P. Duch, M. Gubinelli, P. Rinaldi, \emph{Parabolic stochastic quantization of the fractional $\phi^4_3$ model in the full subcritical regime},
[arXiv:2303.18112]


\bibitem{EW24}
S. Esquivel, H. Weber,
\emph{A priori bounds for the dynamic fractional $\Phi^4$ model on $\mathbb{T}^3$ in the full subcritical regime} [arXiv:2411.16536]

\bibitem{gubinelli2015}
M. Gubinelli, P. Imkeller, N. Perkowski, 
\emph{Paracontrolled distributions and singular PDEs}, 
Forum Math. Pi {\bf 3}, e6 (2015)
[arXiv:1210.2684]

\bibitem{GM24}
M. Gubinelli, S.-J. Meyer, \emph{The FBSDE approach to sine-Gordon up to $6\pi$},
[arXiv:2401.13648]


\bibitem{hairer2014structures} 
M. Hairer, 
\emph{A theory of regularity structures}, 
Invent. Math. {\bf 198}(2), 269--504 (2014)
[arXiv:1303.5113]

\bibitem{Hai15}
M. Hairer, \emph{Introduction to regularity structures}, Brazilian Journal of Probability and Statistics {\bf 29}(2), 175--210 (2015)
[arXiv:1401.3014]


\bibitem{HS23}
M. Hairer, R. Steele, \emph{The BPHZ theorem for regularity structures via the spectral gap inequality}, 
Arch. Rational Mech. Anal. {\bf 248}, 9 (2024) 
[arXiv:2301.10081]




\bibitem{kupiainen2016rg}
A. Kupiainen, \emph{Renormalization group and stochastic PDEs}, 
Ann. Henri Poincaré {\bf 17}(3), 497--535 (2016)
[arXiv:1410.3094]

\bibitem{kupiainen2017kpz}
A. Kupiainen, M. Marcozzi, 
\emph{Renormalization of generalized KPZ equation},
J.~Stat. Phys. {\bf 166}, 876--902 (2017) 
[arXiv:1604.08712]

\bibitem{LOTT21}
P. Linares, F. Otto, M. Tempelmayr, P. Tsatsoulis, \emph{A diagram-free approach to the stochastic estimates in regularity structures},
Invent. Math. {\bf 237}, 1469--1565 (2024) 
[arXiv:2112.10739]

\bibitem{MWX16}
J.-C. Mourrat, H. Weber, W. Xu, \emph{Construction of $\Phi^4_3$ diagrams for pedestrians}, From particle systems to partial differential equations, Springer Proc. Math. Stat. {\bf 209}, 1--46 (2016)
[arXiv:1610.08897]


\bibitem{muller2003}
V. Müller, 
\emph{Perturbative renormalization by flow equations},
Rev. Math. Phys. {\bf 15}(05), 491--558 (2003)
[arXiv:hep-th/0208211]



\bibitem{OSSW21}
F. Otto, J. Sauer, S. Smith, H. Weber, \emph{A priori estimates for quasi-linear
SPDEs in the full sub-critical regime},
J. Eur. Math. Soc. {\bf 27}(1) 71--118 (2025)
[arXiv:2103.11039]

\bibitem{parisi1981}
G. Parisi, Y.S. Wu, 
\emph{Perturbation theory without gauge fixing},
Sci. Sin. {\bf 24}(4), 483--496 (1981)

\bibitem{peccati2011wiener}
G. Peccati, M. Taqqu,
\emph{Wiener Chaos: Moments, Cumulants and Diagrams: A~survey with computer implementation},
Springer (2011)

\bibitem{polchinski1984}
J. Polchinski, 
\emph{Renormalization and effective lagrangians}, 
Nuclear Physics B, {\bf 231}(2), 269--295 (1984)

\bibitem{wilson1971}
K. Wilson,
\emph{Renormalization Group and Critical Phenomena. I. Renormalization Group and the Kadanoff Scaling Picture},
Phys. Rev. B {\bf 4}, 3174--3183 (1971)









\end{thebibliography}
\end{document}